\documentclass[english]{book}
\usepackage{babel}
\usepackage{amsmath}
\usepackage{amsfonts}
\usepackage{amssymb}
\usepackage{latexsym}
\usepackage{graphicx}
\usepackage{makeidx}
\usepackage{float}
\usepackage{mathrsfs,skak}

\newtheorem{Thm}{Theorem}[section]
\newtheorem{Thm*}{Theorem*}[section]

\newtheorem{Prop}[Thm]{Proposition}

\newtheorem{Lem}[Thm]{Lemma}
\newtheorem{Def}[Thm]{Definition}

\newenvironment{proof}[1][Proof]{\noindent\textbf{#1.} }{\ \rule{0.5em}{0.5em}}

\def\bC {\mathbf{C}}

\def\bE {\mathbf{E}}

\def\bF {\mathbf{F}}

\def\bN {\mathbf{N}}
\def\bP {\mathbf{P}}

\def\bR {\mathbf{R}}

\def\bZ {\mathbf{Z}}
\def\bz {\mathbf{z}}

\def\scrL{\mathscr{L}}

\def\fH {\mathfrak{H}}
\def\fS {\mathfrak{S}}
\def\fZ {\mathfrak{Z}}

\def\cA {\mathcal{A}}
\def\cB {\mathcal{B}}

\def\cD {\mathcal{D}}
\def\cE {\mathcal{E}}
\def\cF {\mathcal{F}}
\def\cH {\mathcal{H}}

\def\cK {\mathcal{K}}
\def\cL {\mathcal{L}}

\def\cN {\mathcal{N}}

\def\cP {\mathcal{P}}

\def\cS {\mathcal{S}}

\def\cX {\mathcal{X}}

\def\a {{\alpha}}

\def\de {{\delta}}
\def\g {{\gamma}}

\def\eps {{\epsilon}}

\def\ka {{\kappa}}
\def\l {{\lambda}}

\def\si {{\sigma}}

\def\om {{\omega}}
\def\Om {{\Omega}}

\def\rstr {{\big |}}
\def\indc {{\bf 1}}
\def\wto {{\rightharpoonup}}

\def\la {\langle}
\def\ra {\rangle}
\def \La {\bigg\langle}
\def \Ra {\bigg\rangle}

\def\d {{\partial}}
\def\grad {{\nabla}}
\def\Dlt {{\Delta}}

\def\bu {{\noindent$\bullet$ }}

\newcommand{\Div}{\operatorname{div}}
\newcommand{\Rot}{\operatorname{rot}}
\newcommand{\Dist}{\operatorname{dist}}
\newcommand{\DDist}{\operatorname{Dist}}

\newcommand{\Sign}{\operatorname{sign}}
\newcommand{\Supp}{\operatorname{supp}}
\newcommand{\Span}{\operatorname{span}}
\newcommand{\Det}{\operatorname{\hbox{det}}}
\newcommand{\Tr}{\operatorname{trace}}
\newcommand{\Ker}{\operatorname{Ker}}

\newcommand{\Lip}{\operatorname{Lip}}

\newcommand{\be}{\begin{equation}}
\newcommand{\ee}{\end{equation}}
\newcommand{\ba}{\begin{aligned}}
\newcommand{\ea}{\end{aligned}}
\newcommand{\lb}{\label}

\newcommand{\IM}{\operatorname{Im}}

\makeindex

\begin{document}

\frontmatter

\title{\textbf{On the Dynamics of Large Particle Systems in the Mean Field Limit}}

\author{\sc{Fran\c cois Golse}}

\date{{Ecole Polytechnique\\ Centre de Math\'ematiques Laurent Schwartz\\ 91128 Palaiseau Cedex, France }
\\
\vskip 8cm
\rightline{\textit{In memory of Seiji Ukai (1940-2012)}}
}

\maketitle


\tableofcontents


\mainmatter

\chapter[Mean Field Limit]{On the Dynamics of Large Particle Systems in the Mean Field Limit}

The general purpose of all the mean field equations considered below is to describe the dynamics of a very large number of identical particles, assuming 
that the interaction between particles is known exactly. For instance, at a temperature of $273$ K and a pressure of $1.01\cdot 10^5$ Pa, the number of 
molecules of any ideal gas to be found in a volume of $2.24\cdot 10^{-2}\,\mathrm{m}^3$ is the Avogadro number, i.e. $6.02\cdot 10^{23}$. This is typically 
what we have in mind when we think of a number of particles so large that the dynamics of each individual particle cannot be exactly determined. Thus, 
in order to be of any practical interest, these mean field models should be set on (relatively) low dimensional spaces. In any case, they should involve fewer
degrees of freedom (of the order of a few units) than the total number of degrees of freedom of the whole particle system. 

To be more precise: in classical mechanics, the number of degrees of freedom of a single point particle moving without constraint in the $d$-dimensional
Euclidean space $\bR^d$ is $d$. The \textit{single-particle phase space}\index{Single-particle phase space} is the set $\bR^d\times\bR^d$ of pairs of all 
possible positions and momenta of an unconstrained single point particle in $\bR^d$. 

For a system of $N$ identical point particles moving without constraint in the $d$-dimensional Euclidean space $\bR^d$, the number of degrees of freedom 
is therefore $dN$. The space $(\bR^d\times\bR^d)^N$ of $2N$-tuples of all possible positions and momenta of the $N$ point particles is the 
\textit{$N$-particle phase space}\index{$N$-particle phase space}. 

Thus, the laws of motion of classical mechanics (Newton's second law) written for each molecule of a monatomic gas enclosed in a container form a system 
of differential equations set on the $N$-particle phase space, where $N$ is the total number of gas molecules to be found in the container. With $N$ of the
order of the Avogadro number, this is of little practical interest. In the kinetic theory of gases, the evolution of the same gas is described by the Boltzmann
equation, an integro-differential equation set on the single-particle phase space. Although the kinetic theory of gases will not be discussed in these notes,
all the mean field limits considered below will involve the same reduction of the $N$-particle phase space to the single-particle phase space.

\smallskip
Here is a (by no means exhaustive) list of typical examples of mean field equations:

\smallskip
\noindent
(a) the particles are the ions and the electrons in a plasma; the interaction is the Coulomb electrostatic force;  in that case, the mean field equation is the
Vlasov-Poisson system of the kinetic theory of plasmas;

\smallskip
\noindent
(b) the particles are the nuclei and the electrons in a molecule; the interaction is again the Coulomb electrostatic force; the corresponding mean field model
in this context is the Hartree equation or the system of Hartree-Fock equations in atomic physics;

\smallskip
\noindent
(c) the particles are vortices in an incompressible fluid in space dimension $2$; the interaction is given by the Helmholtz potential; the corresponding mean
field model is the vorticity formulation of the Euler equations of incompressible fluid mechanics 

\smallskip
All these models are obtained as approximations of the system of equations of motion for each particle in the limit when the number of particles involved 
tends to infinity. Rigorous justifications of these approximations are based on various mathematical formalisms that are introduced and discussed below.

Excellent references on the issues discussed above are \cite{Spohn80,Spohn91,CIP}

\section[Examples in classical mechanics]{Examples of mean field models in classical mechanics}

In this section, we consider the examples mentioned above in which the motion of each particle is described in the formalism of classical mechanics, i.e.
examples (a) and (c). All these examples have a common structure, to be discussed later, which involves the Liouville equation recalled below.

\subsection{The Liouville equation}

The Liouville equation\index{Liouville equation} governs the evolution of the distribution function for a system of particles subject to an external force field. 

The notion of distribution function is fundamental in the kinetic theories of gases and plasmas, and more generally in statistical mechanics. It was introduced 
by Maxwell in one of his most famous articles\footnote{\textit{Illustrations of the Dynamical Theory of Gases}, Philosophical Magazine (1860); reprinted in 
``The Scientific Papers of James Clerk Maxwell'', edited by W.D. Niven, Cambridge University Press, 1890; pp. 377--409.}. 

The distribution function\index{Distribution function} of a system of identical point particles is $f\equiv f(t,x,v)$, that is the number density of particles that are 
located at the position $x$ and have instantaneous velocity $v$ at time $t$. In other words, the number of particles to be found at time $t$ in an infinitesimal 
volume $dxdv$ of the single-particle phase space centered at $(x,v)$ is $\simeq f(t,x,v)dxdv$. 

Assume that a particle located at the position $x$ with instantaneous velocity $v$ at time $t$ is subject to some external (or imposed) force field $F(t,x,v)$. 
As a consequence of Newton's second law of classical mechanics, the distribution function $f$ of the system of particles under consideration satisfies the
Liouville equation
$$
\d_tf+v\cdot\grad_xf+\tfrac1m\Div_v(F(t,x,v)f)=0\,,
$$
where $m>0$ is the particle mass. The Liouville equation is a partial differential equation (PDE) of order $1$, whose solution can be expressed by the 
method of characteristics.

To the PDE of order $1$ above, one associates the system of ordinary differential equations (ODE)
$$
\left\{
\ba
{}&\dot{X}=V\,,
\\
&\dot{V}=\tfrac1mF(t,X,V)\,,
\ea
\right.
$$
with the usual notation
$$
\dot{\phi}(t)=\frac{d\phi}{dt}(t)\,.
$$
These ODEs are referred to as the ``equations of characteristics''\index{Method of characteristics} for the Liouville equation. Denote by 
$t\mapsto(X(t,t_0,x,v),V(t,t_0,x,v))$ the solution of this ODE system such that
$$
X(t_0,t_0,x,v)=x\,,\quad V(t_0,t_0,x,v)=v\,;
$$
the map $(t,t_0,x,v)\mapsto(X(t,t_0,x,v),V(t,t_0,x,v))$ will be henceforth referred to as ``the characteristic flow''\index{Characteristic flow} associated to the 
Liouville equation above.

One immediately recognizes in the system of ODEs above the equations of motion of classical mechanics for a particle of mass $m$ subject to the force
field $F\equiv F(t,x,v)$ (the first equality being the definition of the velocity, while the second is Newton's second law).

Assume that the force field $F$ is such that the characteristic flow $(X,V)$ is globally defined (i.e. defined on $\bR\times\bR\times\bR^d\times\bR^d$). This
characteristic flow is used as follows to express the solution of the Cauchy problem
$$
\left\{
\ba
{}&\d_tf+v\cdot\grad_xf+\tfrac1m\Div_v(F(t,x,v)f)=0\,,\quad x,v\in\bR^d\times\bR^d\,,\,\,t\in\bR\,,
\\
&f\rstr_{t=0}=f^{in}\,.
\ea
\right.
$$
For each test function\footnote{For each topological space $X$ and each finite dimensional vector space $E$ on $\bR$, we designate by $C(X,E)$ the set 
of continuous functions defined on $X$ with values in $E$, and by $C_c(X,E)$ the set of functions belonging to $C(X,E)$ whose support is compact in $X$. 
For each $n,k\ge 1$, we denote by $C^k_c(\bR^n,E)$ the set of functions of class $C^k$ defined on $\bR^n$ with values in $E$ whose support is compact
in $\bR^n$. We also denote $C(X):=C(X,\bR)$, $C_c(X):=C_c(X,\bR)$ and $C^k_c(\bR^n):=C^k_c(\bR^n,\bR)$.} $\phi\in C^1_c(\bR^d\times\bR^d)$, and for 
each $t\in\bR$, one has
$$
\iint_{\bR^d\times\bR^d}f(t,x,v)\phi(x,v)dxdv=\iint_{\bR^d\times\bR^d}f^{in}(x,v)\phi((X,V)(t,0,x,v))dxdv\,,
$$
and this completely determines the distribution function $f(t,x,v)$.

\noindent
\textbf{Exercise:} Deduce from the equality above an explicit formula giving $f(t,x,v)$ in terms of $f^{in}$ and of the characteristic flow $(X,V)$.

\smallskip
For a concise discussion of the method of characteristics (and a solution of the exercise above), see chapter 1, section 1, in \cite{BouGolPul}.

\subsection{The Vlasov-Poisson system}\lb{SS-VP}

Our first example of a mean field kinetic model is the Vlasov-Poisson system\index{Vlasov-Poisson system} used in plasma physics. 

Consider a system of identical point particles with mass $m$ and charge $q$. The electrostatic (Coulomb) force exerted on any such particle located at
the position $x\in\bR^3$ by another particle located at the position $y\not=x$ is
$$
\frac{q^2}{4\pi\eps_0}\frac{x-y}{|x-y|^3}\,,
$$
where $\eps_0$ is the dielectric permittivity of vacuum. This force is indeed repulsive, as it is of the form $\l(x-y)$ with $\l>0$.

More generally, the electrostatic force exerted on a particle located at the position $x$ by a cloud of particles with number density $\rho(t,y)$ (which means 
that, at time $t$, approximately $\rho(t,y)dy$ particles are to be found in any infinitesimal volume element $dy$ centered at $y$) is
$$
F(t,x):=\frac{q^2}{4\pi\eps_0}\int_{\bR^3}\frac{x-y}{|x-y|^3}\rho(t,y)dy\,.
$$

Recall that
$$
G(x)=\frac1{4\pi|x|}\,,\quad x\in\bR^3\setminus\{0\}
$$
is the solution of 
$$
\left\{
\ba
{}&-\Dlt G=\de_0\quad&&\hbox{ in }\cD'(\bR^3)\,,
\\
&\,\,G(x)\to 0&&\hbox{ as }|x|\to\infty\,,
\ea
\right.
$$
where $\de_0$ designates the Dirac measure at the origin.

Thus the electrostatic force $F$ is given by
$$
F(t,x)=qE(t,x)\,,
$$
where $E\equiv E(t,x)$ is the electric field, i.e.
$$
E(t,x)=-\grad_x\phi(t,x)
$$
with electrostatic potential $\phi$ given by
$$
\phi(t,\cdot)=\frac{q}{\eps_0}G\star\rho(t,\cdot)\,.
$$
In particular
$$
\ba
-\Dlt\phi(t,\cdot)&=-\frac{q}{\eps_0}\Dlt(G\star\rho(t,\cdot))=-\frac{q}{\eps_0}(\Dlt G)\star\rho(t,\cdot)
\\
&=\frac{q}{\eps_0}\de_0\star\rho(t,\cdot)=\frac{q}{\eps_0}\rho(t,\cdot)\,.
\ea
$$

The Vlasov-Poisson system describes the motion of a system of identical charged point particles accelerated by the electrostatic force created by their
own distribution of charges --- referred to as the ``self-consistent electrostatic force''. It consists of the Liouville equation governing the evolution of the
distribution function, coupled to the Poisson equation satisfied by the self-consistent electrostatic potential as follows: 
$$
\left\{
\ba
{}&(\d_t+v\cdot\grad_x)f(t,x,v)-\tfrac{q}{m}\grad_x\phi(t,x)\cdot\grad_vf(t,x,v)=0\,,
\\	\\
&-\Dlt_x\phi(t,x)=\tfrac1{\eps_0}q\rho_f(t,x)\,,
\\	\\
&\rho_f(t,x)=\int_{\bR^3}f(t,x,v)dv\,.
\ea
\right.
$$

\smallskip
\noindent
\textbf{Exercise:} Let $f\equiv f(t,x,v)$ and $\phi\equiv\phi(t,x)$ be a solution of the Vlasov-Poisson system. To avoid technicalities, we assume that $f$
and $\phi$ belong to $C^\infty(\bR_+\times\bR^3\times\bR^3)$ and $C^\infty(\bR_+\times\bR^3)$ respectively, and that, for each $t\ge 0$, the functions
$(x,v)\mapsto f(t,x,v)$ and $x\mapsto\phi(t,x)$ belong to the Schwartz class\footnote{For each $n\ge 1$, the Schwartz class $\cS(\bR^n)$ is the set of
real-valued $C^\infty$ functions defined on $\bR^n$ all of whose partial derivatives are rapidly decreasing at infinity:
$$
\cS(\bR^n):=\{f\in C^\infty(\bR^n)\hbox{ s.t. }|x|^m\d^\a f(x)\to 0\hbox{ as }|x|\to\infty\hbox{Ê for all }m\ge 1\hbox{  and }\a\in\bN^n\}\,.
$$} 
$\cS(\bR^3\times\bR^3)$ and $\cS(\bR^3)$ respectively.

\noindent
1) Prove that
$$
\ba
\d_t\rho_f(t,x)+\Div_x\int_{\bR^3}vf(t,x,v)dv=0&\,,
\\
\d_t\int_{\bR^3}mvf(t,x,v)dv+\Div_x\int_{\bR^3}mv^{\otimes 2}f(t,x,v)dv&
\\
+q\rho_f(t,x)\grad_x\phi(t,x)=0&\,,
\\
\d_t\int_{\bR^3}\tfrac12m|v|^2f(t,x,v)dv+\Div_x\int_{\bR^3}v\tfrac12m|v|^2f(t,x,v)dv&
\\
+\grad_x\phi(t,x)\cdot\int_{\bR^3}qvf(t,x,v)dv=0&\,.
\ea
$$
These three equalities are respectively the local conservation laws of mass, momentum and energy\index{Local conservation laws of mass, momentum 
and energy}.

\noindent
2) Prove the global conservation of mass\index{Global conservation of mass, momentum and energy} (or equivalently of the total number of particles) 
$$
\frac{d}{dt}\iint_{\bR^3\times\bR^3}f(t,x,v)dxdv=0\,.
$$
3) Prove that, for each $\psi\in C^2(\bR^3)$, one has
$$
\Dlt\psi\grad\psi=\Div((\grad\psi)^{\otimes 2})-\tfrac12\grad|\grad\psi|^2
$$
and conclude that
$$
\frac{d}{dt}\iint_{\bR^3\times\bR^3}mvf(t,x,v)dxdv=0\,.
$$
(This is the global conservation of momentum).

\noindent
4) Prove the global conservation of energy:
$$
\frac{d}{dt}\left(\iint_{\bR^3\times\bR^3}\tfrac12m|v|^2f(t,x,v)dxdv+\tfrac12\eps_0\int_{\bR^3}|\grad_x\phi(t,x)|^2dx\right)=0\,.
$$
(Hint: the term
$$
\int_{\bR^3}\phi(t,x)\Div_x\left(\int_{\bR^3}qvf(t,x,v)dv\right)dx
$$
can be simplified by using the result in question 1).)

The Vlasov-Poisson system above is written in the case of a single species of identical charged particles. In reality, a plasma is a globally neutral system, 
and therefore involves many different species of particles. Denoting by $f_k\equiv f_k(t,x,v)$ for $k=1,\ldots,M$ the distribution function of the $k$th species in 
a system of $M$ different kinds of particles, the Vlasov-Poisson describes the motion of this system of particles submitted to the self-consistent electrostatic 
force resulting from the distribution of charges of all the particles in that system:
$$
\left\{
\ba
{}&(\d_t+v\cdot\grad_x)f_k(t,x,v)-\tfrac{q_k}{m_k}\grad_x\phi(t,x)\cdot\grad_vf_k(t,x,v)=0\,,\quad k=1,\ldots,M\,,
\\	\\
&-\Dlt_x\phi(t,x)=\tfrac1{\eps_0}\sum_{k=1}^Mq_k\rho_k(t,x)\,,
\\	\\
&\rho_k(t,x)=\int_{\bR^3}f_k(t,x,v)dv\,.
\ea
\right.
$$
Here $m_k$ and $q_k$ designate respectively the mass and the charge of particles of the $k$th species. In practice, considering systems of particles of
different species in the Vlasov-Poisson system does not involve additional mathematical difficulties. Therefore, we shall  consider only the (unphysical)
case of a single species of particles in our mathematical discussion of this system.

\smallskip
There is a huge literature on the Vlasov-Poisson system, which is of considerable importance in plasma physics. The global existence and uniqueness 
of classical solutions of the Cauchy problem for the Vlasov-Poisson system was obtained by Ukai-Okabe \cite{UkaiOka} in the $2$-dimensional case, and
independently by Lions-Perthame \cite{LionsPerth} and Pfaffelmoser \cite{Pfaffel} in the $3$-dimensional case. For a presentation of the mathematical 
analysis of this system, the interested reader is referred to \cite{Glassey,BouGolPul,ReinEVEQ}.

\subsection{The Euler equation for two-dimensional incompressible fluids}

The Euler equation for an incompressible fluid\index{Euler equation} with constant density (set to $1$ in the sequel without loss of generality) governs the 
evolution of the unknown velocity field $u\equiv u(t,x)\in\bR^2$\index{Velocity field} and of the unknown pressure field $p\equiv p(t,x)\in\bR$ in the fluid, 
assumed to fill the Euclidean plane $\bR^2$. It takes the form
$$
\d_tu(t,x)+(u\cdot\grad_x)u(t,x)+\grad_xp(t,x)=0\,,\quad\Div_xu(t,x)=0\,.
$$
The pressure field $p$ can be viewed as the Lagrange multiplier associated to the constraint $\Div_xu=0$. 

There is another formulation of the Euler equation in which the pressure field is eliminated. To the velocity field $u\equiv u(t,x)\in\bR^2$, one associates
its scalar vorticity field $\om\equiv\om(t,x)$\index{Vorticity field} defined as follows
$$
\om(t,x):=\d_{x_1}u_2(t,x)-\d_{x_2}u_1(t,x)\,.
$$
The vorticity field satisfies
$$
\d_t\om+\Div_x(u\om)=0\,.
$$
One can check that the Euler equation is (formally) equivalent to the system\index{Vorticity formulation Euler's equation}
$$
\left\{
\ba
{}&\d_t\om+\Div_x(u\om)=0\,,
\\
&\Div_xu=0\,,
\\
&\om=\Div_x(Ju)\,,
\ea
\right.
$$
where $J$ designates the rotation of an angle $-\tfrac{\pi}2$:
$$
J=\left(\begin{matrix} 0\,\,\,&1\\-1&0\end{matrix}\right)
$$

There is yet another formulation of this system, where the velocity field $u$ is represented in terms of a stream function. Indeed, since $\Div_xu=0$, there
exists a stream function $\phi\equiv\phi(t,x)$\index{Stream function} such that
$$
u(t,x)=J\grad_x\phi(t,x)\,.
$$
The vorticity field $\om$ is related to the stream function $\phi$ as follows:
$$
\om=\Div_x(Ju)=\Div_x(J^2\grad_x\phi)=-\Dlt_x\phi\,,
$$
so that the vorticity formulation of the Euler equation becomes
$$
\left\{
\ba
{}&\d_t\om+\Div_x(\om J\grad_x\phi)=0\,,
\\
&-\Dlt_x\phi=\om\,.
\ea
\right.
$$
In this last system, the unknown is the pair $(\om,\phi)$; once the stream function $\phi$ is known, the velocity field $u$ is obtained as the spatial gradient
of $\phi$ rotated of an angle $-\tfrac{\pi}2$.

There is an obvious analogy between this formulation of the Euler equation and the Vlasov-Poisson system: the vorticity field $\om$ is the analogue of the
distribution function $f$ in the Vlasov-Poisson system, the stream function is the analogue of the electrostatic potential in the Vlasov-Poisson system, and
the velocity field $u$ is the analogue of the electric field in the Vlasov-Poisson system. There is however a subtle difference between both systems: the
equations of characteristics associated to the vorticity formulation of the Euler equations are
$$
J\dot{X}=-\grad_x\phi(t,X)\,,
$$
while those corresponding with the Vlasov-Poisson system are
$$
\ddot{X}=-\frac{q}m\grad_x\phi(t,X)\,,
$$
(after eliminating $V=\dot{X}$). The first ODE system is of order $1$, while the second is of order $2$, because the gradient of the solution of the Poisson
equation is a velocity field in the vorticity formulation of the Euler equation, while it is an acceleration field in the case of the Vlasov-Poisson system.

\smallskip
An excellent reference on the Euler equations of incompressible fluids and on issues related to the mean field limit is \cite{MarchioPulvi}.

\subsection{The Vlasov-Maxwell system}

Observe that magnetic effects are not taken into account in the Vlasov-Poisson system. Yet charged particles in motion generate both an electric and a
magnetic field; on the other hand, the electric and magnetic fields both act on charged particles via the Lorentz force. Studying magnetized plasmas 
(such as tokamak plasmas, or the solar wind, for instance) requires using the Vlasov-Maxwell, instead of the Vlasov-Poisson system.

The unknown in the Vlasov-Maxwell system is the triple $(f,E,B)$, where $E\equiv E(t,x)\in\bR^3$ and $B\equiv B(t,x)\in\bR^3$ are respectively the electric
and the magnetic fields, while $f\equiv f(t,x,\xi)$ is the particle distribution function. Notice that the velocity variable $v$ is replaced with the momentum $\xi$
in the distribution function. In other words, $f(t,x,\xi)$ is the phase space density of particles located at the position $x$ at time $t$, with momentum $\xi$. The 
momentum and the velocity are related as follows. The relativistic energy of a particle of mass $m$ with momentum $\xi$ is 
$$
e(\xi):=\sqrt{m^2c^4+c^2|\xi|^2}\,,
$$
where $c$ is the speed of light (notice that $e(0)=mc^2$, the energy at rest of a particle with mass $m$). Then
$$
v(\xi)=\grad e(\xi)=\frac{c^2\xi}{\sqrt{m^2c^4+c^2|\xi|^2}}\,.
$$

The (relativistic) Vlasov-Maxwell system\index{Vlasov-Maxwell system} takes the form
$$
\left\{
\ba
{}&\d_tf+v(\xi)\cdot\grad_xf+q(E+v(\xi)\times B)\cdot\grad_\xi f=0\,,
\\	\\
&\Div_xB=0\,,\qquad\qquad\d_tB+\Rot_xE=0\,,
\\	\\
&\Div_xE=\tfrac1{\eps_0}q\rho_f\,,\qquad\tfrac1{c^2}\d_tE-\Rot_xB=-\mu_0qj_f\,,
\\	\\
&\rho_f=\int_{\bR^3}fd\xi\,,\qquad\quad\,\,j_f=\int_{\bR^3}v(\xi)fd\xi\,.
\ea
\right.
$$
The first equation in this system is the relativistic Liouville equation governing $f$; the term
$$
q(E(t,x)+v(\xi)\times B(t,x))
$$
is the Lorentz force\index{Lorentz force} field at time $t$ and position $x$ exerted by the electromagnetic field $(E,B)$ on a particle with charge $q$ and 
momentum $\xi$. The second and third equations are respectively the equation expressing the absence of magnetic monopoles and the Faraday equation, 
while the fourth and fifth equations are respectively the Gauss equation (as in electrostatics) and the Maxwell-Amp\`ere equation. Since the source terms
in the Maxwell system are the charge density $q\rho_f$ and the current density $qj_f$, the Lorentz force in the Liouville equation is the self-consistent 
Lorentz force, i.e. the electromagnetic force generated by the motion of the charged particles accelerated by this force itself.

For the same reason as in the case of the Vlasov-Poisson system, the case of a single species of charged particles is somewhat unrealistic. Physically 
relevant models in plasma physics involve different species of particles so as to maintain global neutrality of the particle system.

\smallskip
\noindent
\textbf{Exercise:} Following the discussion in the previous section, write the relativistic Vlasov-Maxwell system for a system of $M$ species of particles 
with masses $m_k$ and charges $q_k$, for $k=1,\ldots,M$. Write the local conservation laws of mass, momentum and energy for the resulting system,
following the analogous discussion above in the case of the Vlasov-Poisson system.

\smallskip
For more information on the Vlasov-Maxwell system, see \cite{Glassey,BouGolPul,ReinEVEQ}.

\section[General formalism in classical mechanics]{A general formalism for mean field limits in classical mechanics}

We first introduce a formalism for mean field limits in classical mechanics that encompasses all the examples discussed above.

Consider a system of $N$ particles, whose state at time $t$ is defined by phase space coordinates $\hat z_1(t),\ldots,\hat z_N(t)\in\bR^d$. For instance,
$z_j$ is the position $x_j$ of the $j$th vortex center in the case of the two dimensional Euler equations for incompressible fluids, and the phase space 
dimension is $d=2$. In the case of the Vlasov-Poisson system, the phase space is $\bR^3\times\bR^3\simeq\bR^6$, so that $d=6$, and $z_j=(x_j,v_j)$,
where $x_j$ and $v_j$ are respectively the position and the velocity of the $j$th particle.

The interaction between the $i$th and the $j$th particle is given by $K(\hat z_i,\hat z_j)$, where 
$$
K:\,\bR^d\times\bR^d\to\bR^d
$$
is a map whose properties will be discussed below.

The evolution of $\hat z_1(t),\ldots,\hat z_N(t)\in\bR^d$ is governed by the system of ODEs
$$
\frac{d\hat z_i}{dt}(t)=\sum_{j=1\atop j\not=i}^NK(\hat z_i(t),\hat z_j(t))\,,\quad i,j=1,\ldots,N\,.
$$

\smallskip
\noindent
\textbf{Problem:} to describe the behavior of $\hat z_1(t),\ldots,\hat z_N(t)\in\bR^d$ in the large $N$ limit and in some appropriate time scale.

\smallskip
First we need to rescale the time variable, and introduce a new time variable $\hat t$ so that, in new time scale, the action on any one of the $N$ particles
due to the $N-1$ other particles is of order $1$ as $N\to+\infty$\index{Mean-field scaling}. In other words, the new time variable $\hat t$ is chosen so that
$$
\frac{d\hat z_i}{d\hat t}=O(1)\hbox{ for each }i=1,\ldots,N\hbox{ as }N\to\infty\,.
$$

The action on the $i$th particle of the $N-1$ other particles is
$$
\sum_{j=1\atop j\not=i}^NK(\hat z_i,\hat z_j)\,,
$$
and it obviously contains $N-1$ terms of order $1$ (assuming each term $K(\hat z_i,\hat z_j)$ to be of order $1$, for instance). Set $\hat t=t/N$, then
$$
\frac{d\hat z_i}{d\hat t}=\frac1N\sum_{j=1\atop j\not=i}^NK(\hat z_i,\hat z_j)\,.
$$

From now on, we drop hats on all variables and consider as our starting point the rescaled problem
$$
\dot{z_i}(t)=\frac1N\sum_{j=1\atop j\not=i}^NK(z_i(t),z_j(t))\,,\qquad i=1,\ldots,N\,.
$$

\smallskip
At this point, we introduce an important assumption on the interaction kernel: the action of the $j$th particle on the $i$th particle must exactly balance the 
action of the $i$th particle on the $j$th particle. When the interaction is a force, this is precisely Newton's third law of mechanics. Thus we assume that the
interaction kernel satisfies
$$
K(z,z')=-K(z',z)\,,\qquad z,z'\in\bR^d\,.
$$

We have assumed here that the interaction kernel $K$ is defined on the whole $\bR^d\times\bR^d$ space; in particular, the condition above implies that
$K$ vanishes identically on the diagonal, i.e.
$$
K(z,z)=0\,,\qquad z\in\bR^d\,.
$$
Hence the restriction $j\not=i$ can be removed in the summation that appears on the right hand side of the ODEs governing the $N$-particle dynamics:
since $K(z_i(t),z_i(t))=0$ for all $i=1,\ldots,N$, one has
$$
\dot{z_i}(t)=\frac1N\sum_{j=1}^NK(z_i(t),z_j(t))\qquad i=1,\ldots,N\,.
$$

\smallskip
At this point, we can explain the key idea in the mean field limit: if the points $z_j(t)$ for $j=1,\ldots,N$ are ``distributed at time $t$ under the probability 
measure $f(t,dz)$'' in the large $N$ limit, then, one expects that
$$
\frac1N\sum_{j=1}^NK(z_i(t),z_j(t))\to\int_{\bR^d}K(z_i(t),z')f(t,dz')\qquad\hbox{ as }N\to+\infty\,.
$$
This suggests replacing the $N$-particle system of differential equations with the single differential equation 
$$
\dot{z}(t)=\int_{\bR^d}K(z(t),z')f(t,dz')\,.
$$
Here $f(t,dz)$ is unknown, as is $z(t)$, so that it seems that this single differential equation is insufficient to determine both these unknowns.

But one recognizes in the equality above the equation of characteristics for the mean field PDE\index{Mean field PDE}
$$
\d_tf+\Div_z(f\cK f)=0\,,
$$
where the notation $\cK$ designates the integral operator defined by the formula
$$
\cK f(t,z):=\int_{\bR^d}K(z,z')f(t,dz')\,.
$$
Now, this is a single PDE (in fact an integro-differential equation) for the single unknown $f$.

A priori $f$ is a time dependent Borel probability measure on $\bR^d$, so that the mean field PDE is to be understood in the sense of distributions on
$\bR^d$. In other words, 
$$
\frac{d}{dt}\int_{\bR^d}\phi(z)f(t,dz)=\int_{\bR^d}\cK f(t,z)\cdot\grad\phi(z)f(t,dz)
$$
for each test function\footnote{For each topological space $X$ and each finite dimensional vector space $E$ on $\bR$, we denote by $C_b(X,E)$ the 
set of continuous functions defined on $X$ with values in $E$ that are bounded on $X$. For each $n,k\ge 1$, we denote by $C^k_b(\bR^n,E)$ the set 
of functions of class $C^k$ defined on $\bR^n$ with values in $E$ all of whose partial derivatives are bounded on $\bR^n$: for each norm $|\cdot|_E$ 
on $E$, one has
$$
C^k_b(\bR^n,E):=\{f\in C^k(\bR^n,E)\hbox{ s.t. }\sup_{x\in\bR^n}|\d^\a f(x)|_E<\infty\hbox{Ê for each }\a\in\bN^n\}\,.
$$
We also denote $C_b(X):=C_b(X,\bR)$ and $C^k_b(\bR^n):=C^k_b(\bR^n,\bR)$.} $\phi\in C^1_b(\bR^d)$.

\smallskip
A very important mathematical object in the mathematical theory of the mean field limit is the empirical measure, which is defined below.

\begin{Def}
To each $N$-tuple $Z_N=(z_1,\ldots,z_N)\in(\bR^d)^N\simeq\bR^{dN}$, one associates its empirical measure\index{Empirical measure}
$$
\mu_{Z_N}:=\frac1N\sum_{j=1}^N\de_{z_j}\,.
$$
\end{Def}

\smallskip
The empirical measure of a $N$-tuple $Z_N\in(\bR^d)^N$ is a Borel probability measure on $\bR^d$. As we shall see in the next section, the $N$-tuple
$$
t\mapsto Z_N(t)=(z_1(t),\ldots,z_N(t))
$$ 
is a solution of the $N$-particle ODE system
$$
\dot{z}_i(t)=\frac1N\sum_{j=1}^NK(z_i(t),z_j(t))\,,\quad i=1,\ldots,N
$$
if and only if the empirical measure $\mu_{Z_N(t)}$ is a solution of the mean field PDE
$$
\d_t\mu_{Z_N(t)}+\Div_z(\mu_{Z_N(t)}\cK\mu_{Z_N(t)})=0\,.
$$

\smallskip
We conclude this section with a few exercises where the reader can verify that the formalism introduced here encompasses the two main examples of 
mean-field theories presented above, i.e. the two dimensional Euler equation and the Vlasov-Poisson system.

\smallskip
\noindent
\textbf{Exercise:} 

\noindent
1) Compute $\Dlt\ln|x|$ in the sense of distributions on $\bR^2$ (answer: $2\pi\de_0$).

\noindent
2) Define
$$
K(x,x'):=-\tfrac1{2\pi}\frac{J(x-x')}{|x-x'|^2}\,,\quad x\not=x'\in\bR^2\,,
$$
where $J$ designates the rotation of an angle $-\tfrac{\pi}2$:
$$
J=\left(\begin{matrix} 0\,\,\,&1\\-1&0\end{matrix}\right)
$$
For each $\om\equiv\om(t,x)$ belonging to $C^1_b(\bR_+\times\bR^2)$ such that $\Supp(\om(t,\cdot))$ is compact for each $t\ge 0$, prove that
the vector field $u$ defined by 
$$
u(t,x):=\int_{\bR^2}K(x,x')\om(x')dx'
$$
is of class $C^1_b$ on $\bR_+\times\bR^2$ and satisfies
$$
\Div_xu(t,x)=0\,,\quad\Div_x(Ju)(t,x)=\om(t,x)\,.
$$
3) Conclude that the two dimensional Euler equation for incompressible fluids can be put in the formalism described in the present section, except 
for the fact that the interaction kernel $K$ is singular on the diagonal of $\bR^2\times\bR^2$.

\smallskip
\noindent
\textbf{Exercise:} Let $(f,\phi)$ be a solution of the Vlasov-Poisson system satisfying the same assumptions as in the exercise of section \ref{SS-VP}.
Assume further that
$$
\iint_{\bR^3\times\bR^3}f(0,x,v)dxdv=1\,,\quad\hbox{ and }\iint_{\bR^3\times\bR^3}vf(0,x,v)dxdv=0\,.
$$
1) Prove that
$$
\iint_{\bR^3\times\bR^3}f(t,x,v)dxdv\!=\!1\quad\hbox{ and }\iint_{\bR^3\times\bR^3}vf(t,x,v)dxdv\!=\!0\quad\hbox{ for all }t\ge 0\,.
$$
2) Set $z=(x,v)$ and
$$
K(z,z')=K(x,v,x',v'):=\left(v-v',\tfrac{q^2}{4\pi\eps_0m}\frac{x-x'}{|x-x'|^3}\right)\,.
$$
Prove that
$$
\iint_{\bR^3\times\bR^3}K(x,v,x',v')f(t,x',v')dx'dv'=(v,-\tfrac{q}{m}\grad_x\phi(t,x))\,,
$$
where
$$
-\Dlt_x\phi(t,x)=\frac{q}{\eps_0}\int_{\bR^3}f(t,x,v)dv\,.
$$
3) Conclude that the Vlasov-Poisson system can be put in the formalism described in the present section, except for the fact that the interaction kernel 
$K$ is singular on the set $\{(x,v,x',v')\in(\bR^3)^4\hbox{ s.t. }x=x'\}$.

\section[Mean field characteristic flow]{The mean field characteristic flow}

Henceforth we assume that the interaction kernel $K:\,\bR^d\times\bR^d\to\bR^d$ satisfies the following assumptions. 

First $K$ is skew-symmetric:
$$
K(z,z')=-K(z',z)\quad\hbox{ for all }z,z'\in\bR^d\,.\leqno{(HK1)}
$$

Besides, $K\in C^1(\bR^d\times\bR^d;\bR^d)$, with bounded partial derivatives of order $1$. In other words,  there exists a constant $L\ge 0$ such that
$$
\sup_{z'\in\bR^d}|\grad_zK(z,z')|\le L\,,\quad\hbox{ and }\sup_{z\in\bR^d}|\grad_{z'}K(z,z')|\le L\,.\leqno{(HK2)}
$$

Applying the mean value theorem shows that assumption (HK2) implies that $K$ is Lipschitz continuous in $z$ uniformly in $z'$ (and conversely):
$$
\left\{\ba
\sup_{z'\in\bR^d}|K(z_1,z')\!-\!K(z_2,z')|\le L|z_1-z_2|\,,
\\
\sup_{z\in\bR^d}\,|\,K(z,z_1)-K(z,z_2)|\le L|z_1-z_2|\,.
\ea
\right.
$$

Assumption (HK2) also implies that $K$ grows at most linearly at infinity:
$$
|K(z,z')|\le L(|z|+|z'|)\,,\quad z,z'\in\bR^d\,.
$$

Notice also that the integral operator $\cK$ can be extended to the set of Borel probability measures\footnote{Henceforth, the set of Borel probability 
measures on $\bR^d$ will be denoted by $\cP(\bR^d)$.} on $\bR^d$ with finite moment of order $1$, i.e.
$$
\cP_1(\bR^d):=\left\{p\in\cP(\bR^d)\hbox{ s.t. }\int_{\bR^d}|z|p(dz)<\infty\right\}\,,
$$
in the obvious manner, i.e.
$$
\cK p(z):=\int_{\bR^d}K(z,z')p(dz')\,.
$$
The extended operator $\cK$ so defined maps $\cP_1(\bR^d)$ into the class $\Lip(\bR^d;\bR^d)$ of Lipschitz continuous vector fields on $\bR^d$.

With the assumptions above, one easily arrives at the existence and uniqueness theory for the $N$-body ODE system.

\begin{Thm}\lb{T-ExUnNbody}
Assume that the interaction kernel $K\in C^1(\bR^d\times\bR^d,\bR^d)$ satisfies assumptions (HK1-HK2). Then

\smallskip
\noindent
a) for each $N\ge 1$ and each $N$-tuple $Z_N^{in}=(z_1^{in},\ldots,z_N^{in})$, the Cauchy problem for the $N$-particle ODE system
$$
\left\{
\ba
{}&\dot{z_i}(t)=\frac1N\sum_{j=1}^NK(z_i(t),z_j(t))\,,\qquad i=1,\ldots,N\,,
\\
&z_i(0)=z_i^{in}\,,
\ea
\right.
$$
has a unique solution of class $C^1$ on $\bR$
$$
t\mapsto Z_N(t)=(z_1(t),\ldots,z_N(t))=:T_tZ_N^{in}\,;
$$

\smallskip
\noindent
b) the empirical measure $f(t,dz):=\mu_{T_tZ_N^{in}}$ is a weak solution of the Cauchy problem for the mean field PDE
$$
\left\{
\ba
{}&\d_tf+\Div_z(f\cK f)=0\,,
\\
&f\rstr_{t=0}=f^{in}\,.
\ea
\right.
$$
\end{Thm}

\smallskip
Statement a) follows from the Cauchy-Lipschitz theorem. Statement b) follows from the method of characteristics for the transport equation. For the sake of
being complete, we sketch the main steps in the proof of statement b), and leave the details as an exercise to be treated by the reader.

\smallskip
\noindent
\textbf{Exercise:} Let $b\equiv b(t,y)\in C([0,\tau];\bR^d)$ be such that $D_yb\in C([0,\tau];\bR^d)$ and\index{Method of characteristics}
$$
|b(t,y)|\le\ka(1+|y|)\leqno(H)
$$
for all $t\in[0,\tau]$ and $y\in\bR^d$, where $\ka$ is a positive constant.

\noindent
1) Prove that, for each $t\in[0,\tau]$, the Cauchy problem for the ODE
$$
\left\{
\ba
{}&\dot{Y}(s)=b(s,Y(s))\,,
\\
&Y(t)=y\,,
\ea
\right.
$$
has a unique solution $s\mapsto Y(s,t,y)$. What is the maximal domain of definition of this solution? What is the regularity of the map $Y$ viewed as a 
function of the 3 variables $s,t,y$?

\noindent
2) What is the role of assumption (H)?

\noindent
3) Prove that, for each $t_1,t_2,t_3\in[0,\tau]$ and $y\in\bR^d$, one has
$$
Y(t_3,t_2,Y(t_2,t_1,y))=Y(t_3,t_1,y)\,.
$$
4) Compute
$$
\d_tY(s,t,y)+b(t,y)\cdot\grad_yY(s,t,y)\,.
$$
5) Let $f^{in}\in C^1(\bR^d)$. Prove that the Cauchy problem for the transport equation
$$
\left\{
\ba
{}&\d_tf(t,y)+b(t,y)\cdot\grad_yf(t,y)=0\,,
\\
&f\rstr_{t=0}=f^{in}\,,
\ea
\right.
$$
has a unique solution $f\in C^1([0,\tau]\times\bR^d)$, and that this solution is given by the formula
$$
f(t,y)=f^{in}(Y(0,t,y))\,.
$$
6) Let $\mu^{in}$ be a Borel probability measure on $\bR^d$. Prove that the push-forward measure\footnote{Given two measurable spaces $(X,\cA)$ and
$(Y,\cB)$, a measurable map $\Phi:\,(X,\cA)\to(Y,\cB)$ and a measure $m$ on $(X,\cA)$, the push-forward of $m$ under $\Phi$ is the measure on $(Y,\cB)$
defined by the formula\index{Push-forward of a measure}
$$
\Phi\#m(B)=m(\Phi^{-1}(B))\,,\quad \hbox{ for all }B\in\cB\,.
$$}
$$
\mu(t):=Y(t,0,\cdot)\#\mu^{in}
$$
is a weak solution of
$$
\left\{
\ba
{}&\d_t\mu+\Div_y(\mu b)=0\,,
\\
&\mu\rstr_{t=0}=\mu^{in}\,.
\ea
\right.
$$
Hint: for $\phi\in C^1_c(\bR^d)$, compute
$$
\frac{d}{dt}\int_{\bR^d}\phi(Y(t,0,y))\mu^{in}(dy)\,.
$$
7) Prove that the unique weak solution\footnote{We designate by $w-\cP(\bR^d)$ the set $\cP(\bR^d)$ equipped with the weak topology of probability 
measures, i.e. the topology defined by the family of semi-distances
$$
d_\phi(\mu,\nu):=\left|\int_{\bR^d}\phi(z)\mu(dz)-\int_{\bR^d}\phi(z)\nu(dz)\right|
$$
as $\phi$ runs through $C_b(\bR^d)$.} $\mu\in C([0,\tau],w-\cP(\bR^d))$ of the Cauchy problem considered in 6) is the push-forward measure defined
by the formula
$$
\mu(t):=Y(t,0,\cdot)\#\mu^{in}
$$
for each $t\in[0,\tau]$. (Hint: for $\phi\in C^1_c(\bR^d)$, compute
$$
\frac{d}{dt}\la Y(0,t,\cdot)\#\mu(t),\phi\ra
$$
in the sense of distributions on $(0,\tau)$.)

\smallskip
For a solution of this exercise, see chapter 1, section 1 of \cite{BouGolPul}.

\smallskip
Our next step is to formulate and solve a new problem that will contain both the $N$-particle ODE system in the mean-field scaling and the mean-field PDE.

\begin{Thm}\lb{T-MFChar}
Assume that the interaction kernel $K\in C^1(\bR^d\times\bR^d,\bR^d)$ satisfies assumptions (HK1-HK2). For each $\zeta^{in}\in\bR^d$ and each Borel 
probability measure $\mu^{in}\in\cP_1(\bR^d)$, there exists a unique solution denoted by
$$
\bR\ni t\mapsto Z(t,\zeta^{in},\mu^{in})\in\bR^d
$$
of class $C^1$ of the problem
$$
\left\{
\ba
{}&\d_tZ(t,\zeta^{in},\mu^{in})=(\cK\mu(t))(Z(t,\zeta^{in},\mu^{in}))\,,
\\
&\mu(t)=Z(t,\cdot,\mu^{in})\#\mu^{in}\,,
\\
&Z(0,\zeta^{in},\mu^{in})=\zeta^{in}\,.
\ea
\right.
$$
\end{Thm}

Notice that the ODE governing the evolution of $t\mapsto Z(t,\zeta^{in},\mu^{in})$ is set in the single-particle phase space $\bR^d$, and not in the $N$-particle
phase space, as is the case of the ODE system studied in Theorem \ref{T-ExUnNbody}. 

Obviously, the ODE appearing in Theorem \ref{T-MFChar} is precisely the equation of characteristics for the mean field PDE. Henceforth, we refer to this ODE 
as the equations of ``mean field characteristics'', and to its solution $Z$ as the ``mean field characteristic flow''\index{Mean field characteristic flow}.

How the mean field characteristic flow $Z$ and the flow $T_t$ associated to the $N$-particle ODE system are related is explained in the next proposition.

\begin{Prop}\lb{P-MFNbody}
Assume that the interaction kernel $K\in C^1(\bR^d\times\bR^d,\bR^d)$ satisfies assumptions (HK1-HK2). For each $Z_N^{in}=(z_1^{in},\ldots,z_N^{in})$, the solution 
$$
T_tZ_N^{in}=(z_1(t),\ldots,z_N(t))
$$ 
of the $N$-body problem and the mean field characteristic flow $Z(t,\zeta^{in},\mu^{in})$ satisfy
$$
z_i(t)=Z(t,z_i^{in},\mu_{Z_N^{in}})\,,\quad i=1,\ldots,N\,,
$$
for all $t\in\bR$.
\end{Prop}

\begin{proof}[Proof of Proposition \ref{P-MFNbody}]
Define 
$$
\zeta_i(t):=Z(t,z_i^{in},\mu_{Z_N^{in}})\,,\quad i=1,\ldots,N\,.
$$
Then\footnote{The reader should be aware of the following subtle point. In classical references on distribution theory, such as \cite{Horm1}, the Dirac mass
is viewed as a distribution, therefore as an object that generalizes the notion of function. There is a notion of pull-back of a distribution under a $C^\infty$
diffeomorphism such that the pull-back of the Dirac mass at $y_0$ with a $C^\infty$ diffeomorphism $\chi:\,\bR^N\to\bR^N$ satisfying $\chi(x_0)=y_0$ is
$$
\de_{y_0}\circ\chi=|\Det(D\chi(x_0))|^{-1}\de_{x_0}\,.
$$
This notion is not to be confused with the push-forward under $\chi$ of the Dirac mass at $\de_{x_0}$ viewed as a probability measure, which, according to 
the definition in the previous footnote is\index{Push-forward of a measure}
$$
\chi\#\de_{x_0}=\de_{y_0}\,.
$$
In particular
$$
\chi\#\de_{x_0}\not=\de_{x_0}\circ\chi^{-1}
$$
unless $\chi$ has Jacobian determinant $1$ at $x_0$.}
$$
\mu(t)=Z(t,\cdot,\mu_{Z_N^{in}})\#\mu_{Z_N^{in}}=\frac1N\sum_{j=1}^N\de_{\zeta_j(t)}
$$
for all $t\in\bR$. Therefore, $\zeta_i$ satisfies
$$
\dot\zeta_i(t)=(\cK\mu(t))(\zeta_i(t))=\frac1N\sum_{j=1}^NK(\zeta_i(t),\zeta_j(t))\,,\quad i=1,\ldots,N\,,
$$
for all $t\in\bR$. Moreover
$$
\zeta_i(0)=Z(0,z_i^{in},\mu^{in})=z_i^{in}\,,\quad i=1,\ldots,N\,.
$$
Therefore, by uniqueness of the solution of the $N$-particle equation (Theorem \ref{T-ExUnNbody}), one has
$$
\zeta_i(t)=z_i(t)\,,
$$
for all $i=1,\ldots,N$ and all $t\in\bR$.
\end{proof}

\smallskip
The proof of Theorem \ref{T-MFChar} is a simple variant of the proof of the Cauchy-Lipschitz theorem. 

\begin{proof}[Proof of Theorem \ref{T-MFChar}]
Let $\mu^{in}\in\cP_1(\bR^d)$, and denote
$$
C_1:=\int_{\bR^d}|z|\mu^{in}(dz)\,.
$$
Let 
$$
X:=\left\{v\in C(\bR^d;\bR^d)\hbox{ s.t. }\sup_{z\in\bR^d}\frac{|v(z)|}{1+|z|}<\infty\right\}\,,
$$
which is a Banach space for the norm
$$
\|v\|_X:=\sup_{z\in\bR^d}\frac{|v(z)|}{1+|z|}\,.
$$

By assumption (HK2) on the interaction kernel $K$, for each $v,w\in X$, one has
$$
\ba
\left|\int_{\bR^d}K(v(z),v(z'))\mu^{in}(dz')-\int_{\bR^d}K(w(z),w(z'))\mu^{in}(dz')\right|
\\
\le L\int_{\bR^d}(|v(z)-w(z)|+|v(z')-w(z')|)\mu^{in}(dz')
\\
\le L\|v-w\|_X(1+|z|)+L\|v-w\|_X\int_{\bR^d}(1+|z'|)\mu^{in}(dz')
\\
=L\|v-w\|_X(1+|z|+1+C_1)
\\
\le L\|v-w\|_X(2+C_1)(1+|z|)\,.
\ea
$$

Define a sequence $(Z_n)_{n\ge 0}$ by induction, as follows:
$$
\left\{
\ba
{}&Z_{n+1}(t,\zeta)=\zeta+\int_0^t\int_{\bR^d}K(Z_n(t,\zeta),Z_n(t,\zeta'))\mu^{in}(d\zeta')ds\,,\quad n\ge 0\,,
\\
&Z_0(t,\zeta)=\zeta\,.
\ea
\right.
$$
One checks by induction with the inequality above that, for each $n\in\bN$,
$$
\|Z_{n+1}(t,\cdot)-Z_n(t,\cdot)\|_X\le\frac{((2+C_1)L|t|)^n}{n!}\|Z_1(t,\cdot)-Z_0(t,\cdot)\|_X\,.
$$
Since 
$$
\ba
|Z_1(t,\zeta)-\zeta|&=\left|\int_0^t\int_{\bR^d}K(\zeta,\zeta')\mu^{in}(d\zeta')ds\right|
\\
&\le\int_0^{|t|}\int_{\bR^d}L(|\zeta|+|\zeta'|)\mu^{in}(d\zeta')ds
\\
&=\int_0^{|t|}L(|\zeta|+C_1)ds\le L(1+C_1)(1+|\zeta|)|t|\,,
\ea
$$
one has
$$
\|Z_{n+1}(t,\cdot)-Z_n(t,\cdot)\|_X\le\frac{((2+C_1)L|t|)^{n+1}}{n!}\,.
$$

Thus, for each $\tau>0$, 
$$
Z_n(t,\cdot)\to Z(t,\cdot)\quad\hbox{ in }X\hbox{ uniformly on }[-\tau,\tau]\,,
$$
where $Z\in C(\bR;X)$ satisfies
$$
Z(t,\zeta)=\zeta+\int_0^t\int_{\bR^d}K(Z(s,\zeta),Z(s,\zeta'))\mu^{in}(d\zeta')ds
$$
for all $t\in\bR$ and all $\zeta\in\bR^d$. 

If $Z$ and $\tilde Z\in C(\bR;X)$ satisfy the integral equation above, then
$$
Z(t,\zeta)-\tilde Z(t,\zeta)=\int_{\bR^d}(K(Z(s,\zeta),Z(s,\zeta'))-K(\tilde Z(s,\zeta),\tilde Z(s,\zeta')))\mu^{in}(d\zeta')\,,
$$
so that, for all $t\in\bR$, one has
$$
\|Z(t,\cdot)-\tilde Z(t,\cdot)\|_X\le L(2+C_1)\left|\int_0^t\|Z(s,\cdot)-\tilde Z(s,\cdot)\|_Xds\right|\,.
$$
This implies that 
$$
\|Z(t,\cdot)-\tilde Z(t,\cdot)\|_X=0
$$
by Gronwall's inequality, so that $Z=\tilde Z$. Hence the integral equation has only one solution $Z\in C(\bR;X)$.

Since $Z\in C(\bR_+;X)$, $K\in C^1(\bR^d\times\bR^d,\bR^d)$ satisfies (HK2) and $\mu^{in}\in \cP_1(\bR^d)$, the function
$$
s\mapsto\int_{\bR^d}K(Z(s,\zeta),Z(s,\zeta'))\mu^{in}(d\zeta')
$$
is continuous on $\bR$. 

Using the integral equation shows that the function $t\mapsto Z(t,\zeta)$ is of class $C^1$ on $\bR$ and satisfies
$$
\left\{
\ba
{}&\d_tZ(t,\zeta)=\int_{\bR^d}K(Z(t,\zeta),Z(t,\zeta'))\mu^{in}(d\zeta')\,,
\\
&Z(0,\zeta)=\zeta\,.
\ea
\right.
$$
Substituting $z'=Z(t,\zeta')$ in the integral above, one has
$$
\int_{\bR^d}K(Z(t,\zeta),Z(t,\zeta'))\mu^{in}(d\zeta')=\int_{\bR^d}K(Z(t,\zeta),z')Z(t,\cdot)\#\mu^{in}(dz')
$$
so that the element $Z$ of $C(\bR;X)$ so constructed is the unique solution of the mean field characteristic equation.
\end{proof}

\smallskip
References for this and the previous section are \cite{BraunHepp,NeunWick}.

\section[Dobrushin's estimate]{Dobrushin's stability estimate and the mean field limit}

\subsection{The Monge-Kantorovich distance}

For each $r>1$, we denote by $\cP_r(\bR^d)$ the set of Borel probability measures on $\bR^d$ with a finite moment of order $r$, i.e. satisfying
$$
\int_{\bR^d}|z|^r\mu(dz)<\infty\,.
$$

Given $\mu,\nu\in\cP_r(\bR^d)$, we define $\Pi(\mu,\nu)$ to be the set of Borel probability measures $\pi$ on $\bR^d\times\bR^d$ with first and second
marginals $\mu$ and $\nu$ respectively. Equivalently, for each $\pi\in\cP(\bR^d\times\bR^d)$, 
$$
\pi\in\Pi(\mu,\nu)\Leftrightarrow\iint_{\bR^d\times\bR^d}(\phi(x)+\psi(y))\pi(dxdy)=\int_{\bR^d}\phi(x)\mu(dx)+\int_{\bR^d}\psi(y)\nu(dy)
$$
for each $\phi,\psi\in C(\bR^d)$ such that $\phi(z)=O(|z|^r)$ and $\psi(z)=O(|z|^r)$ as $|z|\to\infty$.

Probability measures belonging to $\Pi(\mu,\nu)$ are sometimes referred to as ``couplings of $\mu$ and $\nu$''\index{Coupling of two probability measures}.

\smallskip
\noindent
\textbf{Exercise:} Check that, if $\mu$ and $\nu\in\cP_r(\bR^d)$ for some $r>0$, then one has $\Pi(\mu,\nu)\subset\cP_r(\bR^d\times\bR^d)$.

With these elements of notation, we now introduce the notion of Monge-Kantorovich distance\index{Monge-Kantorovich distance}.

\begin{Def}
For each $r\ge 1$ and each $\mu,\nu\in\cP_r(\bR^d)$, the Monge-Kantoro\-vich distance $\Dist_{MK,r}(\mu,\nu)$  between $\mu$ and $\nu$ is defined by the
formula
$$
\Dist_{MK,r}(\mu,\nu)=\inf_{\pi\in\Pi(\mu,\nu)}\left(\iint_{\bR^d\times\bR^d}|x-y|^r\pi(dxdy)\right)^{1/r}\,.
$$
\end{Def}

\smallskip
These distances also go by the name of ``Kantorovich-Rubinstein distances'' or ``Wasserstein distances'' --- although the minimization problem in the right
hand side of the formula defining $\Dist_{MK,r}$ had been considered for the first time by Monge\footnote{Monge's original problem was to minimize over
the class of all Borel measurable transportation maps $T:\,\bR^d\to\bR^d$  such that $T\#\mu=\nu$ the transportation cost
$$
\int_{\bR^d}|x-T(x)|\mu(dx)\,.
$$ }
and systematically studied by Kantorovich.

We shall use the Monge-Kantorovich distances only as a convenient tool for studying the stability of the mean field characteristic flow. Therefore, we shall
not attempt to present the mathematical theory of these distances and refer instead to the C. Villani's books \cite{VillaniTOT,VillaniOTON} for a very detailed
discussion of this topic.

However, it is  useful to know the following property that is special to the case $r=1$.

\begin{Prop}\lb{P-MKDual}
The Monge-Kantorovich distance with exponent $1$ is also given by the formula
$$
\Dist_{MK,1}(\mu,\nu)=\sup_{\phi\in\Lip(\bR^d)\atop\Lip(\phi)\le 1}\left|\int_{\bR^d}\phi(z)\mu(dz)-\int_{\bR^d}\phi(z)\nu(dz)\right|\,,
$$
with the notation
$$
\Lip(\phi):=\sup_{x\not=y\in\bR^d}\frac{|\phi(x)-\phi(y)|}{|x-y|}\,.
$$
for the Lipschitz constant of $\phi$.
\end{Prop}

\smallskip
The proof of this proposition is based on a duality argument in optimization: see for instance Theorems 1.14 and 7.3 (i) in \cite{VillaniTOT}.

\subsection{Dobrushin's estimate}\index{Dobrushin's estimate}

As explained in Proposition \ref{P-MFNbody}, the mean field characteristic flow contains all the relevant information about both the mean field PDE and
the $N$-particle ODE system.

Dobrushin's approach to the mean field limit is based on the idea of proving the stability of the mean field characteristic flow $Z(t,\zeta^{in},\mu^{in})$ in 
\textit{both} the initial position in phase space $\zeta^{in}$ and the initial distribution $\mu^{in}$. As we shall see, the Monge-Kantorovich distance is the
best adapted mathematical tool to measure this stability.

Dobrushin's idea ultimately rests on the following key computation. Let $\zeta_1^{in},\zeta_2^{in}\in\bR^d$, and let $\mu_1^{in},\mu_2^{in}\in\cP_1(\bR^d)$.
Then
$$
\ba
Z(t,\zeta_1,\mu_1^{in})&-Z(t,\zeta_2,\mu_2^{in})=\zeta_1-\zeta_2
\\
&+\int_0^t\int_{\bR^d}K(Z(s,\zeta_1,\mu_1^{in}),z')\mu_1(s,dz')ds
\\
&-\int_0^t\int_{\bR^d}K(Z(s,\zeta_2,\mu_2^{in}),z')\mu_2(s,dz')ds\,.
\ea
$$
Since $\mu_j(t)=Z(t,\cdot,\mu_j^{in})\#\mu_j^{in}$ for $j=1,2$, each inner integral on the right hand side of the equality above can be expressed as follows:
$$
\ba
\int_{\bR^d}K(Z(s,\zeta_j,\mu_j^{in}),z')&\mu_j(s,dz')
\\
&=\int_{\bR^d}K(Z(s,\zeta_j,\mu_j^{in}),Z(s,\zeta'_j,\mu_j^{in}))\mu_j^{in}(d\zeta'_j)
\ea
$$
for $j=1,2$. Therefore, for each coupling $\pi^{in}\in\cP_1(\mu_1^{in},\mu_2^{in})$, one has
$$
\ba
\int_{\bR^d}K(Z(s,\zeta_1,\mu_1^{in}),Z(s,\zeta'_1,\mu_1^{in}))\mu_1^{in}(d\zeta'_1)&
\\
-
\int_{\bR^d}K(Z(s,\zeta_2,\mu_2^{in}),Z(s,\zeta'_2,\mu_2^{in}))\mu_2^{in}(d\zeta'_2)&
\\
=
\iint_{\bR^d\times\bR^d}(K(Z(s,\zeta_1,\mu_1^{in}),Z(s,\zeta'_1,\mu_1^{in}))&
\\
-K(Z(s,\zeta_2,\mu_2^{in}),Z(s,\zeta'_2,\mu_2^{in})))&\pi^{in}(d\zeta'_1,d\zeta'_2)\,,
\ea
$$
so that
$$
\ba
Z(t,\zeta_1,\mu_1^{in})-Z(t,\zeta_2,\mu_2^{in})=\zeta_1-\zeta_2\qquad\qquad\qquad\qquad&
\\
+\int_0^t\iint_{\bR^d\times\bR^d}(K(Z(s,\zeta_1,\mu_1^{in}),Z(s,\zeta'_1,\mu_1^{in}))&
\\
-K(Z(s,\zeta_2,\mu_2^{in}),Z(s,\zeta'_2,\mu_2^{in})))&\pi^{in}(d\zeta'_1,d\zeta'_2)ds\,.
\ea
$$
This last equality is the key observation in Dobrushin's argument, which explains the role of couplings of $\mu_1^{in}$ and $\mu_2^{in}$ in this problem,
and therefore why it is natural to use the Monge-Kantorovich distance.
 
After this, the end of the argument is plain sailing. By assumption (HK2) on the interaction kernel $K$, for all $a,a',b,b'\in\bR^d$, one has
$$
\ba
|K(a,a')-K(b,b')|&\le|K(a,a')-K(b,a')|+|K(b,a')-K(b,b')|
\\
&\le L|a-b|+L|a'-b'|\,.
\ea
$$
Therefore
$$
\ba
{}&|Z(t,\zeta_1,\mu_1^{in})-Z(t,\zeta_2,\mu_2^{in})|
\\
&\qquad\le|\zeta_1-\zeta_2|+L\int_0^t|Z(s,\zeta_1,\mu_1^{in})-Z(s,\zeta_2,\mu_2^{in})|ds
\\
&\qquad+
L\int_0^t\iint_{\bR^d\times\bR^d}|Z(s,\zeta'_1,\mu_1^{in})-Z(s,\zeta'_2,\mu_2^{in})|\pi^{in}(d\zeta'_1d\zeta'_2)ds\,.
\ea
$$

It is convenient at this point to introduce the notation
$$
D[\pi](s):=\iint_{\bR^d\times\bR^d}|Z(s,\zeta'_1,\mu_1^{in})-Z(s,\zeta'_2,\mu_2^{in})|\pi(d\zeta'_1d\zeta'_2)
$$
for each $\pi\in\cP_1(\bR^d\times\bR^d)$. Thus, the previous inequality becomes
$$
\ba
{}&|Z(t,\zeta_1,\mu_1^{in})-Z(t,\zeta_2,\mu_2^{in})|\le|\zeta_1-\zeta_2|
\\
&\qquad+L\int_0^t|Z(s,\zeta_1,\mu_1^{in})-Z(s,\zeta_2,\mu_2^{in})|ds+L\int_0^tD[\pi^{in}](s)ds\,.
\ea
$$

Integrating both sides of the inequality above with respect to $\pi^{in}(d\zeta_1d\zeta_2)$ leads to
$$
\ba
D[\pi^{in}](t)&\le\iint_{\bR^d\times\bR^d}|\zeta_1-\zeta_2|\pi^{in}(d\zeta_1d\zeta_2)+2L\int_0^tD[\pi^{in}](s)ds
\\
&=D[\pi^{in}](0)+2L\int_0^tD[\pi^{in}](s)ds\,.
\ea
$$
By Gronwall's inequality, we conclude that, for all $t\in\bR$, one has
$$
D[\pi^{in}](t)\le D[\pi^{in}](0)e^{2L|t|}\,.
$$

\smallskip
Now we can state Dobrushin's stability theorem.

\begin{Thm}[Dobrushin]\lb{T-Dobru}\index{Dobrushin's estimate}
Assume that $K\in C^1(\bR^d\times\bR^d,\bR^d)$ satisfies (HK1-HK2). Let $\mu_1^{in},\mu_2^{in}\in\cP_1(\bR^d)$. 
For all $t\in\bR$, let 
$$
\left\{
\ba
\mu_1(t)=Z(t,\cdot,\mu_1^{in})\#\mu_1^{in}\,,
\\
\mu_2(t)=Z(t,\cdot,\mu_2^{in})\#\mu_2^{in}\,,
\ea
\right.
$$
where $Z$ is the mean field characteristic flow defined in Theorem \ref{T-MFChar}. 

Then, for all $t\in\bR$, one has
$$
\Dist_{MK,1}(\mu_1(t),\mu_2(t))\le e^{2L|t|}\Dist_{MK,1}(\mu_1^{in},\mu_2^{in})\,.
$$
\end{Thm}

\begin{proof}
We have seen that, for all $\mu_1^{in},\mu_2^{in}\in\cP_1(\bR^d)$ and all $\pi^{in}\in\Pi(\mu_1^{in},\mu_2^{in})$, one has
$$
D[\pi^{in}](t)\le D[\pi^{in}](0)e^{2L|t|}
$$
for all $t\in\bR$. 

Since $Z(t,\cdot,\mu_j^{in})\#\mu_j^{in}=\mu_j(t)$ for $j=1,2$, the map
$$
\Phi_t:\,(\zeta_1,\zeta_2)\mapsto(Z(t,\zeta_1,\mu_1^{in}),Z(t,\zeta_2,\mu_2^{in}))
$$
satisfies
$$
\Phi_t\#\pi^{in}=\pi(t)\in\Pi(\mu_1(t),\mu_2(t))
$$
for all $t\in\bR$, since $\pi^{in}\in\Pi(\mu_1^{in},\mu_2^{in})$.

Thus
$$
\ba
\Dist_{MK,1}&(\mu_1(t),\mu_2(t))
=\inf_{\pi\in\Pi(\mu_1(t),\mu_2(t))}\iint_{\bR^d\times\bR^d}|\zeta_1-\zeta_2|\pi(d\zeta_1d\zeta_2)
\\
&\le\inf_{\pi^{in}\in\Pi(\mu_1^{in},\mu_2^{in})}\iint_{\bR^d\times\bR^d}|Z(t,\zeta_1,\mu_1^{in})-Z(t,\zeta_2,\mu_2^{in})|\pi^{in}(d\zeta_1d\zeta_2)
\\
&=\inf_{\pi^{in}\in\Pi(\mu_1^{in},\mu_2^{in})}D[\pi^{in}](t)
\le e^{2L|t|}\inf_{\pi^{in}\in\Pi(\mu_1^{in},\mu_2^{in})}D[\pi^{in}](0)
\\
&=e^{2L|t|}\Dist_{MK,1}(\mu_1^{in},\mu_2^{in})
\ea
$$
which concludes the proof.
\end{proof}

\smallskip
The discussion in this section is inspired from \cite{Dobrushin}; see also \cite{MarchioPulvi}. The interested reader is also referred to the very interesting paper 
\cite{Loeper} where Monge-Kantorovich distances with exponents different from $1$ are used in the same context --- see also \cite{Hauray09}.

\subsection{The mean field limit}

The mean field limit of the $N$-particle system is a consequence of Dobrushin's stability theorem, as explained below.

\begin{Thm}\lb{T-ClassMFLimit}
Assume that the interaction kernel $K\in C^1(\bR^d\times\bR^d)$ and satisfies assumptions (HK1-HK2). Let $f^{in}$ be a probability density on $\bR^d$ 
such that
$$
\int_{\bR^d}|z|f^{in}(z)dz<\infty\,.
$$
Then the Cauchy problem for the mean field PDE
$$
\left\{
\ba
{}&\d_tf(t,z)+\Div_z(f(t,z)\cK f(t,z))=0\,,\quad z\in\bR^d\,,\,\,t\in\bR\,,
\\
&f\rstr_{t=0}=f^{in}
\ea
\right.
$$
has a unique weak solution $f\in C(\bR;L^1(\bR^d))$. 

For each $N\ge 1$, let $\fZ(N)=(z_{1,N}^{in},\ldots,z_{N,N}^{in})\in(\bR^d)^N$ be such that 
$$
\mu_{\fZ(N)}=\frac1N\sum_{=1}^N\de_{z_{j,N}^{in}}
$$
satisfies
$$
\Dist_{MK,1}(\mu_{\fZ(N)},f^{in})\to 0\quad\hbox{ as }N\to\infty\,.
$$

Let $t\mapsto T_t\fZ(N)=(z_{1,N}(t),\ldots,z_{N,N}(t))\in(\bR^d)^N$ be the solution of the $N$-particle ODE system with initial data $\fZ(N)$, i.e.
$$
\left\{
\ba
{}&\dot{z}_i(t)=\frac1N\sum_{j=1}^NK(z_i(t),z_j(t))\,,\quad i=1,\ldots,N,
\\
&z_i(0)=z_i^{in}\,.
\ea
\right.
$$

Then\footnote{The notation $\scrL^d$ designates the Lebesgue measure on $\bR^d$.} 
$$
\mu_{T_t\fZ(N)}\wto f(t,\cdot)\scrL^d\hbox{ as }N\to\infty
$$
in the weak topology of probability measures, with convergence rate 
$$
\Dist_{MK,1}(\mu_{T_t\fZ(N)},f(t,\cdot)\scrL^d)\le e^{2L|t|} \Dist_{MK,1}(\mu_{\fZ(N)},f^{in})\to 0
$$
as $N\to\infty$ for each $t\in\bR$.
\end{Thm}

\begin{proof}
By Theorem \ref{T-MFChar} and questions 6 and 7 in the exercise on the method of characteristics before Theorem \ref{T-MFChar}, one has
$$
f(t,\cdot)\scrL^d=Z(t,\cdot,f^{in}\scrL^d)\#f^{in}\scrL^d
$$
for all $t\in\bR$. This implies in particular the uniqueness of the solution of the Cauchy problem in $C(\bR;L^1(\bR^d))$ for the mean field PDE.

By Proposition \ref{P-MFNbody},
$$
\mu_{T_t\fZ(N)}=Z(t,\cdot,\mu_{\fZ(N)})\#\mu_{\fZ(N)}
$$
for all $t\in\bR$. 

By Dobrushin's stability estimate,
$$
\Dist_{MK,1}(\mu_{T_t\fZ(N)},f(t,\cdot)\scrL^d)\le e^{2L|t|} \Dist_{MK,1}(\mu_{\fZ(N)},f^{in})
$$
for all $t\in\bR$, and since we have chosen $\fZ(N)$ so that
$$
\Dist_{MK,1}(\mu_{\fZ(N)},f^{in})\to 0
$$
as $N\to\infty$, we conclude that
$$
\Dist_{MK,1}(\mu_{T_t\fZ(N)},f(t,\cdot)\scrL^d)\to 0
$$
as $N\to\infty$ for each $t\in\bR$.

As for weak convergence, pick $\phi\in\Lip(\bR^d)$; then
$$
\ba
\left|\int_{\bR^d}\phi(z)\mu_{T_t\fZ(N)}(dz)-\int_{\bR^d}\phi(z)f(t,z)dz\right|
\\
=
\left|\iint_{\bR^d\times\bR^d}(\phi(x)-\phi(y))\pi(dxdy)\right|
\\
\le
\iint_{\bR^d\times\bR^d}|\phi(x)-\phi(y)|\pi(dxdy)
\\
\le
\Lip(\phi)\iint_{\bR^d\times\bR^d}|x-y|\pi(dxdy)
\ea
$$
for each $\pi\in\Pi(\mu_{T_t\fZ(N)},f(t,\cdot)\scrL^d)$. Thus
$$
\ba
\left|\int_{\bR^d}\phi(z)\mu_{T_t\fZ(N)}(dz)-\int_{\bR^d}\phi(z)f(t,z)dz\right|
\\
\le
\Lip(\phi)\inf_{\pi\in\Pi(\mu_{T_t\fZ(N)},f(t,\cdot)\scrL^d)}\iint_{\bR^d\times\bR^d}|x-y|\pi(dxdy)
\\
=\Lip(\phi)\Dist_{MK,1}(\mu_{T_t\fZ(N)},f(t,\cdot)\scrL^d)\to 0
\ea
$$
for each $t\in\bR$ as $N\to\infty$. (Notice that the inequality above is an obvious consequence of the definition of $\Dist_{MK,1}$, so that the equality in 
Proposition \ref{P-MKDual} is not needed here.)

This is true in particular for each $\phi\in C^1_c(\bR^d)$, and since $C^1_c(\bR^d)$ is dense in $C_c(\bR^d)$, we conclude that
$$
\int_{\bR^d}\phi(z)\mu_{T_t\fZ(N)}(dz)\to\int_{\bR^d}\phi(z)f(t,z)dz
$$
as $N\to\infty$ for each $\phi\in C_c(\bR^d)$. Since 
$$
\int_{\bR^d}\mu_{T_t\fZ(N)}(dz)=\int_{\bR^d}f(t,z)dz=1
$$
for all $t\in\bR$, we conclude that the convergence above holds for each $\phi\in C_b(\bR^d)$, which means that
$$
\mu_{T_t\fZ(N)}\to f(t,\cdot)\scrL^d
$$
as $N\to\infty$ in the weak topology of probability measures, by applying Theorem 6.8 in chapter II of \cite{Mallia}, sometimes referred to as the ``portmanteau 
theorem''.
\end{proof}

\smallskip
The theorem above is the main result on the mean field limit in \cite{NeunWick,BraunHepp,Dobrushin}.

\subsection{On the choice of the initial data}

In practice, using Theorem \ref{T-ClassMFLimit} as a rigorous justification of the mean field limit requires being able to generate $N$-tuples of the form
$\fZ(N)=(z_{1,N}^{in},\ldots,z_{N,N}^{in})\in(\bR^d)^N$ such that 
$$
\mu_{\fZ(N)}=\frac1N\sum_{=1}^N\de_{z_{j,N}^{in}}
$$
satisfies
$$
\Dist_{MK,1}(\mu_{\fZ(N)},f^{in})\to 0\quad\hbox{ as }N\to\infty\,.
$$

Assume that $f^{in}$ is a probability density on $\bR^d$ such that
$$
\int_{\bR^d}|z|^2f(z)dz<\infty\,.
$$

Let $\Om:=(\bR^d)^{\bN^*}$, the set of sequences of points in $\bR^d$ indexed by $\bN^*$. Let $\cF$ be the $\si$-algebra on $\Om$ generated by cylinders, 
i.e. by sets of the form
$$
\ba
\prod_{n\ge 1}B_n\quad\hbox{ with }&B_n\hbox{ Borel set in }\bR^d
\\
\hbox{ and }&B_n=\bR^d\hbox{ for all but finitely many }n\,.
\ea
$$
Finally, we endow the measurable space $(\Om,\cF)$ with the probability measure $\bP:=(f^{in})^{\otimes\infty}$, defined on the set of cylinders of $\Om$ 
by the formula
$$
\bP\left(\prod_{n\ge 1}B_n\right)=\prod_{n\ge 1}f^{in}(B_n)\,.
$$
(Notice that $f^{in}(B_n)=1$ for all but finitely many $n$, since $B_n=\bR^d$ except for finitely many $n$.)

\begin{Thm}\lb{T-LLN}
For each $\bz^{in}=(z_k^{in})_{k\ge 1}\in\Om$, let $Z_N^{in}=(z_1^{in},\ldots,z_N^{in})$. Then
$$
\Dist_{MK,1}(\mu_{Z_N^{in}},f^{in}\scrL^d)\to 0
$$
as $N\to\infty$ for $\bP$-a.e. $\bz^{in}\in\Om$.
\end{Thm}

\begin{proof}
For $\phi\in C_c(\bR^d)$ or $\phi(z)=|z|$, consider the sequence of random variables on $(\Om,\cF)$ defined by
$$
Y_n(\bz)=\phi(z_n)\,,
$$
where
$$
\bz:=(z_1,\ldots,z_n,\ldots)\in\Om\,.
$$

The random variables $Y_n$ are identically distributed, since
$$
\bP(Y_n\ge a)=\int_{\bR^d}\indc_{\phi(z)\ge a}f^{in}(z)dz
$$
is independent of $n$. 

The random variables $Y_n$ are also independent, since for all $N\ge 1$ and all $g_1,\ldots,g_N\in C_b(\bR)$, one has
$$
\bE^\bP(g_1(Y_1)\ldots g_N(Y_N))=\prod_{k=1}^N\int_{\bR^d}g_k(\phi(z))f^{in}(z)dz=\prod_{k=1}^N\bE^\bP(g_k(Y_k))\,.
$$

Finally, the random variables $Y_n$ have finite variance since
$$
\bE^\bP(|Y_n|^2)=\int_{\bR^d}|z|^2f^{in}(z)dz<\infty\,.
$$

By the strong law of large numbers (see Theorem 3.27 in \cite{Breiman}), one has
$$
\ba
\La\frac1N\sum_{k=1}^N\de_{z_k},\phi\Ra=\frac1N\sum_{k=1}^NY_k\to\bE^\bP(Y_1)=\int_{\bR^d}\phi(z)f^{in}(z)dz
\ea
$$
for $\bP$-a.e. $\bz$.

Since $C_c(\bR^d)$ is separable, one can assume that the $\bP$-negligible set is the same for all $\phi\in C_c(\bR^d)$, and take its union with the one 
corresponding to $\phi(z)=|z|$. This means precisely that
$$
\frac1N\sum_{k=1}^N\de_{z_k}\to f^{in}\scrL^d
$$
weakly in $\cP_1(\bR^d)$ for $\bP$-a.e. $\bz\in\Om$. One concludes the proof with the lemma below.
\end{proof}

\begin{Lem}\lb{L-DMKWeakTopo}
The Monge-Kantorovich distance $\Dist_{MK,1}$ metricizes the topology of weak convergence on $\cP_1(\bR^d)$. In other words, given a sequence
$(\mu_n)_{n\ge 1}$ of elements of $\cP_1(\bR^d)$ and $\mu\in\cP_1(\bR^d)$, the two following statements are equivalent:

\smallskip
\noindent
(1) $\Dist_{MK,1}(\mu_n,\mu)\to 0$ as $n\to\infty$;

\smallskip
\noindent
(2) $\mu_n\to\mu$ weakly in $\cP(\bR^d)$ as $n\to\infty$ and
$$
\sup_{n}\int_{\bR^d}|z|\indc_{|z|\ge R}\mu_n(dz)\to 0\quad\hbox{ as }R\to\infty\,.
$$
\end{Lem}

\smallskip
For a proof of Lemma \ref{L-DMKWeakTopo}, see \cite{VillaniTOT}.

\medskip
\noindent
\textbf{Exercise:} The reader is invited to verify the fact that one can choose the $\bP$-negligible set that appears in the proof of Theorem \ref{T-LLN} to be the
same for all $\phi\in C_c(\bR^d)$ and for $\phi(z)=|z|$. Here is an outline of the argument.

\noindent
a) Let $R>0$; let $E_R$ be the set of real-valued continuous functions defined on $[-R,R]^d$ that vanish identically on $\d[-R,R]^d$, equipped with
the sup-norm
$$
\|\phi\|:=\sup_{x\in[-R,R]^d}|\phi(x)|\,.
$$
Prove that $E_R$ is a separable Banach space.

\smallskip
Denote by $\cN_\phi$ be the set of $\bz\in\Om$ such that
$$
\La\frac1N\sum_{k=1}^N\de_{z_k},\phi\Ra
$$
does not converge to 
$$
\int_{\bR^d}\phi(z)f^{in}(z)dz
$$
as $N\to\infty$. Let $R>0$ and let $(\phi_n)_{n\ge 1}$ be a dense sequence of elements of $E_R$, extended by $0$ to $\bR^d$. Define
$$
\cN_R:=\bigcup_{n\ge 1}\cN_{\phi_n}\,.
$$
b) Prove that 
$$
\La\frac1N\sum_{k=1}^N\de_{z_k},\phi\Ra\to\int_{\bR^d}\phi(z)f^{in}(z)dz
$$ 
as $N\to\infty$ for all $\phi\in E_R$ and all $\bz\notin\cN_R$. (Hint: pick $\phi\in E_R$ and $\eps>0$, and choose $m:=m(\phi,\eps)$ such that 
$\|\phi-\phi_m\|<\eps$. With the decomposition
$$
\ba
\La\frac1N\sum_{k=1}^N\de_{z_k},\phi\Ra-\int_{\bR^d}\phi(z)f^{in}(z)dz=\La\frac1N\sum_{k=1}^N\de_{z_k},\phi-\phi_m\Ra
\\
+\La\frac1N\sum_{k=1}^N\de_{z_k},\phi_m\Ra-\int_{\bR^d}\phi_m(z)f^{in}(z)dz
\\
+\int_{\bR^d}(\phi_m(z)-\phi(z))f^{in}(z)dz\,,
\ea
$$
prove that
$$
\left|\La\frac1N\sum_{k=1}^N\de_{z_k},\phi\Ra-\int_{\bR^d}\phi(z)f^{in}(z)dz\right|<3\eps
$$
for all $\bz\notin\cN_R$ provided that $N\ge N_0=N_0(\eps,\phi)$.)

\noindent
c) Complete the proof of Theorem \ref{T-LLN}.

\medskip
Thus, using Theorem \ref{T-ClassMFLimit} to prove the mean field limit requires choosing 
$$
\fZ(N)=(z_{1,N}^{in},\ldots,z_{N,N}^{in})\in(\bR^d)^N
$$
for each $N\ge 1$ so that
$$
\Dist_{MK,1}(\mu_{\fZ(N)},f^{in}\scrL^d)\to 0\quad\hbox{ as }N\to\infty\,.
$$

Theorem \ref{T-LLN} provides us with a strategy for making this choice, which is to draw an infinite sequence $z_j^{in}$ at random and independently with 
distribution $f^{in}\scrL^d$, and to set $z_{j,N}^{in}:=z_j^{in}$. This strategy avoids the unpleasant task of having to change the first terms in $\fZ(N)$ as 
$N\to\infty$. 

Since Dobrushin's estimate bounds $\Dist_{MK,1}(f(t,\cdot)\scrL^d,\mu_{T_t\fZ(N)})$ in terms of $\Dist_{MK,1}(f^{in}\scrL^d,\mu_{\fZ(N)})$, having an explicit 
bound on $\Dist_{MK,1}(f^{in}\scrL^d,\mu_{\fZ(N)})$ would provide us with a quantitative error estimate for the mean field limit. Such a bound will be given 
below --- see Theorem \ref{T-HoroKara}.

\smallskip
More details on the topics discussed in the present section are to be found in \cite{BraunHepp}, as well as a precise statement concerning the behavior
of fluctuations around the mean field limit --- in some sense, the asymptotic behavior at next order after the mean field limit (see Theorem 3.5 in 
\cite{BraunHepp}).

\section[Method based on BBGKY hierarchy]{The BBGKY hierarchy and the mean field limit}

In the previous derivation of the mean field limit of the $N$-particle system with interaction kernel $K$ satisfying assumptions (HK1)-(HK2), we benefited
from a happy circumstance, i.e. the fact that the empirical measure built on any solution of the $N$-particle ODE system is an \textit{exact} solution of the 
mean field PDE. This is why the mean field limit was reduced to the stability of the solution of the mean field PDE in terms of its initial data, which follows
from Dobrushin's stability estimate.

However, there are other situations in statistical mechanics where the empirical measure built on solutions of the $N$-particle ODE system may not be 
an exact solution of the target equation --- the best known example of this being the Boltzmann equation of the kinetic theory of gases. There are various
examples of such situations where the mean field limit can nevertheless be justified rigorously --- see for instance \cite{MischlerMouhotWenn}, and 
\cite{Sznitman} in the case of random dynamics.

There are also situations where there is no clear notion of empirical measure --- think for instance to the $N$-body problem in quantum mechanics: in that 
case, it is impossible to exactly localize any one of the $N$ particles in phase space, according to the Heisenberg uncertainty principle. 

In the present section, we present another approach to the mean field limit of $N$-particle systems, that is, in some sense, more systematic than the method based on empirical measure and that can be applied to a greater variety of situations (including quantum models, as we shall see later). 

\subsection{$N$-particle distributions}
The state at time $t$ of a system of $N$ identical particles located at the positions $z_1(t),\ldots,z_N(t))$ in the single-particle phase space $\bR^d$ was 
described  in the previous section by means of the empirical measure 
$$
\mu_{Z_N(t)}:=\frac1N\sum_{i=1}^N\de_{z_i(t)}\,,
$$
where $Z_N(t):=(z_1(t),\ldots,z_N(t))$. The empirical is a probability measure in the single-particle phase space $\bR^d$ as mentioned above. This measure
is parametrized by the element $Z_N(t)$ of the $N$-particle phase space $(\bR^d)^N$.

Another way of describing the state of the same system of $N$ particles at time $t$ is to use its $N$-particle \index{$N$-particle distribution function} 
distribution function, that is
$$
F_N(t,z_1,\ldots,z_N)\,.
$$
More generally, one could think of the $N$-particle distribution as being a probability measure on the $N$-body phase space $(\bR^d)^N$
$$
F_N(t,dz_1\ldots dz_N)\,.
$$
The meaning of this $N$-particle distribution is as follows. Let $A_j\subset\bR^d$ be Borel measurable sets for $j=1,\ldots,N$; then, the joint probability 
at time $t$ to have particle $1$ in $A_1$, particle $2$ in $A_2$\dots and particle $N$ in $A_N$ is
$$
\int_{A_1}\int_{A_2}\ldots\int_{A_N}F(t,dz_1dz_2\ldots dz_N)\,.
$$

Now, we are interested in situations where all the particles in the $N$-particle system considered are identical. (For example, all electrons in the universe
are identical; ions of any given species in a plasma are identical too.) Therefore, for any permutation $\si\in\fS_N$, the joint probability of having particle 
$1$ in $A_1$, particle 2 in $A_2$\ldots and particle $N$ in $A_N$ is equal to the joint probability of having particle $1$ in $A_{\si^{-1}(1)}$, particle 2 in 
$A_{\si^{-1}(2)}$\ldots and particle $N$ in $A_{\si^{-1}(N)}$. This is indeed obvious since it is impossible to distinguish particle $1$ from particle $\si^{-1}(1)$,
particle $2$ from particle $\si^{-1}(2)$\ldots and particle $N$ from particle $\si^{-1}(N)$. Thus\index{Indistinguishable particles}
$$
\ba
\int_{A_1}\int_{A_2}\ldots\int_{A_N}F(t,dz_1dz_2\ldots dz_N)&
\\
=
\int_{A_{\si^{-1}(1)}}\int_{A_{\si^{-1}(2)}}\ldots\int_{A_{\si^{-1}(N)}}F(t,dz_1dz_2\ldots dz_N)&\,.
\ea
$$
Equivalently
$$
\ba
\int_{(\bR^d)^N}\indc_{A_1\times A_2\times\ldots\times A_N}(z_1,z_2,\ldots,z_N)F(t,dz_1dz_2\ldots dz_N)&
\\
=
\int_{(\bR^d)^N}\indc_{A_1\times A_2\times\ldots\times A_N}(z_{\si(1)},z_{\si(2)},\ldots,z_{\si(N)})F(t,dz_1dz_2\ldots dz_N)&\,.
\ea
$$
 
For $\si\in\fS_N$, define 
$$
S_\si:\,(z_1,\ldots,z_N)\mapsto(z_{\si(1)},\ldots,z_{\si(N)})\,.
$$
The equality above is recast as
$$
\ba
\int_{(\bR^d)^N}\indc_{A_1\times A_2\times\ldots\times A_N}(z_1,z_2,\ldots,z_N)F(t,dz_1dz_2\ldots dz_N)
\\
=
\int_{(\bR^d)^N}\indc_{A_1\times A_2\times\ldots\times A_N}\circ S_\si(z_1,\ldots,z_N)F(t,dz_1dz_2\ldots dz_N)
\ea
$$
for all $A_1,A_2,\ldots,A_N$ Borel subsets of $\bR^d$. This is equivalent to the equality
$$
S_\si\#F(t,\cdot)=F(t,\cdot)
$$
for all $\si\in\fS_N$. 

Obviously, when $F$ is a probability density instead of a probability measure, the condition 
$$
S_\si\#F(t,\cdot)\scrL^{dN}=F(t,\cdot)\scrL^{dN}
$$
for all $\si\in\fS_N$ is equivalent to the condition
$$
F(t,z_{\si(1)},\ldots,z_{\si(N)})=F(t,S_\si(z_1,\ldots, z_N))=F(t,z_1,\ldots,z_N)
$$
for all $\si\in\fS_N$ and for all $z_1,\ldots,z_N\in\bR^d$. In other words, the function
$$
(z_1,\ldots,z_N)\mapsto F(t,z_1,\ldots,z_N)
$$
is symmetric.

\subsection{Marginal distributions of symmetric $N$-particle distributions}

There is however one serious difficulty in considering $N$-particle distributions in the context of the mean field limit. Indeed, this limit assumes that
$N\to\infty$, so that one would have to deal with ``functions of infinitely many variables'' in this limit, which does not make much sense at first 
sight\footnote{This last statement is not completely correct, as P.-L. Lions recently proposed a well defined mathematical object that would play the 
role of a ``symmetric function of infinitely many variables that is slowly varying in each variable'': see \cite{PLLionsCdF07} and section \ref{SS-Lions}
below.}. 

A traditional way of circumventing this difficulty is by considering the string of marginal distributions of the $N$-particle distribution. Before giving precise
definitions, let us explain the idea in simple geometrical terms.

Consider a sphere centered at the origin in the $3$ dimensional Euclidean space. The only missing information in order to completely define this 
sphere is its diameter. In other words, the sphere is completely determined as soon as one knows its orthogonal projection on any axis passing
through the origin. However, if one does not know a priori that the object is a sphere, its orthogonal projections on each axis passing through the 
origin is not enough in order to reconstruct completely the object, since they will not distinguish between a ball centered at the origin and its
boundary that is a sphere of equal radius. 

The situation that we consider here is slightly more complicated, since the group of symmetries is not the orthogonal group, but the group generated
by the reflections exchanging two coordinate axis in the $N$ dimensional Euclidean space. Knowing the orthogonal projection of a set that is invariant 
under the action of this group on the first coordinate axis is again not sufficient as it will not distinguish between a sphere of radius $r$ centered at the
origin and the (hyper)cube of side $2r$ centered at the origin with edges parallel to the coordinate axis. Knowing the orthogonal projection on any one
of the planes defined by two coordinate axis removes this ambiguity.

Considering marginals of an $N$-particle distribution is the analogous operation on probability measures. Denote by $\cP_{sym}((\bR^d)^N)$ the set 
of symmetric probability measures on the $N$-particle phase space, i.e.
$$
\cP_{sym}((\bR^d)^N):=\{P\in\cP((\bR^d)^N)\,|\,S_\si\#P=P\hbox{ for all }\si\in\fS_N\}\,,
$$
where we recall that $S_\si$ is the transformation on $(\bR^d)^N$ defined by 
$$
S_\si(z_1,\ldots,z_N)=(z_{\si(1)},\ldots,z_{\si(N)})
$$
for all $\si\in\fS_N$ and all $z_1,\ldots,z_N\in\bR^d$.

\begin{Def}
For each $N\in\bN^*$, each $P_N\in\cP_{sym}((\bR^d)^N)$ and each $k\in\{1,\ldots,N\}$, the $k$-particle marginal of $P_N$ is the element of 
$\cP_{sym}((\bR^d)^k)$ defined by the formula\index{Marginals of a $N$-particle distribution}
$$
\int_{(\bR^d)^k}\phi(z_1,\ldots,z_k)P_{N:k}(dz_1\ldots dz_k)=\int_{(\bR^d)^N}\phi(z_1,\ldots,z_k)P_N(dz_1\ldots dz_N)
$$
for each test function $\phi\in C_b((\bR^d)^k)$. We shall systematically use the convention
$$
P_{N:k}=0\quad\hbox{ whenever }j>N\,.
$$
\end{Def}

\smallskip
If
$$
P_N(dz_1\ldots dz_N)=F_N(z_1,\ldots,z_N)dz_1\ldots dz_N
$$
where $F_N$ is a symmetric probability density on $(\bR^d)^N$, then, for each $k=1,\ldots,N$, one has
$$
P_{N:k}(dz_1\ldots dz_k)=F_{N:k}(z_1,\ldots,z_k)dz_1\ldots dz_k
$$
with
$$
F_{N:k}(z_1,\ldots,z_k)=\int_{(\bR^d)^{N-k}}F_N(z_1,\ldots,z_N)dz_{k+1}\ldots dz_N\,.
$$
Obviously, $F_{N:k}$ is also a symmetric probability density on $(\bR^d)^k$.

The following elementary exercise confirms the analogy between the orthogonal projections of a subset of the Euclidean space on the subspaces generated
by the coordinate axis and the marginal distributions associated to a symmetric probability on the $N$-particle phase space.

\smallskip
\noindent
\textbf{Exercise:} Consider for each $N\in\bN^*$ and each $k=1,\ldots,N$ the orthogonal projection
$$
\bP_N^k:\,(\bR^d)^N\ni(z_1,\ldots,z_N)\mapsto(z_1,\ldots,z_k)\in(\bR^d)^k\,.
$$
Check that, for each $P_N\in\cP_{sym}((\bR^d)^N)$, one has 
$$
P_{N:k}=\bP_N^k\#P_N\,,
$$
and that
$$
(P_{N:k})_{:j}=P_{N:j}\quad\hbox{ for all }j,k\hbox{ such that }1\le j\le k\le N\,.
$$

\smallskip
An important example of symmetric $N$-particle distributions is the case of \textit{factorized distributions}. Given a probability density $f$ on $\bR^d$,
we consider for each $N\in\bN^*$ the $N$-particle probability density $F_N$ defined by the formula
$$
F_N(z_1,\ldots,z_N):=\prod_{j=1}^Nf(z_j)\,.
$$
This $N$-particle probability density is denoted as follows:
$$
F_N=f^{\otimes N}\,.
$$
Obviously $F_N=f^{\otimes N}$ is a symmetric $N$-particle probability distribution, and its marginals are also factorized distributions, since
$$
F_N=f^{\otimes N}\Rightarrow F_{N:k}=f^{\otimes k}
$$
for all $k=1,\ldots,N$.

There is a very nice characterization of factorized distributions in terms of entropy\index{Entropy}. We shall not use it in the sequel. Nevertheless, it is important 
to know it, and we leave it as an exercise.

\smallskip
\noindent
\textbf{Exercise:} For each probability density $f$ on $\bR^d$ and each $N\in\bN^*$, define
$$
E_N(f)=\{F_N\hbox{ symmetric probability density on }(\bR^d)^N\hbox{ s.t. }F_{N:1}=f\}\,.
$$
We want to prove that $F_N=f^{\otimes N}$ realizes
$$
\inf_{F_N\in E_N(f)}\int_{(\bR^d)^N}F_N\ln F_N(z_1,\ldots,z_N)dz_1\ldots dz_N\,.
$$
1) Let $a=(a_1,\ldots,a_d)\in(\bR_+^*)^d$ satisfy $a_1+\ldots+a_d=1$, and consider
$$
M(a):=\{A=A^T\in M_d(\bR_+^*)\hbox{ s.t. }(1,\ldots,1)\cdot A=a\}\,.
$$
Find the critical points of the function
$$
H:\,M(a)\ni A\mapsto\sum_{i,j=1}^dA_{ij}\ln A_{ij}\in\bR\,.
$$
2) Prove that, for each $x,y>0$
$$
\phi(x,y):=x\ln\left(\frac{x}{y}\right)-x+y\ge 0\,,
$$
with equality if and only if $x=y$.

\noindent
3) Express in terms of $\phi(A_{ij},a_ia_j)$ the quantity
$$
\sum_{i,j=1}^d(A_{ij}\ln A_{ij}-a_ia_j\ln(a_ia_j))\,.
$$
4) Find 
$$
\inf_{A\in M(a)}H(A)\,.
$$
5) Using the intuition provided by questions 1-4, solve the minimization problem
$$
\inf_{F_N\in E_N(f)}\int_{(\bR^d)^N}F_N\ln F_N(z_1,\ldots,z_N)dz_1\ldots dz_N\,.
$$

\smallskip
References for this and the previous section are chapter 3 in \cite{BouGolPul} and chapter 3 in \cite{CIP}.

\subsection{The $N$-particle Liouville equation}

We have explained above how the state of a system of $N$ identical particles is described by a symmetric probability measure on the $N$-particle phase 
space. Our next task is to define the evolution of such a probability measure, knowing that the positions of the particles in phase space are governed by
the system of $N$-particle ODEs
$$
\left\{
\ba
{}&\dot{z}_i(t)=\frac1N\sum_{j=1}^NK(z_i(t),z_j(t))\,,\quad i=1,\ldots,N\,,
\\
&z_i(0)=z_i^{in}\,.
\ea
\right.
$$
As explained in Theorem \ref{T-ExUnNbody}, whenever $K\in C^1(\bR^d\times\bR^d,\bR^d)$ satisfies assumptions (HK1-HK2), the system of ODEs above 
generates a flow on the $N$-particle phase space $(\bR^d)^N$ denoted by $T_t$ and defined by the formula
$$
T_t(z_1^{in},\ldots,z_N^{in}):=(z_1(t),\ldots,z_N(t))
$$
for all $t\in\bR$.

Given a $N$-particle symmetric probability measure $F_N^{in}\in\cP_{sym}((\bR^d)^N)$, we set
$$
F_N(t):=T_t\#F_N^{in}\,,\qquad t\in\bR\,.
$$
This formula defines $F_N(t)$ as the unique weak solution in $C(\bR;w-\cP((\bR^d)^N))$ of the Cauchy problem for the $N$-particle Liouville equation
$$
\left\{
\ba
{}&\d_tF_N+\frac1N\sum_{i,j=1}^N\Div_{z_i}(F_NK(z_i,z_j))=0\,,\quad z_1,\ldots,z_N\in\bR^d\,,\,\,t\in\bR\,,
\\
&F_N\rstr_{t=0}=F_N^{in}\,.
\ea
\right.
$$

\smallskip
\noindent
\textbf{Exercise:} Check this, by using the method of characteristics as in the exercise following the statement of Theorem \ref{T-ExUnNbody}.
(For the solution, see chapter 1 in \cite{BouGolPul}).

\smallskip
As explained above, it is important that the probability measure in the $N$-particle phase space describing the state of a system of identical particles should 
be symmetric. Whether this symmetry property is propagated by the flow of the $N$-particle Liouville equation is therefore a very natural question. The answer 
to that question is given by the following proposition.

\begin{Prop} 
Assume that the interaction kernel $K\in C^1(\bR^d\times\bR^d,\bR^d)$ satisfies the assumptions (HK1)-(HK2) so that the $N$-particle ODE system defines 
a unique flow $T_t$ on $(\bR^d)^N$ as proved in Theorem \ref{T-ExUnNbody}.  For each $\si\in\fS_N$ we denote as above by $S_\si$ the transformation on 
$(\bR^d)^N$ defined by 
$$
S_\si(z_1,\ldots,z_N)=(z_{\si(1)},\ldots,z_{\si(N)})
$$
for all $z_1,\ldots,z_N\in\bR^d$.

\smallskip
\noindent
1) For all $\si\in\fS_N$ and all $t\in\bR$, one has 
$$
T_tS_\si=S_\si T_t\,,
$$
viz.
$$
\ba
t\mapsto(z_{\si(1)}(t),\ldots,z_{\si(N)}(t))&\hbox{ is the solution of the $N$-particle ODE system}
\\
&\hbox{ with initial condition }(z_{\si(1)}^{in},\ldots,z_{\si(N)}^{in})\,;
\ea
$$
2) For each $F_N^{in}\in\cP((\bR^d)^N)$ and all $t\in\bR$, the probability measure 
$$
F_N(t):=T_t\#F_N^{in}
$$ 
is symmetric if $F_N^{in}$ is symmetric, i.e.
$$
S_\si\#F_N(t)=F_N(t)\quad\hbox{ for all }t\in\bR\,,
$$
if 
$$
S_\si\#F_N^{in}=F_N^{in}\,.
$$
\end{Prop}

\begin{proof}
An elementary computation shows that 
$$
t\mapsto(z_{\si(1)}(t),\ldots,z_{\si(N)}(t))
$$
satisfies the same $N$-particle ODE system as
$$
t\mapsto(z_1(t),\ldots,z_N(t))\,.
$$
By uniqueness of the solution of the Cauchy problem for this ODE system (Theorem \ref{T-ExUnNbody}), this is therefore the unique solution of that problem 
with initial data $(z_{\si(1)}^{in},\ldots,z_{\si(N)}^{in})$. In other words,
$$
T_tS_\si(z_1^{in},\ldots,z_N^{in})=(z_{\si(1)}(t),\ldots,z_{\si(N)}(t))=S_\si T_t(z_1^{in},\ldots,z_N^{in})
$$
for all $(z_1^{in},\ldots,z_N^{in})\in(\bR^d)^N$ and all $t\in\bR$, which proves statement 1).

As for statement 2), observe that
$$
\ba
S_\si\#F_N(t)&=S_\si\#(T_t\#F_N^{in})=(S_\si T_t)\#F_N^{in}
\\
&=(T_tS_\si)\#F_N^{in}=T_t\#(S_\si\#F_N^{in})=T_t\#F_N^{in}=F_N(t)
\ea
$$
for all $t\in\bR$ and all $\si\in\fS_N$, which is precisely the desired relation.
\end{proof}

\smallskip
Finally, we discuss the growth of $F_N(t,z_1,\ldots,z_N)$ as $|z_1|+\ldots+|z_N|\to\infty$. 

\begin{Lem}\lb{L-PropMom1}
Under the assumptions (HK1)-(HK2) on the interaction kernel $K$, one has
$$
\|T_t(z_1^{in},\ldots,z_N^{in})\|_1\le e^{2L|t|}\|(z_1^{in},\ldots,z_N^{in})\|_1
$$
with the notation
$$
\|(z_1,\ldots,z_N)\|_1=|z_1|+\ldots+|z_N|\,.
$$
In particular, if $F_N^{in}\in\cP_1((\bR^d)^N)$, then $T_t\#F_N^{in}\in\cP_1((\bR^d)^N)$, and one has
$$
\ba
\int_{(\bR^d)^N}&\|(z_1,\ldots,z_N)\|_1T_t\#F_N^{in}(dz_1\ldots dz_N)
\\
&\le e^{2L|t|}\int_{(\bR^d)^N}\|(z_1,\ldots,z_N)\|_1F_N^{in}(dz_1\ldots dz_N)
\ea
$$
for all $t\in\bR$.
\end{Lem}

\begin{proof}
Observe that
$$
\ba
\frac{d}{dt}\|(z_1(t),\ldots,z_N(t))\|_1&=\frac1N\sum_{i,j=1}^NK(z_i(t),z_j(t))\cdot\frac{z_i(t)}{|z_i(t)|}
\\
&\le\frac1N\sum_{i,j=1}^N|K(z_i(t),z_j(t))|
\\
&\le\frac1N\sum_{i,j=1}^NL(|z_i(t)|+|z_j(t)|)
\\
&=\sum_{i=1}^NL\left(|z_i(t)|+\frac1N\|(z_1(t),\ldots,z_N(t))\|_1\right)
\\
&=2L\|(z_1(t),\ldots,z_N(t))\|_1\,,
\ea
$$
and conclude by the Gronwall inequality.

Moreover, if $F_N^{in}\in\cP_1((\bR^d)^N)$, then
$$
\ba
\int_{(\bR^d)^N}&\|(z_1,\ldots,z_N)\|_1T_t\#F_N^{in}(dz_1\ldots dz_N)
\\
&=
\int_{(\bR^d)^N}\|T_t(z_1,\ldots,z_N)\|_1F_N^{in}(dz_1\ldots dz_N)
\\
&\le
e^{2L|t|}\int_{(\bR^d)^N}\|(z_1,\ldots,z_N)\|_1F_N^{in}(dz_1\ldots dz_N)<\infty\,,
\ea
$$
so that $T_t\#F_N^{in}\in\cP_1((\bR^d)^N)$.
\end{proof}

\subsection{The BBGKY hierarchy}

The curious designation for this procedure finds its origin in the names of N.N. Bogoliubov, M. Born, H.S. Green, J.G. Kirkwood  and J. Yvon, who 
introduced it in various contexts.

Before presenting the BBGKY hierarchy in detail, we first discuss the main reason for considering it in the first place.

As explained above, the $N$-particle distribution $F_N$ is defined on the $N$-particle phase space $(\bR^d)^N$, whose dimension increases as 
$N\to\infty$. In other words, the number of variables in $F_N$ goes to infinity with $N$, so that the exact nature of the limiting object associated with 
$F_N$ is not entirely obvious. Therefore, we seek to describe the behavior of $F_N$ in the large $N$ limit by considering instead its first marginal 
$F_{N:1}$ in that limit. By doing so, we avoid the problem of having the number of variables going to infinity with $N$, at the expense of losing some
amount of information by reducing $F_N$ to $F_{N:1}$.

To fulfill this program, it would be desirable to know the evolution of $F_{N:1}$, typically by means of a PDE in the $1$-particle phase space to be 
satisfied by $F_{N:1}$. 

Unfortunately, because of the interaction modeled by the kernel $K$, it is impossible to find a closed equation for $F_{N:1}$. Instead, the equation 
for $F_{N:1}$ deduced from the $N$-particle Liouville equation involves $F_{N:2}$. Again it is impossible to find a closed equation governing the
evolution of $F_{N:2}$ when $N>2$, as the pairwise interaction between particles will involve $F_{N:3}$ in the equation for $F_{N:2}$ deduced 
from the $N$-particle Liouville equation. By the same token, one can check that it is impossible to derive from the $N$-particle Liouville equation
a closed system of equations for \textit{finitely many} --- i.e. $m<N$ --- of the marginals $F_{N:k}$.

Even though the system of PDEs governing the evolution of the marginal distributions $F_{N:k}$ cannot be put in closed form, these equations are
nevertheless interesting, and we explain how to derive them below.

We start from the $N$-particle Liouville equation satisfied by $F_N$:
$$
\d_tF_N+\frac1N\sum_{i,j=1}^N\Div_{z_i}(K(z_i,z_j)F_N)=0\,.
$$
That $t\mapsto F_N(t)$ is a weak solution of this equation means that, for each $\Phi\equiv\Phi(z_1,\ldots,z_N)$ in $C^1_b((\bR^d)^N)$, one has
$$
\ba
\d_t\int_{(\bR^d)^N}\Phi_N(z_1,\ldots,z_N)F_N(t,dz_1\ldots dz_N)
\\
=
\frac1N\sum_{i,j=1}^N\int_{(\bR^d)^N}K(z_i,z_j)\grad_{z_i}\Phi(z_1,\ldots,z_N)F_N(t,dz_1\ldots dz_N)
\ea
$$
in the sense of distributions on $\bR$.

We first seek an equation for the first marginal $F_{N:1}$. Let $\phi\in C^1_b(\bR^d)$; writing the weak formulation of the $N$-particle Liouville 
equation for the test function $\Phi(z_1,\ldots,z_N)=\phi(z_1)$, we see that
$$
\ba
\d_t\int_{(\bR^d)^N}\phi(z_1)&F_N(t,dz_1\ldots dz_N)
\\
&=\frac1N\sum_{j=2}^N\int_{(\bR^d)^N}K(z_1,z_j)\cdot\grad\phi(z_1)F_N(t,dz_1\ldots dz_N)\,.
\ea
$$
(Notice that the term associated to $j=1$ in the summation on the right hand side of the equality above vanishes identically since $K$ vanishes on
the diagonal by assumption (HK1).)

In the term
$$
\int_{(\bR^d)^N}K(z_1,z_j)\cdot\grad\phi(z_1)F_N(t,dz_1\ldots dz_N)
$$
we exchange the variables $z_2$ and $z_j$. Denoting by $\si_{2j}\in\fS_N$ the transposition exchanging $2$ and $j$, one has
$$
\ba
\int_{(\bR^d)^N}K(z_1,z_j)\cdot\grad\phi(z_1)F_N(t,dz_1\ldots dz_N)&
\\
=
\int_{(\bR^d)^N}K(z_1,z_2)\cdot\grad\phi(z_1)S_{\si_{2j}}F_N(t,dz_1\ldots dz_N)&
\\
=
\int_{(\bR^d)^N}K(z_1,z_2)\cdot\grad\phi(z_1)F_N(t,dz_1\ldots dz_N)&\,.
\ea
$$
by symmetry of $F_N(t)$. 

Therefore 
$$
\ba
\d_t\int_{(\bR^d)^N}\phi(z_1)&F_N(t,dz_1\ldots dz_N)
\\
&=\frac{N-1}N\int_{(\bR^d)^N}K(z_1,z_2)\cdot\grad\phi(z_1)F_N(dz_1\ldots dz_N)\,.
\ea
$$
The integral on the left hand side of the equality above is recast as follows:
$$
\int_{(\bR^d)^N}\phi(z_1)F_N(t,dz_1\ldots dz_N)=\int_{\bR^d}\phi(z_1)F_{N:1}(t,dz_1)\,.
$$
By the same token, the integral on the right hand side becomes
$$
\ba
\int_{(\bR^d)^N}K(z_1,z_2)\cdot\grad\phi(z_1)F_N(t,dz_1\ldots dz_N)&
\\
=\int_{(\bR^d)^2}K(z_1,z_2)\cdot\grad\phi(z_1)F_{N:2}(t,dz_1dz_2)&\,.
\ea
$$
Thus, for each $\phi\in C^1_b(\bR^d)$, one has
$$
\d_t\int_{\bR^d}\phi(z_1)F_{N:1}(t,dz_1)=\frac{N-1}{N}\int_{(\bR^d)^2}K(z_1,z_2)\cdot\grad\phi(z_1)F_{N:2}(t,dz_1dz_2)\,,
$$
which is the weak formulation of
$$
\d_tF_{N:1}+\frac{N-1}{N}\Div_{z_1}\int_{\bR^d}K(z_1,z_2)F_{N:2}(\cdot,dz_2)=0\,.
$$

Equivalently
$$
\d_tF_{N:1}+\frac{N-1}{N}\Div_{z_1}[K(z_1,z_2)F_{N:2}]_{:1}=0\,,
$$
where $K(z_1,z_2)F_{N:2}$ designates the Radon measure defined on $(\bR^d)^2$ as the linear functional on bounded continuous functions given 
by the formula
$$
\la K(z_1,z_2)F_{N:2},\psi\ra:=\int_{(\bR^d)^2}\psi(z_1,z_2)K(z_1,z_2)F_{N:2}(dz_1dz_2)\,,
$$
while $[K(z_1,z_2)F_{N:2}]_{:1}$ designates the Radon measure defined on $\bR^d$ by the formula
$$
\la [K(z_1,z_2)F_{N:2}]_{:1},\phi\ra:=\int_{(\bR^d)^2}\phi(z_1)K(z_1,z_2)F_{N:2}(dz_1dz_2)\,.
$$
(Equivalently, 
$$
[K(z_1,z_2)F_{N:2}]_{:1}:=\bP_2^1\#[K(z_1,z_2)F_{N:2}]
$$
where we recall that $\bP_2^1$ is the orthogonal projection defined by the formula $\bP_2^1(z_1,z_2):=z_1$.)

In any case, as anticipated, the equation for the first marginal distribution $F_{N:1}$ involves the second marginal distribution $F_{N:2}$. 

We next proceed to derive the equations satisfied by the sequence of marginal distributions $F_{N:j}$ for $j=2,\ldots,N$; this derivation will proceed
as in the case $j=1$, except for one additional term.

For $1<j<N$, we write the weak formulation of the $N$-particle Liouville equation with test function $\Phi(z_1,\ldots,z_N)=\phi(z_1,\ldots,z_j)$, where
$\phi\in C_b((\bR^d)^j)$. Thus
$$
\ba
\d_t\int_{(\bR^d)^N}\phi(z_1,\ldots,z_j)F_N(t,dz_1\ldots dz_N)&
\\
=\frac1N\sum_{l=1}^j\sum_{k=j+1}^N\int_{(\bR^d)^N}K(z_l,z_k)\cdot\grad_{z_l}\phi(z_1,\ldots,z_j)F_N(t,dz_1\ldots dz_N)&
\\
+\frac1N\sum_{l=1}^j\sum_{k=1}^j\int_{(\bR^d)^N}K(z_l,z_k)\cdot\grad_{z_l}\phi(z_1,\ldots,z_j)F_N(t,dz_1\ldots dz_N)&\,.
\ea
$$
Notice that the range of the index $l$ is limited to $\{1,\ldots,j\}$ since the test function $\phi$ does not depend on the variables $z_{j+1},\ldots,z_N$.
The range of the index $k$ remains $\{1,\ldots,N\}$, and we have decomposed it into $\{1,\ldots,j\}$ and $\{j+1,\ldots,N\}$. This decomposition is quite
natural, as the sum involving $k,l\in\{1,\ldots,j\}$ accounts for the pairwise interactions between the $j$ particles whose state is described by $F_{N:j}$,
while the sum  involving $l\in\{1,\ldots,j\}$ and $k\in\{j+1,\ldots,N\}$ accounts for the pairwise interactions between each one of the $j$ particles whose 
state is described by $F_{N:j}$ and the $N-j$ other particles in the system.

As in the case $j=1$,
$$
\int_{(\bR^d)^N}\phi(z_1,\ldots,z_j)F_N(t,dz_1\ldots dz_N)=\int_{(\bR^d)^j}\phi(z_1,\ldots,z_j)F_{N:j}(t,dz_1\ldots dz_j)\,,
$$
and by the same token, if $1\le k,l\le j$
$$
\ba
\int_{(\bR^d)^N}K(z_l,z_k)\cdot\grad_{z_l}\phi(z_1,\ldots,z_j)F_N(t,dz_1\ldots dz_N)&
\\
=\int_{(\bR^d)^j}K(z_l,z_k)\cdot\grad_{z_l}\phi(z_1,\ldots,z_j)F_{N:j}(t,dz_1\ldots dz_j)&\,.
\ea
$$

If $1\le l\le j<k\le N$, denote by $\si_{j+1,k}\in\fS_N$ the transposition exchanging $j+1$ and $k$. Then
$$
\ba
\int_{(\bR^d)^N}K(z_l,z_k)\cdot\grad_{z_l}\phi(z_1,\ldots,z_j)F_N(t,dz_1\ldots dz_N)
\\
=\int_{(\bR^d)^N}K(z_l,z_{j+1})\cdot\grad_{z_l}\phi(z_1,\ldots,z_j)S_{\si_{j+1,k}}F_N(t,dz_1\ldots dz_N)
\\
=\int_{(\bR^d)^N}K(z_l,z_{j+1})\cdot\grad_{z_l}\phi(z_1,\ldots,z_j)F_N(t,dz_1\ldots dz_N)
\ea
$$
by symmetry of $F_N(t)$. Then
$$
\ba
\int_{(\bR^d)^N}K(z_l,z_{j+1})\cdot\grad_{z_l}\phi(z_1,\ldots,z_j)F_N(t,dz_1\ldots dz_N)&
\\
=\int_{(\bR^d)^{j+1}}K(z_l,z_{j+1})\cdot\grad_{z_l}\phi(z_1,\ldots,z_j)F_{N:j+1}(t,dz_1\ldots dz_{j+1})&\,.
\ea
$$

Finally, we obtain the equality
$$
\ba
\frac{d}{dt}\int_{(\bR^d)^j}\phi(z_1,\ldots,z_j)F_{N:j}(t,dz_1\ldots dz_j)
\\
=
\frac1N\sum_{k,l=1}^j\int_{(\bR^d)^j}K(z_l,z_k)\cdot\grad_{z_l}\phi(z_1,\ldots,z_j)F_{N:j}(t,dz_1\ldots dz_j)
\\
+\frac{N-j}N\sum_{l=1}^j\int_{(\bR^d)^{j+1}}K(z_l,z_{j+1})\cdot\grad_{z_l}\phi(z_1,\ldots,z_j)F_{N:j+1}(t,dz_1\ldots dz_{j+1})
\ea
$$
to be verified for each $\phi\in C^1_b((\bR^d)^j)$. 

This is the weak formulation of the equation
$$
\ba
\d_tF_{N:j}+\frac{N-j}N\sum_{l=1}^j\Div_{z_l}\int_{\bR^d}K(z_l,z_{j+1})F_{N:j+1}(\cdot,dz_{j+1})&
\\
+\frac1N\sum_{k,l=1}^j\Div_{z_l}(K(z_l,z_k)F_{N:j})=0&\,.
\ea
$$
Equivalently
$$
\ba
\d_tF_{N:j}+\frac{N-j}N\sum_{l=1}^j\Div_{z_l}[K(z_l,z_{j+1})F_{N:j+1}]_{:j}&
\\
+\frac1N\sum_{k,l=1}^j\Div_{z_l}(K(z_l,z_k)F_{N:j})=0&\,,
\ea
$$
where $K(z_l,z_{j+1})F_{N:j+1}$ designates the Radon measure defined on $(\bR^d)^{j+1}$ as the linear functional on bounded continuous functions 
given by the formula
$$
\ba
\la K(z_l,z_{j+1})F_{N:j+1},\psi\ra&
\\
:=\int_{(\bR^d)^{j+1}}\psi(z_1,\ldots,z_{j+1})K(z_l,z_{j+1})F_{N:j+1}(dz_1\ldots dz_{j+1})&\,,
\ea
$$
while $[K(z_l,z_{j+1})F_{N:j+1}]_{:j}$ designates the Radon measure defined on $(\bR^d)^j$ by the formula
$$
\ba
\la [K(z_l,z_{j+1})F_{N:j+1}]_{:j},\phi\ra&
\\
:=\int_{(\bR^d)^j}\phi(z_1,\ldots,z_j)K(z_l,z_{j+1})F_{N:j+1}(dz_1\ldots dz_{j+1})&\,.
\ea
$$
Equivalently, 
$$
[K(z_l,z_{j+1})F_{N:j+1}]_{:j}:=\bP_{j+1}^j\#[K(z_l,z_{j+1})F_{N:j+1}]\,,
$$
where we recall that $\bP_{j+1}^j$ is the orthogonal projection defined by the formula 
$$
\bP_{j+1}^j(z_1,\ldots,z_{j+1}):=(z_1,\ldots,z_j)\,.
$$

The equation obtained in the case $j=N$ is nothing but the $N$-particle Liouville equation itself since $F_{N:N}=F_N$ and $F_{N:j}=0$ for all $j>N$: thus
$$
\d_tF_{N:N}+\frac1N\sum_{k,l=1}^N\Div_{z_l}(K(z_l,z_k)F_{N:N})=0\,.
$$

We summarize the above lengthy computations in the following theorem, where the notation $\cP_{1,sym}((\bR^d)^N)$ designates
$$
\cP_{1,sym}((\bR^d)^N)=\cP_1((\bR^d)^N)\cap \cP_{sym}((\bR^d)^N)\,.
$$

\begin{Thm}
Assume that the interaction kernel $K\in C^1(\bR^d\times\bR^d,\bR^d)$ satisfies (HK1-HK2).
Let $F_N^{in}\in\cP_{1,sym}((\bR^d)^N)$, and let $F_N(t)=T_t\#F_N^{in}$ for all $t\in\bR$, where $T_t$ is the flow defined on $(\bR^d)^N$ by the 
$N$-particle ODE system as in Theorem \ref{T-ExUnNbody}. The sequence of marginal distributions $F_{N:j}$ of $F_N$ with $j=1,\ldots,N$ is a
weak solution of the string of equations\index{BBGKY hierarchy}
$$
\left\{
\ba
{}&\d_tF_{N:1}+\frac{N-1}{N}\Div_{z_1}[K(z_1,z_2)F_{N:2}]_{:1}=0\,,
\\
&\d_tF_{N:j}+\frac{N-j}N\sum_{l=1}^j\Div_{z_l}[K(z_l,z_{j+1})F_{N:j+1}]_{:j}
\\
&\qquad\qquad\qquad+\frac1N\sum_{k,l=1}^j\Div_{z_l}(K(z_l,z_k)F_{N:j})=0\,,\qquad j=2,\ldots,N-1\,,
\\
&\d_tF_{N:N}+\frac1N\sum_{k,l=1}^N\Div_{z_l}(K(z_l,z_k)F_{N:N})=0\,,
\ea
\right.
$$
and satisfies the initial conditions
$$
F_{N:j}\rstr_{t=0}=F^{in}_{N:j}\,,\quad j=1,\ldots,N\,.
$$
\end{Thm}

This string of equations bears the name of \textit{BBGKY hierarchy} for the $N$-particle system whose dynamics is defined by the ODE system
$$
\dot{z}_i(t)=\frac1N\sum_{j=1}^NK(z_i(t),z_j(t))\,,\quad 1\le i\le N\,.
$$

In spite of all the (somewhat) technical computations involved in the derivation of the BBGKY hierarchy, the careful reader will notice that

\smallskip
\noindent
a) the BBGKY hierarchy is a consequence of the $N$-particle Liouville equation, but

\noindent
b) it contains the $N$-particle Liouville equation (which is the last equation in the hierarchy).

\smallskip
This observation might cast some doubts as to the interest of considering the BBGKY hierarchy instead of the Liouville equation itself, since both 
contain exactly the same amount of information.

\smallskip
References for this section are chapters 3-4 in \cite{CIP}, \cite{Spohn80,Spohn91}, chapter 3 in \cite{BouGolPul} and \cite{BGMMinneapolis}.

\subsection{The mean field hierarchy and factorized distributions}

In the present section, our discussion of the mean field limit becomes purely formal.

Our aim is to pass to the limit in each equation in the BBGKY hierarchy as $N\to\infty$, keeping $j\ge 1$ fixed. Assume that $F_{N:j}\to F_j$ as $N\to\infty$ 
(in some sense to be made precise) for all $j\ge 1$. Then, in the limit as $N\to\infty$
$$
\frac{N-j}N\int_{\bR^d}K(z_l,z_{j+1})F_{N:j+1}(dz_{j+1})\to\int_{\bR^d}K(z_l,z_{j+1})F_{j+1}(dz_{j+1})\,,
$$
while
$$
\frac1NK(z_l,z_k)F_{N:j}\to 0\,,
$$
so that
$$
\d_tF_j+\sum_{l=1}^j\Div_{z_l}\int_{\bR^d}K(z_l,z_{j+1})F_{j+1}(dz_{j+1})=0\,,\qquad j\ge 1\,.
$$

This hierarchy of equations is henceforth referred to as \textit{the mean field hierarchy}, or \textit{the Vlasov hierarchy}\index{Mean-field hierarchy}. The 
similarities between this mean field hierarchy and the BBGKY hierarchy are striking. Yet there is an important difference: the mean field hierarchy is an 
\textit{infinite} hierarchy of equations --- unlike the BBGKY hierarchy, which contains only $N$ equations, where $N$ is the total number of particles. The
physical meaning of this infinite hierarchy of equations will be explained in section \ref{SS-InfHierar}.

However, this infinite hierarchy is directly related to the mean field equation by the following observation.

\begin{Prop}\lb{P-FactSolInfHier}
Assume that the interaction kernel $K\in C^1(\bR^d\times\bR^d,\bR^d)$ satisfies assumptions (HK1)-HK2). Let $f^{in}$ be a smooth (at least $C^1$) 
probability density on $\bR^d$ such that
$$
\int_{\bR^d}|z|f^{in}(z)dz<\infty\,.
$$
Assume that the Cauchy problem for the mean field equation
$$
\left\{
\ba
{}&\d_tf(t,z)+\Div_z\left(f(t,z)\int_{\bR^d}K(z,z')f(t,z')dz'\right)=0\,,
\\
&f\rstr_{t=0}=f^{in}\,,
\ea
\right.
$$
has a classical (at least of class $C^1$) solution $f\equiv f(t,z)$. Set $f_j(t,\cdot)=f(t,\cdot)^{\otimes j}$, i.e.
$$
f_j(t,z_1,\ldots,z_j)=\prod_{k=1}^jf(t,z_k)
$$
for each $t\in\bR$ and each $z_1,\ldots,z_j\in\bR^d$.

Then the sequence $(f_j)_{j\ge 1}$ is a solution of the infinite mean field hierarchy
$$
\d_tf_j(z_1,\ldots,z_j)+\sum_{l=1}^j\Div_{z_l}\int_{\bR^d}K(z_l,z_{j+1})f_{j+1}(z_1,\ldots,z_{j+1})dz_{j+1}=0
$$
for all $j\ge 1$.
\end{Prop}

\begin{proof}
Since $f$ is of class $C^1$ at least, one has
$$
\ba
\d_tf_j(t,z_1,\ldots,z_j)&=\sum_{k=1}^j\prod_{l=1\atop l\not=k}^jf(t,z_l)\d_tf(t,z_k)
\\
&=-\sum_{k=1}^j\prod_{l=1\atop l\not=k}^jf(t,z_l)\Div_{z_k}\left(f(t,z_k)\int_{\bR^d}K(z_k,z')f(t,z')dz'\right)
\\
&=-\sum_{k=1}^j\Div_{z_k}\left(\prod_{l=1}^jf(t,z_l)\int_{\bR^d}K(z_k,z')f(t,z')dz'\right)
\\
&=-\sum_{k=1}^j\Div_{z_k}\left(\int_{\bR^d}K(z_k,z')f_{j+1}(t,z_1,\ldots,z_j,z')dz'\right)\,,
\ea
$$
which is precisely the $j$-th equation in the mean field hierarchy.
\end{proof}

\smallskip
This crucial observation suggests the following strategy to prove the mean field limit by the method of hierarchies.

Choose factorized initial data for $N$-particle Liouville equation: given $f^{in}$ a probability density on $\bR^d$ such that
$$
\int_{\bR^d}|z|f(t,z)dz<\infty\,,
$$
define
$$
F_N^{in}=(f^{in})^{\otimes N}\qquad\hbox{ for each }N\ge 1\,,
$$
i.e.
$$
F_N^{in}(t,z_1,\ldots,z_N)=\prod_{k=1}^Nf^{in}(t,z_k)\,.
$$
Let $F_N$ be the solution of the Cauchy problem for the $N$-particle Liouville equation
$$
\left\{
\ba
{}&\d_tF_N+\frac1N\sum_{k,l=1}^N\Div_{z_k}(F_NK(z_k,z_l))=0\,,
\\
&F_N\rstr_{t=0}=(f^{in})^{\otimes N}\,,
\ea
\right.
$$
with that initial data. 

Assume that one can prove

\smallskip
\noindent
a) that $F_{N:j}\to F_j$ (in some sense to be made precise) for each $j\ge 1$, where $F_j$ is a solution of the infinite hierarchy, and

\noindent
b) that the Cauchy problem for the infinite hierarchy has a unique solution. 

\smallskip
Let $f$ be a solution of the mean field PDE with initial data $f^{in}$.

Since the sequence $f_j:=f^{\otimes j}$ (for $j\ge 1$) is a solution of the infinite mean field hierarchy with initial data $(f^{in})^{\otimes j}$ by Proposition
\ref{P-FactSolInfHier}, statement b) implies that it is \textit{the} solution of the infinite hierarchy for that initial data. Therefore
$$
F_{N:j}\to F_j=f^{\otimes j}\qquad\hbox{ as }N\to\infty\hbox{ for all }j\ge 1\,.
$$

In particular, for $j=1$, one finds that the solution $F_N$ of the Liouville equation satisfies
$$
F_{N:1}\to f\qquad\hbox{ as }N\to\infty\,.
$$
In other words, the first marginal of the solution of the $N$-particle Liouville equation with factorized initial data converges to the solution of the mean
field PDE in the large $N$ limit.

This is precisely the strategy outlined by Cercignani \cite{CerciTTSP} for justifying rigorously the Boltzmann equation in the case of the hard sphere gas.

Notice that, in this approach, one needs to know that the Cauchy problem for the mean field PDE is well-posed (i.e. that it has a unique solution for each
initial data in some appropriate functional space). On the contrary, in the previous approach based on the notion of empirical measure, the existence of
a solution of the Cauchy problem for the mean field PDE is a consequence of the existence for finitely many particles and of the mean field limit itself, and
the uniqueness of that solution is a consequence of Dobrushin's estimate.

The reader might be under the impression that proving the uniqueness of the solution of the Cauchy problem for the infinite mean field hierarchy 
is a matter of pure routine, since the mean field hierarchy is a linear problem, at variance with the mean field PDE, which is nonlinear. This is obviously
wrong, since the uniqueness of the solution of the infinite mean field hierarchy implies the uniqueness of the solution of the mean field PDE. In fact,
the uniqueness property for the infinite mean field hierarchy is a very strong property and proving it is by no means obvious. See \cite{SpohnM2AS} and
section \ref{SS-InfHierar} for a precise discussion of this point --- as well as of the physical meaning of the infinite mean field hierarchy.

However, the strategy described in the present section (starting from the $N$-particle Liouville equation, deriving the BBGKY hierarchy, passing to the 
limit to arrive at an infinite hierarchy of equations and concluding with the uniqueness of the solution of the infinite hierarchy with given initial data) has 
been used successfully in a greater variety of problems than the mean field limit considered in this course. For instance, the only rigorous derivation of
the Boltzmann equation of the kinetic theory of gases known to this date (proposed by Lanford \cite{Lanford}) follows exactly these steps\footnote{For
the case of the Boltzmann-Grad limit for a system of $N$ hard spheres, the infinite hierarchy cannot be derived rigorously from the Liouville equation 
by passing to the limit in the sense of distributions in each equation of the (finite) BBGKY hierarchy: see the discussion on pp. 74--75 in \cite{CIP}. The
infinite Boltzmann hierarchy is derived by a different, more subtle procedure that is the core of the Lanford proof --- see section 4.4 in \cite{CIP}.} --- see 
also chapter 3 of \cite{BouGolPul} and chapters 2-4 of \cite{CIP} for an account of this fundamental result. The recent paper \cite{GalLSRTex} extends 
Lanford's result to short range potentials other than hard spheres, and gives a more detailed presentation of the Boltzmann-Grad limit than all previous 
references, even in the hard sphere case.

\section[Chaotic sequences]{Chaotic sequences, empirical measures and BBGKY hierarchies}\lb{S-CHAO}

Our discussion of BBGKY hierarchies shows the importance of the following property of symmetric $N$-particle probability measures $F_N$:
$$
F_{N:j}\to f^{\otimes j}\quad\hbox{ weakly as }N\to\infty\hbox{ for all }j\ge 1\hbox{ fixed.}
$$

Of course, if $\phi$ is a probability density on $\bR^d$, one has
$$
\Phi_N=\phi^{\otimes N}\Rightarrow\Phi_{N:j}=\phi^{\otimes j}\,.
$$

But if $\psi$ is another probability density on $\bR^d$, defining
$$
\tilde\Phi_N=\frac1N\sum_{k=1}^N\phi^{\otimes (k-1)}\otimes\psi\otimes\phi^{\otimes {N-k}}\,,
$$
which is in general a non-factorized symmetric probability density on $(\bR^d)^N$, then
$$
\tilde\Phi_{N:j}=\frac{N-j}{N}\phi^{\otimes j}+\frac1N\sum_{k=1}^j\phi^{\otimes (k-1)}\otimes\psi\otimes\phi^{\otimes {j-k}}\to\phi^{\otimes j}
$$
for all $j\ge 1$ as $N\to\infty$.

Thus, the property above can be verified, in the limit as $N\to\infty$, by sequences of $N$-particle probability measures that are not factorized exactly
for each finite $N\ge 1$.

\begin{Def}
Let $p$ be a probability measure on $\bR^d$. A sequence $P_N$ of symmetric $N$-particle probability measures on $(\bR^d)^N$ for all $N\ge 1$ is said 
to be chaotic, and more precisely $p$-chaotic, if\index{Chaotic sequence}
$$
P_{N:j}\to p^{\otimes j}\quad\hbox{Êweakly in }\cP((\bR^d)^j)
$$
as $N\to\infty$, for all $j\ge 1$ fixed.
\end{Def}

\smallskip
The notion of chaotic sequences appeared in the context of the derivation of kinetic equations from particle dynamics, for the first time in \cite{Kac56}. 
Perhaps the reason for this terminology is that this property corresponds to asymptotic independence of the $N$-particles in the large $N$ limit.

\subsection[Chaotic sequences and empirical measures]{Chaotic sequences and empirical measures}

We begin our discussion of chaotic sequences with a characterization of chaotic sequences in terms of empirical measures.

\begin{Thm}\lb{T-CharaChao}
Let $p\in\cP(\bR^d)$, and let $P_N\in\cP_{sym}((\bR^d)^N)$ for each $N\ge 1$. Then the two following properties are equivalent:

\smallskip
\noindent
(a) for each $j\ge 1$
$$
P_{N:j}\to p^{\otimes j}
$$
weakly in $\cP((\bR^d)^j)$ as $N\to\infty$;

\smallskip
\noindent
(b) for each $\phi\in C_b(\bR^d)$ and each $\eps>0$,
$$
P_N(\{Z_N\in(\bR^d)^N \hbox{ s.t. }|\la\mu_{Z_N}-p,\phi\ra|\ge\eps\})\to 0
$$
as $N\to\infty$, where we recall that
$$
\mu_{Z_N}:=\frac1N\sum_{k=1}^N\de_{z_k}
$$
with $Z_N=(z_1,\ldots,z_N)$.
\end{Thm}

\smallskip
In very informal terms, 
$$
P_{N:j}\to p^{\otimes j}
$$
weakly in $\cP((\bR^d)^j)$ as $N\to\infty$ for all $j\ge 1$ if and only if
$$
P_N\to\de_p
$$
``weakly in $\cP(\cP(\bR^d))$'' as $N\to\infty$. In this statement, each $N$-tuples $Z_N$ is identified with the corresponding empirical measure $\mu_{Z_N}$
and $P_N$ is viewed as a probability measure on $\cP(\bR^d)$ that is concentrated on the set of $N$-particle empirical measures. This identification goes
back to Gr\"unbaum \cite{Grunbaum71} and is discussed in section \ref{SS-Lions} below.

\medskip
\textbf{Remark:} Observe that property (a) is equivalent to

\smallskip
\noindent
(a') $P_{N:1}\to p$ and $P_{N:2}\to p\otimes p$ weakly in $\cP(\bR^d)$ and $\cP((\bR^d)^2)$ as $N\to\infty$.

\smallskip
In fact, as we shall see, the proof given below establishes that (a')$\Rightarrow$(b)$\Rightarrow (a)$.

\bigskip

\begin{proof} First we prove that property (a) implies property (b).

Applying Bienaym\'e-Chebyshev's inequality shows that
$$
P_N(\{Z_N\in(\bR^d)^N \hbox{ s.t. }|\la\mu_{Z_N}-p,\phi\ra|\ge\eps\})\le\frac1{\eps^2}\bE^{P_N}|\la\mu_{Z_N}-p,\phi\ra|^2\,.
$$

Then we compute
$$
\ba
\bE^{P_N}|\la\mu_{Z_N}-p,\phi\ra|^2&
\\
=
\bE^{P_N}\left(\frac1N\sum_{j=1}^N\phi(z_j)\right)^2+\bE^{P_N}\la p,\phi\ra^2-2\bE^{P_N}\left(\la p,\phi\ra\frac1N\sum_{j=1}^N\phi(z_j)\right)&
\\
=\bE^{P_N}\left(\frac1N\sum_{j=1}^N\phi(z_j)\right)^2+\la p,\phi\ra^2-2\la p,\phi\ra\bE^{P_N}\frac1N\sum_{j=1}^N\phi(z_j)&\,.
\ea
$$
Observe that, by symmetry of $P_N$, one has
$$
\bE^{P_N}(\phi(z_j)\phi(z_k))=\left\{\ba
{}&\bE^{P_N}(\phi(z_1)\phi(z_2))&&\quad\hbox{ if }j\not=k\,,
\\
&\bE^{P_N}(\phi(z_1)^2)&&\quad\hbox{ if }j=k\,.
\ea
\right.
$$
Thus
$$
\ba
\bE^{P_N}\left(\frac1N\sum_{j=1}^N\phi(z_j)\right)^2=\frac1{N^2}\sum_{j,k=1}^N\bE^{P_N}(\phi(z_j)\phi(z_k))
\\
=\frac1N\bE^{P_N}(\phi(z_1)^2)+\frac{N-1}{N}\bE^{P_N}(\phi(z_1)\phi(z_2))
\\
=\frac1N\la P_{N:1},\phi^2\ra+\frac{N-1}{N}\la P_{N:2},\phi\otimes\phi\ra
\\
\to\la p^{\otimes 2},\phi^{\otimes 2}\ra=\la p,\phi\ra^2
\ea
$$
as $N\to\infty$ by (a), while
$$
\ba
\bE^{P_N}\left(\frac1N\sum_{j=1}^N\phi(z_j)\right)&=\frac1N\sum_{j=1}^N\bE^{P_N}(\phi(z_j))
\\
&=\frac1N\sum_{j=1}^N\bE^{P_N}(\phi(z_1))=\la P_{N:1},\phi\ra\to\la p,\phi\ra
\ea
$$
as $N\to\infty$, again by (a). 

Therefore, property (a) implies that
$$
\bE^{P_N}|\la\mu_{Z_N}-p,\phi\ra|^2\to\la p,\phi\ra^2+\la p,\phi\ra^2-2\la p,\phi\ra^2=0
$$
as $N\to\infty$, and therefore
$$
P_N(\{Z_N\in(\bR^d)^N \hbox{ s.t. }|\la\mu_{Z_N}-p,\phi\ra|\ge\eps\})\le\frac1{\eps^2}\cdot o(1)
$$
by Bienaym\'e-Chebyshev's inequality, which is precisely property (b).

\medskip
Next we prove that, conversely, property (b) implies property (a).

\smallskip
\noindent
\textit{Step 1: } let us prove first that property (b) implies that $P_{N:1}\to p$ weakly as $N\to\infty$.

Let $\phi\in C_b(\bR^d)$. Denote
$$
\ba
U_N^\eps(\phi)&:=\{Z_N\in(\bR^d)^N \hbox{ s.t. }|\la\mu_{Z_N}-p,\phi\ra|>\eps\}\,,
\\
V_N^\eps(\phi)&:=(\bR^d)^N\setminus U_N^\eps(\phi)\,.
\ea
$$

Then
$$
\ba
|\bE^{P_N}\la\mu_{Z_N},\phi\ra-\la p,\phi\ra|=|\bE^{P_N}\la\mu_{Z_N}-p,\phi\ra|
\\
\le
\bE^{P_N}(|\la\mu_{Z_N}-p,\phi\ra|\indc_{U_N^\eps(\phi)})+\bE^{P_N}(|\la\mu_{Z_N}-p,\phi\ra|\indc_{V_N^\eps(\phi)})\,.
\ea
$$

Obviously
$$
\bE^{P_N}(|\la\mu_{Z_N}-p,\phi\ra|\indc_{U_N^\eps(\phi)})\le 2\|\phi\|_{L^\infty}P_N(U_N^\eps(\phi))\,,
$$
and 
$$
\bE^{P_N}(|\la\mu_{Z_N}-p,\phi\ra|\indc_{V_N^\eps(\phi)})\le\eps\,.
$$

By (b), there exists $N(\eps,\phi)$ such that
$$
N>N(\eps,\phi)\Rightarrow P_N(U_N^\eps(\phi))<\eps\,,
$$
so that
$$
|\bE^{P_N}\la\mu_{Z_N},\phi\ra-\la p,\phi\ra|\le(2\|\phi\|_{L^\infty}+1)\eps\,.
$$

On the other hand
$$
\bE^{P_N}(\la\mu_{Z_N},\phi\ra)=\bE^{P_N}\left(\frac1N\sum_{j=1}^N\phi(z_j)\right)=\bE^{P_N}(\phi(z_1))=\la P_{N:1},\phi\ra\,,
$$
so that
$$
N>N(\eps,\phi)\Rightarrow|\la P_{N:1},\phi\ra-\la p,\phi\ra|\le(2\|\phi\|_{L^\infty}+1)\eps\,.
$$
which concludes step 1.

\smallskip
\noindent
\textit{Step 2: } next we prove that property (b) implies that $P_{N:j}\to p^{\otimes j}$ weakly as $N\to\infty$, for all $j>1$.

Let $\phi\in C_b(\bR^d)\setminus\{0\}$. Let $\cE_N^j:=\{1,\ldots,N\}^{\{1,\ldots,j\}}$ (the set of maps from $\{1,\ldots,j\}$ to $\{1,\ldots,N\}$) and let $\cA_N^j$
be the set of one-to-one elements of $\cE_N^j$. Then
$$
\ba
\bE^{P_N}(\la\mu_{Z_N}^{\otimes j},\phi^{\otimes j}\ra)
&=
\bE^{P_N}\left(\frac1{N^j}\sum_{s\in\cE_N^j}\phi(z_{s(1)})\ldots\phi(z_{s(j)})\right)
\\
&=
\frac1{N^j}\sum_{s\in\cA_N^j}\bE^{P_N}(\phi(z_{s(1)})\ldots\phi(z_{s(j)}))
\\
&+
\frac1{N^j}\sum_{s\in\cE_N^j\setminus\cA_N^j}\bE^{P_N}(\phi(z_{s(1)})\ldots\phi(z_{s(j)}))\,.
\ea
$$

For $s\in\cA_N^j$, one has
$$
\bE^{P_N}(\phi(z_{s(1)})\ldots\phi(z_{s(j)}))=\bE^{P_N}(\phi(z_1)\ldots\phi(z_j))=\la P_{N:j},\phi^{\otimes j}\ra
$$
by symmetry of $P_N$, while, for all $s\in\cE_N^j$
$$
|\bE^{P_N}(\phi(z_{s(1)})\ldots\phi(z_{s(j)}))|\le\|\phi\|^j_{L^\infty}\,.
$$

Now, for all $j$ fixed
$$
\left\{
\ba
{}&\#\cA_N^j=N(N-1)\ldots(N-j+1)\sim N^j\hbox{ as }N\to\infty\,,
\\
&\#(\cE_N^j\setminus\cA_N^j)=N^j-N(N-1)\ldots(N-j+1)=o(N^j)\,,
\ea
\right.
$$
so that
$$
\left|\bE^{P_N}\la\mu_{Z_N}^{\otimes j},\phi^{\otimes j}\ra-\la P_{N:j},\phi^{\otimes j}\ra\right|
\le
2\frac{N^j-\#\cA_N^j}{N^j}\|\phi\|_{L^\infty}^j\,.
$$

Introduce
$$
\left\{
\ba
X_N^\eps(j,\phi)&=\{Z_N\in(\bR^d)^N\hbox{ s.t. }|\la\mu_{Z_N}^{\otimes j}-p^{\otimes j},\phi^{\otimes j}\ra|>\eps\}\,,
\\
Y_N^\eps(j,\phi)&=(\bR^d)^N\setminus X_N^\eps(j,\phi)\,.
\ea
\right.
$$
Observe that
$$
\ba
\la\mu_{Z_N}^{\otimes j}-p^{\otimes j},\phi^{\otimes j}\ra
=
\sum_{k=1}^j\la\mu_{Z_N},\phi\ra^{k-1}\la\mu_{Z_N}-p,\phi\ra\la p,\phi\ra^{j-k}\,,
\ea
$$
so that
$$
|\la\mu_{Z_N}^{\otimes j}-p^{\otimes j},\phi^{\otimes j}\ra|\le j\|\phi\|_{L^{\infty}}^{j-1}|\la\mu_{Z_N}-p,\phi\ra|\,.
$$

Therefore, property (b) implies that
$$
P_N(X_N^\eps(j,\phi))\le P_N(U_N^{\eps/j\|\phi\|_{L^\infty}^{j-1}}(\phi))\to 0
$$
as $N\to\infty$ for all $j>1$, all $\eps>0$ and all $\phi\in C_b(\bR^d)\setminus\{0\}$. 

In particular, for each $j\ge 1$, there exists $N_j(\eps,\phi)\ge j$ such that
$$
N>N_j(\eps,\phi)\Rightarrow P_N(X_N^\eps(j,\phi))<\eps\,.
$$

Thus
$$
\ba
|\bE^{P_N}\la\mu_{Z_N}^{\otimes j}-p^{\otimes j},\phi^{\otimes j}\ra|
&\le
\bE^{P_N}(|\la\mu_{Z_N}^{\otimes j}-p^{\otimes j},\phi^{\otimes j}\ra|\indc_{X_N^\eps(j,\phi)})
\\
&+\bE^{P_N}(|\la\mu_{Z_N}^{\otimes j}-p^{\otimes j},\phi^{\otimes j}\ra|\indc_{Y_N^\eps(j,\phi)})
\\
&\le 2\|\phi\|_{L^\infty}^jP_N(X_N^\eps(j,\phi))+\eps
\ea
$$
so that
$$
N>N_j(\eps,\phi)\Rightarrow|\bE^{P_N}\la\mu_{Z_N}^{\otimes j}-p^{\otimes j},\phi^{\otimes j}\ra|<(2\|\phi\|_{L^\infty}^j+1)\eps\,.
$$

Therefore, for each $\eps>0$ and each $j>1$, and for all $N>N_j(\eps,\phi)$, one has
$$
\ba
\left|\la p^{\otimes j}-P_{N:j},\phi^{\otimes j}\ra\right|
&\le
|\bE^{P_N}(\la p^{\otimes j}-\mu_{Z_N}^{\otimes j},\phi^{\otimes j}\ra)|
\\
&+
|\bE^{P_N}(\mu_{Z_N}^{\otimes j},\phi^{\otimes j}\ra)-\la P_{N:j},\phi^{\otimes j}\ra|
\\
&\le(2\|\phi\|_{L^\infty}^j+1)\eps+2\frac{N^j-\#\cA_N^j}{N^j}\|\phi\|_{L^\infty}^j\,,
\ea
$$
and since $N^j-\#\cA_N^j=o(N^j)$ as $N\to\infty$,
$$
\varlimsup_{N\to\infty}\left|\la p^{\otimes j}-P_{N:j},\phi^{\otimes j}\ra\right|\le(2\|\phi\|_{L^\infty}^j+1)\eps\,.
$$
Since this holds for each $\eps>0$, we conclude that
$$
\la P_{N:j},\phi^{\otimes j}\ra\to\la p^{\otimes j},\phi^{\otimes j}\ra\quad\hbox{ as }N\to\infty\,.
$$
This property holds for each $\phi\in C_b(\bR^d)\setminus\{0\}$ so that, by a classical density argument, we conclude that the sequence $P_N$ is 
$p$-chaotic.
\end{proof}

\smallskip
\noindent
\textbf{Exercise:} The purpose of this exercise is to complete the ``classical density argument'' used at the end of the proof of Theorem \ref{T-CharaChao}.
For each $\Phi\in C((\bR^d)^j)$, denote by $S_j\Phi$ the element of $C((\bR^d)^j)$ defined by the formula
$$
S_j\Phi(z_1,\ldots,z_j):=\frac1{j!}\sum_{\si\in\fS_j}\Phi(z_{\si(1)},\ldots,z_{\si(j)})\,.
$$
Finally, for each $\phi_1,\ldots,\phi_j\in C(\bR^d)$, we denote by $\phi_1\otimes\ldots\otimes\phi_j$ the function
$$
\phi_1\otimes\ldots\otimes\phi_j:\,(z_1,\ldots,z_j)\mapsto\phi_1(z_1)\ldots\phi_j(z_j)\,,
$$
and by $\phi^{\otimes j}$ the $j$-fold tensor product of $\phi$ with itself.

\noindent
a) Prove that
$$
S_j(\phi_1\otimes\ldots\otimes\phi_j)=\frac1{j!}\frac{\d^j}{\d X_1\ldots\d X_j}(X_1\phi_1+\ldots+X_j\phi_j)^{\otimes j}\,.
$$
b) Prove that, for each $Q\in\cP_{sym}((\bR^d)^j)$ and each $\Phi\in C_b((\bR^d)^j)$, one has
$$
\la Q,\Phi\ra=\la Q,S_j\Phi\ra\,.
$$
c) Let\footnote{For each locally compact topological space $X$ and each finite dimensional vector space $E$ on $\bR$, we denote by $C_0(X,E)$ the set
of continuous functions on $X$ with values in $E$ that converge to $0$ at infinity. We set $C_0(X):=C_0(X,E)$.} $\Phi\in C_0((\bR^d)^j)$ and $\eps>0$. 
Prove that there exists $\a_1,\ldots,\a_n\in\bR$ and a $jn$-tuple $(\phi_{i,m})_{1\le i\le j,1\le m\le n}$ of elements of $C_0(\bR^d)$ such that
$$
\left\|\Phi-\sum_{m=1}^n\a_m\phi_{1,m}\otimes\ldots\otimes\phi_{j,m}\right\|<\eps\,,
$$
where
$$
\|\Psi\|=\sup_{z_1,\ldots,z_j\in\bR^d}|\Psi(z_1,\ldots,z_j)|\,.
$$
(Hint: apply the Stone-Weierstrass theorem.)

\noindent
d) Let $Q\in\cP_{sym}((\bR^d)^j)$ and $(Q_N)_{N\ge 1}$ be a sequence of elements of $\cP_{sym}((\bR^d)^j)$ such that
$$
\la Q_N,\phi^{\otimes j}\ra\to\la Q,\phi^{\otimes j}\ra
$$
as $N\to\infty$ for each $\phi\in C_0(\bR^d)$. Prove that
$$
\la Q_N,\Phi\ra\to\la Q,\Phi\ra
$$
as $N\to\infty$ for each $\Phi\in C_0((\bR^d)^j)$.

\noindent
e) Conclude that
$$
\la Q_N,\Phi\ra\to\la Q,\Phi\ra
$$
as $N\to\infty$ for each $\Phi\in C_b((\bR^d)^j)$. (Hint: apply Theorem 6.8 in chapter II of \cite{Mallia}.)

\smallskip
References for this section are chapter 1, section 2 in \cite{Sznitman}, and chapter 4, section 6 in \cite{CIP}.

\subsection[Dobrushin's theorem and BBGKY hierarchies]{From Dobrushin's theorem to the BBGKY hierarchy}\lb{SS-DobruBBGKY}

In the proof of the mean-field limit based on the notion of empirical measure, one proves the stability of the limiting, mean field PDE in the weak topology
of probability measures on the single-particle phase space. The method based on the BBGKY hierarchy involves the $N$-particle phase space, and the 
need for considering marginals of the $N$-particle distribution in the limit as $N\to\infty$. Perhaps the best way to understand the relation between these
two approaches of the same problem is to think of the nature of the $N$-particle empirical measure. Indeed
$$
\mu_{Z_N}(dz)
$$
is a symmetric \textit{function} of the $N$ variables $Z_N=(z_1,\ldots,z_N)$, defined on the \textit{$N$-particle phase space} $(\bR^d)^N$, and with values 
in the set of \textit{probability measures on the single-particle phase space} $\bR^d$ --- in the notation above, $z$ is the variable in the single-particle phase space.

Already the characterization of chaotic sequences in terms of empirical measures obtained in the previous section clarifies the respective roles of the
single particle phase space and of the $N$-particle phase space in this limit. The $N$-particle symmetric distribution function $F_N$ can be viewed 
as a probability on $\cP(\bR^d)$ concentrated on the set of empirical measures --- in other words, the $N$-tuple $Z_N$ is, up to permutations of its $N$
components, identified with the empirical measure $\mu_{Z_N}$. (This idea can be found in \cite{Grunbaum71} and will be discussed in more detail in
section \ref{SS-Lions} below.)

Although this is a static picture, it provides the right point of view in order to unify the two approaches of the mean field limit presented above, i.e. the 
approach based on empirical measures and Dobrushin's estimate, and the one based on the BBGKY hierarchy. More precisely, we shall prove that the 
sequence $F_N(t)$ of solutions of the $N$-particle Liouville equation with factorized initial data is $f(t)\scrL^d$-chaotic for each $t\in\bR$, where $f$ is 
the solution of the mean field PDE.

The proof of this result goes as follows.

Since $F_N^{in}$ is factorized, it is of course chaotic (factorized probability measures being the first example of chaotic measures). But in fact, $F_N^{in}$ 
satisfies a stronger estimate than property (b) in Theorem \ref{T-CharaChao}, in terms of the Monge-Kantorovich distance $\Dist_{MK,2}$ (see Theorem 
\ref{T-HoroKara} below). This estimate is propagated for all $t\in\bR$ by Dobrushin's inequality, involving the weaker distance $\Dist_{MK,1}$. The resulting 
bound implies that property (b) in Theorem \ref{T-CharaChao} holds for $F_N(t)$ for all $t$ with $p:=f(t)\scrL^d$. Applying Theorem \ref{T-CharaChao}, we 
conclude that $F_N(t)$ is $f(t)\scrL^d$-chaotic.

\smallskip
Our main result in this section is summarized in the following statement.

\begin{Thm}\lb{T-PropaChao}\index{Propagation of chaos}
Assume that the interaction kernel $K\in C^1(\bR^d\times\bR^d,\bR^d$ satisfies assumptions (HK1)-(HK2). Let $f^{in}$ be a probability density on $\bR^d$ 
such that
$$
\int_{\bR^d}|z|^{d+5}f(z)dz<\infty\,.
$$
Let $F_N^{in}=(f^{in}\scrL^d)^{\otimes N}$, and let $F_N(t)=T_t\#F_N^{in}$ be the solution of the $N$-particle Liouville equation
$$
\left\{
\ba
{}&\d_tF_N+\frac1N\sum_{j=1}^N\Div_{z_i}(K(z_i,z_j)F_N)=0\,,
\\
&F_N\rstr_{t=0}=F_N^{in}\,.
\ea
\right.
$$
Then, for each $j\ge 1$
$$
F_{N:j}(t)\to (f(t,\cdot)\scrL^d)^{\otimes j}
$$
weakly in $\cP((\bR^d)^j)$ as $N\to\infty$, where the probability density $f(t,\cdot)$ is the solution of the mean field PDE
$$
\left\{
\ba
{}&\d_tf(t,z)+\Div_z\left(f(t,z)\int_{\bR^d}K(z,z')f(t,z')dz'\right)=0\,,
\\
&f\rstr_{t=0}=f^{in}\,.
\ea
\right.
$$ 
\end{Thm}

\smallskip
The rather stringent moment condition on the initial single-particle density $f^{in}$ comes from the following important result from statistics, that can be viewed
as a quantitative variant of the law of large numbers.

\begin{Thm}[Horowitz-Karandikar]\lb{T-HoroKara}\index{Horowitz-Karandikar theorem}
For all  $p\in\cP(\bR^d)$ such that
$$
a:=\la p,|z|^{d+5}\ra<\infty\,,
$$
one has
$$
\int_{(\bR^d)^N}\Dist_{MK,2}(\mu_{Z_N},p)^2p^{\otimes N}(dZ_N)\le\frac{C(a,d)^2}{N^{2/(d+4)}}\,,
$$
where $C(a,d)$ is a positive constant that depends only on $a$ and the space dimension $d$.
\end{Thm}

\smallskip
References for this result are the original article \cite{HoroKara}, and the monograph \cite{Rachev}.

\bigskip
Taking this estimate for granted, we give the proof of Theorem \ref{T-PropaChao}.

\begin{proof}[Proof of Theorem \ref{T-PropaChao}]
By the Cauchy-Schwarz inequality, for all $\mu,\nu\in\cP_2(\bR^d)$ and all $\pi\in\Pi(\mu,\nu)$, one has
$$
\iint_{\bR^d}|x-y|\pi(dxdy)\le\left(\iint_{\bR^d}|x-y|^2\pi(dxdy)\right)^{1/2}\,.
$$
Taking the infimum of both sides of the inequality above as $\pi$ runs through $\Pi(\mu,\nu)$ shows that
$$
\Dist_{MK,1}(\mu,\nu)\le\Dist_{MK,2}(\mu,\nu)\,.
$$
Hence
$$
\int_{(\bR^d)^N}\Dist_{MK,1}(\mu_{Z_N},f^{in}\scrL^d)^2(f^{in})^{\otimes N}(Z_N)(dZ_N)\le\frac{C(a,d)}{N^{2/(d+4)}}\,.
$$

Since $F_N(t)=T_t\#F_N^{in}$, one has
$$
\ba
F_N(t)(\{Z_N\in(\bR^d)^N\hbox{ s.t. }|\la\mu_{Z_N}-f(t)\scrL^d,\phi\ra|\ge\eps\})
\\
=
F_N^{in}(\{Z_N^{in}\in(\bR^d)^N\hbox{ s.t. }|\la\mu_{T_tZ_N^{in}}-f(t)\scrL^d,\phi\ra|\ge\eps\})\,.
\ea
$$

For each $\pi\in\Pi(\mu_{T_tZ_N^{in}},f(t)\scrL^d)$, 
$$
\ba
|\la\mu_{T_tZ_N^{in}}-f(t)\scrL^d,\phi\ra|=\left|\iint_{\bR^d\times\bR^d}(\phi(x)-\phi(y))\pi(dxdy)\right|
\\
\le\Lip(\phi)\iint_{\bR^d\times\bR^d}|x-y|\pi(dxdy)\,.
\ea
$$
Taking the inf of both sides of this inequality as $\pi$ runs through the set of couplings $\Pi(\mu_{T_tZ_N^{in}},f(t)\scrL^d)$ shows that
$$
|\la\mu_{T_tZ_N^{in}}-f(t)\scrL^d,\phi\ra|\le\Lip(\phi)\Dist_{MK,1}(\mu_{T_tZ_N^{in}},f(t)\scrL^d)\,.
$$

By Proposition \ref{P-MKDual} and Dobrushin's inequality (Theorem  \ref{T-Dobru}), 
$$
\ba
|\la\mu_{T_tZ_N^{in}}-f(t)\scrL^d,\phi\ra|&\le\Lip(\phi)\Dist_{MK,1}(\mu_{T_tZ_N^{in}},f(t)\scrL^d)
\\
&\le \Lip(\phi)e^{2L|t|}\Dist_{MK,1}(\mu_{Z_N^{in}},f^{in}\scrL^d)\,.
\ea
$$

Therefore
$$
\ba
F_N(t)(\{Z_N\in(\bR^d)^N\hbox{ s.t. }|\la\mu_{Z_N}-f(t)\scrL^d,\phi\ra|\ge\eps\})&
\\
\le
F_N^{in}(\{Z_N^{in}\in(\bR^d)^N\hbox{ s.t. }\Dist_{MK,1}(\mu_{Z_N^{in}},f^{in}\scrL^d)\ge e^{-2L|t|}\eps/\Lip(\phi)\})&\,.
\ea
$$

Applying the Horowitz-Karandikar theorem recalled above and the Bienaym\'e-Chebyshev inequality shows that
$$
\ba
F_N^{in}(\{Z_N^{in}\in(\bR^d)^N\hbox{ s.t. }\Dist_{MK,1}(\mu_{Z_N^{in}},f^{in}\scrL^d)\ge e^{-2L|t|}\eps/\Lip(\phi)\})&
\\
\le\frac{e^{4L|t|}\Lip(\phi)^2}{\eps^2}\int_{(\bR^d)^N}\Dist_{MK,1}(\mu_{Z_N^{in}},f^{in}\scrL^d)^2(f^{in})^{\otimes N}(Z_N)dZ_N&
\\
\le
\frac{e^{4L|t|}\Lip(\phi)^2}{\eps^2}\frac{C(a,d)^2}{N^{2/(d+4)}}&\,.
\ea
$$

Hence
$$
F_N(t)(\{Z_N\in(\bR^d)^N\hbox{ s.t. }|\la\mu_{Z_N}-f(t)\scrL^d,\phi\ra|\ge\eps\})\to 0
$$
as $N\to\infty$ for each $\phi\in L^\infty\cap\Lip(\bR^d)$. By density of $L^\infty\cap\Lip(\bR^d)$ in $C_b(\bR^d)$ and the implication (b) $\Rightarrow$ (a)
in Theorem \ref{T-CharaChao}, we conclude that 
$$
F_{N:j}(t)\to (f(t)\scrL^d)^{\otimes j}
$$
weakly in $\cP((\bR^d)^j)$ as $N\to\infty$ for all $j\ge 1$.

In particular
$$
F_{N:1}(t)\to f(t)\scrL^d
$$
weakly in $\cP(\bR^d)$ as $N\to\infty$, where $f(t,\cdot)$ is the solution of the mean field PDE, and this concludes the proof.
\end{proof}

\section[Further results on mean field limits]{Further results on mean field limits in classical mechanics}\lb{S-MFClassFurther}

\subsection{Propagation of chaos and quantitative estimates}\lb{SS-QuantPropaChao}

Following carefully our proof in the previous section, the convergence result stated in Theorem \ref{T-PropaChao} could obviously have been formulated 
as a quantitative estimate of the form
$$
\Dist_{MK,1}(P_{N:j}(t),(f(t)\scrL^d)^{\otimes j})\le\eps_j(N)\,,
$$
where $\eps_j(N)\to 0$ for each $j\ge 1$ in the limit as $N\to\infty$. Such an estimate would obviously imply Theorem \ref{T-PropaChao} since the 
Monge-Kantorovich distance $\Dist_{MK,1}$ metricizes the weak topology on $\cP_1(\bR^n)$ for all $n\ge 1$ --- see Theorem 6.9 in \cite{VillaniOTON},
although what is needed here is a consequence of Proposition \ref{P-MKDual}.

In fact, there is another approach of Theorem \ref{T-PropaChao} which is perhaps worth mentioning, since it provides additional information on the
relation between the approach with the empirical measure and the approach based on BBGKY hierarchies.

Specifically, one can prove that, if $F_N^{in}\in\cP_{1,sym}((\bR^d)^N)$ and if $t\mapsto F_N(t)$ is the solution of the Cauchy problem for the $N$-particle 
Liouville equation with initial data $F_N^{in}$, assuming that the interaction kernel $K$ satisfies the assumptions (HK1-HK2), then
$$
\int_{(\bR^d)^N}\mu_{T_tZ_N^{in}}^{\otimes m}F_N^{in}(dZ_N^{in})=\frac{N!}{(N-m)!N^m}F_{N:m}(t)+R_{N,m}(t)
$$
where $T_t$ is the flow defined by the $N$-particle ODE system, while $R_{N,m}(t)$ is a positive Radon measure on $(\bR^d)^m$ whose total mass satisfies
$$
\la R_{N,m}(t),1\ra=1-\frac{N!}{(N-m)!N^m}\le\frac{m(m-1)}{2N}\,.
$$

This explicit formula, which can be found in \cite{FGVRicci} expresses the $m$-th marginal of the $N$-particle distribution $F_N(t)$ in terms of the empirical 
measure of the $N$-particle system at time $t$, i.e. $\mu_{T_tZ_N^{in}}$, up to an error that vanishes as $N\to\infty$.

With this formula, one easily arrives at the following quantitative estimate for the propagation of chaos in the mean field problem.  Assume that the initial
data is factorized, i.e.
$$
P_N^{in}=(f^{in}\scrL^d)^{\otimes N}
$$ 
and that
$$
a:=\int_{\bR^d}|z|^{d+5}f^{in}(z)dz<\infty\,.
$$
Then\index{Propagation of chaos (quantitative estimates)}
$$
\|P_{N:m}(t)-(f(t)\scrL^d)^{\otimes m}\|_{W^{-1,1}((\bR^d)^m)}\le m\left(\frac{m-1}{N}+e^{2L|t|}\frac{C(a,d)}{N^{1/(d+4)}}\right)\,.
$$
for all $t\in\bR$ and all $N\ge m\ge 1$. In particular, for $m=1$, one has
$$
\Dist_{MK,1}(P_{N:1}(t),f(t)\scrL^d)\le C(a,d)e^{2L|t|}/N^{1/(d+4)}\,.
$$

This estimate can be found in \cite{FGVRicci}; it is established independently and in a slightly different manner in \cite{MischlerMouhotWenn}.

In fact, the idea of obtaining quantitative estimates for the propagation of chaos in various situations of non equilibrium classical statistical mechanics has 
been systematically pursued in the reference \cite{MischlerMouhotWenn}, which bears on a more general class of models than the one considered so far
in the present notes. In particular, the discussion in \cite{MischlerMouhotWenn} applies to situations where the empirical measure of the $N$-particle system
is \textit{not an exact weak solution} of the limiting mean field equation. 

\subsection{Infinite hierarchies and statistical solutions}\lb{SS-InfHierar}

While the physical content of the BBGKY hierarchy is transparent (it is the string of equations satisfied by the marginals of the $N$-particle distribution), the 
physical meaning of the infinite mean field hierarchy is somewhat less obvious. We discuss this point in the present section.

Consider the Cauchy problem for the mean field PDE:
$$
\left\{
\ba
{}&\d_tf+\Div_z(f\cK f)=0\,,
\\
&f\rstr_{t=0}=f^{in}\,,
\ea
\right.
$$
where we recall that 
$$
\cK f(t,z):=\int_{\bR^d}K(z,z')f(t,dz')\,,
$$
and where the interaction kernel $K\in C^1(\bR^d\times\bR^d,\bR^d)$ is assumed to satisfy conditions (HK1-HK2). 

Denote by $G_t:\,\cP_1(\bR^d)\mapsto\cP_1(\bR^d)$ the nonlinear 1-parameter group defined by
$$
G_tf^{in}:=f(t,\cdot)
$$
where $f$ is the unique solution of the Cauchy problem above.

The notion of statistical solutions of the mean field PDE\index{Statistical solution of the mean field PDE} is defined by analogy with the case of an ODE. 
Consider the Cauchy problem for an ODE with unknown $t\mapsto x(t)\in\bR^n$ in the form
$$
\left\{
\ba
{}&\dot x(t)=v(x(t))\,,
\\
&x(0)=x_0\,.
\ea
\right.
$$
Assuming that $v\in\Lip(\bR^n,\bR^n)$, the Cauchy-Lipschitz theorem provides the existence of a global solution flow $X:\bR\times\bR^n\to\bR^n$, so that
the map 
$$
t\mapsto X(t,x_0)
$$ 
is the solution of the ODE above satisfying $X(0,x_0)=x_0$. The flow $X$ corresponds with a completely deterministic notion of solution: knowing the initial
condition $x_0$ exactly determines the solution $t\mapsto X(t,x_0)$ for all times. Suppose that, instead of knowing exactly the initial condition $x_0$, one is 
given a probability distribution $\mu_0$ on the set $\bR^n$ of initial data $x_0$. In other words, $\mu_0$ can be regarded as a ``statistical initial condition'' 
for the Cauchy problem above. With the flow $X$ and the statistical initial condition $\mu_0$, one defines 
$$
\mu(t):=X(t,\cdot)\#\mu_0\,,\qquad t\in\bR\,.
$$
Applying the method of characteristics (see the exercise following Theorem \ref{T-ExUnNbody}) shows that 
$$
\left\{
\ba
{}&\d_t\mu(t)+\Div(\mu(t)v)=0\,,
\\
&\mu(0)=\mu_0\,.
\ea
\right.
$$
In other words, we recover the well-known fact that solutions of the transport equation (a 1st order PDE) can be viewed as statistical solutions of the ODE
defined by the characteristic field of the transport operator. The idea of considering the time-dependent probability $\mu(t)$ (or distribution function, when
$\mu(t)$ is absolutely continuous with respect to the Lebesgue measure on the single-particle phase space) instead of the deterministic solution $X(t,x_0)$
for each initial data $x_0$ lies at the core of the kinetic theory of gases.

Let us return to the problem of defining a notion of statistical solution of the mean field PDE recalled above. First we need a probability measure $\nu_0$
on the space of initial conditions $f^{in}$, in this case on $\cP_1(\bR^d)$: this probability $\nu_0$ will be the statistical initial condition for the mean field
PDE. By analogy with the case of the simple ODE presented above, we define
$$
\nu(t):=G_t\#\nu_0\,,\qquad t\in\bR\,.
$$

The next question is to find the analogue of the transport equation satisfied by $\mu(t)$. 

First we need to have a better grasp on $\nu(t)$. One way to understand the formula above expressing $\nu(t)$ as the push-forward of $\nu_0$ under the
map $G_t$ defined by the mean field evolution is to write
$$
\int_{\cP_1(\bR^d)}\cF(p)\nu(t,dp)=\int_{\cP_1(\bR^d)}\cF(G_tp)\nu_0(dp)
$$
for some appropriate class of continuous (in some sense to be defined) functions $\cF$ on $\cP_1(\bR^d)$. Certainly this class of functions should contain
``polynomials'' on $\cP_1(\bR^d)$, i.e. linear combinations of ``monomials''. A monomial of degree $k$ on $\cP(\bR^d)$ is a expression of the form
$$
M_{k,\phi}(p):=\int_{(\bR^d)^k}\phi(z_1,\ldots,z_k)p(dx_1)\ldots p(dx_k)=\la p^{\otimes k},\phi\ra
$$
where $\phi\in C_b((\bR^d)^k)$ --- without loss of generality one can assume that $\phi$ is symmetric in the variables $z_1,\ldots,z_k$. 

Specializing the formula above to the case $\cF=M_{j,\phi}$ as $\phi$ runs through $C_b((\bR^d)^j)$ results in the equality
$$
\int_{\cP_1(\bR^d)}p^{\otimes j}\nu(t,dp)=\int_{\cP_1(\bR^d)}(G_tp)^{\otimes j}\nu_0(dp)\,.
$$
Defining
$$
F_j(t):=\int_{\cP_1(\bR^d)}(G_tp)^{\otimes j}\nu_0(dp)\,,\quad j\ge 1\,,
$$
we claim that the sequence $(F_j)_{j\ge 1}$ is a solution of the infinite mean field hierarchy. Indeed, for each initial single particle probability distribution
$p\in\cP_1(\bR^d)$, the sequence $((G_tp)^{\otimes j})_{j\ge 1}$ is a solution of the infinite mean field hierarchy, which is a sequence of linear equations. 
Therefore $(F_j)_{j\ge 1}$ is also a solution of the infinite mean field hierarchy, being an average under $\nu_0$ of solutions of this hierarchy.

There is another formulation of this observation. Set $\Om:=(\bR^d)^{\bN^*}$, equipped with its product topology and the associated Borel $\si$-algebra
$\cB(\Om)$. For each $p\in\cP_1(\bR^d)$, we denote by $p^{\otimes\infty}$ the unique Borel probability measure defined on $\Om$ by the formula
$$
p^{\otimes\infty}\left(\prod_{k\ge 1}E_k\right)=\prod_{k\ge 1}p(E_k)
$$
for each sequence $(E_k)_{k\ge 1}$ of Borel subsets of $\bR^d$ such that $E_k=\bR^d$ for all but finitely many $k$s. Notice that
$$
\left(p^{\otimes\infty}\right)_{:j}=p^{\otimes j}\,,\qquad j\ge 1\,.
$$
Define
$$
\bF(t):=\int_{\cP_1(\bR^d)}p^{\otimes\infty}\nu(t,dp)=\int_{\cP_1(\bR^d)}(G_tp)^{\otimes\infty}\nu_0(dp)\,;
$$
then, for each $j\ge 1$ and each $t\in\bR$,
$$
\bF(t)_{:j}=F_j(t)
$$
so that $(\bF(t)_{:j})_{j\ge 1}$ is a solution of the infinite mean field hierarchy. Let us write the $j$th equation of the infinite mean field hierarchy in terms 
of $\bF(t)$:
$$
\d_t\bF(t)_{:j}+\sum_{i=1}^j\Div_{z_i}(\bF(t)_{:j+1}K(z_i,z_{j+1}))_{:j}=0\,.
$$
This equality is equivalent to the following weak formulation: for each test function $\psi_j\in C^1_c((\bR^d)^j)$
$$
\ba
\d_t\la\bF(t),\psi_j\ra=\d_t\la\bF(t)_{:j},\psi_j\ra=\sum_{i=1}^j\la\bF(t)_{:j+1},K(z_i,z_{j+1})\cdot\grad_{z_i}\psi_j\ra
\\
=\La\bF(t),\sum_{i=1}^jK(z_i,z_{j+1})\cdot\grad_{z_i}\psi_j\Ra
\ea
$$

This suggests the following definition. We denote by $\cP_{sym}(\Om)$ the set of Borel probability measures $\mu$ on $\Om$ such that
$$
U_\si\#\mu=\mu
$$
for each $N>1$ and each $\si\in\fS_N$, where
$$
U_\si(z_1,z_2,\ldots):=(z_{\si(1)},\ldots,z_{\si(N)},z_{N+1},z_{N+2},\ldots)\,.
$$

\begin{Def}[Spohn \cite{SpohnM2AS}]
Let $I\subset\bR $ be an interval. A map 
$$
I\ni t\mapsto\bP(t)\in\cP_{sym}(\Om)
$$ 
is a measure-valued solution of the mean field hierarchy if and only if, for each $j\in\bN^*$ and each\footnote{For each $n,k\ge 1$ and each finite dimensional 
vector space $E$ on $\bR$, we denote by $C^k_0(\bR^n,E)$ the set of functions of class $C^k$ defined on $\bR^n$ with values in $E$ all of whose partial 
derivatives converge to $0$ at infinity. In other words,
$$
C^k_0(\bR^n,E):=\{f\in C^k(\bR^n,E)\hbox{ s.t. }\d^\a f(x)\to 0\hbox{ as }|x|\to\infty\hbox{Ê for each }\a\in\bN^n\}\,.
$$
We denote $C^k_0(\bR^n):=C^k_0(\bR^n,\bR)$.} $\psi_j\in C^1_0((\bR^d)^j)$, the map $t\mapsto\la\bP(t),\psi_j\ra$ is of class $C^1$ on the interval $I$ and
$$
\d_t\la\bP(t),\psi_j\ra=\d_t\la\bP(t)_{:j},\psi_j\ra=\La\bP(t),\sum_{i=1}^jK(z_i,z_{j+1})\cdot\grad_{z_i}\psi_j\Ra
$$
for all $t\in I$.
\end{Def}

\smallskip
With this definition, the map $\bR\ni t\mapsto\bF(t)\in\cP_{sym}(\Om)$ defined above is a measure-valued solution of the mean-field hierarchy satisfying
the initial condition
$$
\bF(0)=\int_{\cP_1(\bR^d)}p^{\otimes\infty}\nu_0(dp)\,.
$$

\medskip
This observation raises two natural questions:

\smallskip
a) are measure-valued solutions of the mean-field hierarchy uniquely determined by their initial data?

b) are all the measure-valued solutions of the mean-field hierarchy defined by statistical solutions of the mean field PDE?

\medskip
A first useful tool in answering these questions is the following result.

\begin{Thm}[Hewitt-Savage \cite{HewSav}]\index{Hewitt-Savage theorem}
For each $\bP\in\bP_{sym}(\Om)$, there exists a unique Borel probability measure $\pi$ on $\cP(\bR^d)$ such that 
$$
\bP=\int_{\cP(\bR^d)}p^{\otimes\infty}\pi(dp)\,.
$$
\end{Thm} 

\smallskip
The Hewitt-Savage theorem is often quoted and used as follows. Given $(P_j)_{j\ge 1}$ a sequence of probability measures such that 
$P_j\in\cP_{sym}((\bR^d)^j)$ satisfying the compatibility condition
$$
P_{k:j}=P_j\quad\hbox{ for each }k>j\ge 1\,,\leqno{(CC)}
$$
by the Kolmogorov extension theorem \cite{StroockVarad}, there exists a unique probability measure $\bP\in\bP_{sym}(\Om)$ such that 
$$
\bP_{:j}=P_j\quad\hbox{ for each }j\in\bN^*\,.
$$
Therefore, by the Hewitt-Savage theorem, there exists a unique Borel probability measure $\pi$ on $\cP(\bR^d)$ such that
$$
P_j=\int_{\cP(\bR^d)}p^{\otimes j}\pi(dp)\,.
$$

\smallskip
With this, we can answer the questions a) and b) above. 

First we consider the problem of admissible initial data. In view of the discussion above, the initial data for the mean field hierarchy can be equivalently
either a sequence $(F_j^{in})_{j\ge 1}$ such that $F_j^{in}\in\cP_{1,sym}((\bR^d)^j)$ for $j\ge 1$ and satisfying the compatibility conditions (CC), or a 
unique element $\bF^{in}\in\cP(\Om)$ such that
$$
\bF^{in}_{:j}=F_j\,,\quad j\ge 1\,,
$$
or a unique Borel probability measure $\nu^{in}$ on $\cP_1(\bR^d)$ such that
$$
F^{in}_j=\int_{\cP_1(\bR^d)}p^{\otimes j}\nu^{in}(dp)\,,\quad j\ge 1\,,
$$
or equivalently
$$
\bF^{in}=\int_{\cP_1(\bR^d)}p^{\otimes\infty}\nu^{in}(dp)\,.
$$

The same reasoning applies for each instant of time $t\not=0$. A solution at time $t$ of the infinite mean field hierarchy is a sequence $(P_j(t))_{j\ge 1}$ 
that satisfies in particular $P_j(t)\in\cP_{1,sym}((\bR^d)^j)$ for $j\ge 1$ together with the compatibility condition (CC). Equivalently, this defines a unique
probability measure $\bP(t)\in\cP_{sym}(\Om)$ such that $\bP(t)_{:j}=P_j(t)$ for each $j\ge 1$, or a unique Borel probability measure $\pi(t)$ on $\cP(\bR^d)$ such that
$$
P_j(t)=\int_{\cP(\bR^d)}p^{\otimes j}\pi(t,dp)\,,\qquad j\ge 1\,,
$$
or equivalently
$$
\bP(t)=\int_{\cP(\bR^d)}p^{\otimes\infty}\pi(t,dp)\,.
$$

The following important result was proved by H. Spohn \cite{SpohnM2AS}. It answers questions a) and b) above.

\begin{Thm}\index{Spohn's uniqueness theorem}
Under the assumptions (HK1-HK2) on the interaction kernel $K\in C^1(\bR^d\times\bR^d,\bR^d)$, for each Borel probability measure $\nu^{in}$ on 
$\cP_1(\bR^d)$, the only measure-valued solution of the infinite mean field hierarchy with initial data 
$$
\bF^{in}=\int_{\cP_1(\bR^d)}p^{\otimes\infty}\nu^{in}(dp)
$$
is 
$$
\bF(t)=\int_{\cP_1(\bR^d)}(G_tp)^{\otimes\infty}\nu^{in}(dp)\,,\quad t\in\bR\,.
$$
\end{Thm}

\smallskip
In other words, Spohn's theorem proves that the only measure valued solution of the infinite mean field hierarchy is the statistical solution of the mean field
PDE
$$
\nu(t):=G_t\#\nu^{in}\,,\quad t\in\bR\,,
$$
where $\nu^{in}$ is the initial probability measure on the space $\cP_1(\bR^d)$ of initial data for the mean field PDE such that the initial condition for the
mean field hierarchy is the element of $\cP_{sym}(\Om)$ given by
$$
\bF^{in}=\int_{\cP_1(\bR^d)}p^{\otimes\infty}\nu^{in}(dp)\,,
$$
or equivalently the sequence
$$
F_j^{in}=\int_{\cP_1(\bR^d)}p^{\otimes j}\nu^{in}(dp)\in\cP_{sym}((\bR^d)^j)\,,\qquad j\ge 1\,.
$$

Notice that, in the case where $\nu^{in}=\de_{f^{in}}$, one has
$$
\bF^{in}=(f^{in})^{\otimes\infty}\,,\hbox{ or equivalently }F_j^{in}=(f^{in})^{\otimes j}\hbox{ for all }j\ge 1\,,
$$
and
$$
\bF(t)=(G_tf^{in})^{\otimes\infty}\,,\hbox{ or equivalently }F_j(t)=(G_tf^{in})^{\otimes j}\hbox{ for all }j\ge 1\,.
$$
In other words the statistical solution $\nu(t)$ of the mean field PDE at time $t$ is given by the formula
$$
\nu(t)=\de_{f(t)}=\de_{G_tf^{in}}=G_t\#\de_{f^{in}}\,.
$$
That the statistical solution $\nu(t)$ is a Dirac measure for each $t\in\bR$, knowing that $\nu(0)=\nu^{in}$ is a Dirac measure, is equivalent to the propagation
of chaos in the context of the mean field limit.

Therefore, Spohn's theorem contains as a particular case the uniqueness theorem with factorized initial data which justifies the validity of the mean field limit 
in the approach with the BBGKY hierarchy.

But more generally, Spohn's theorem shows that solutions of the infinite mean field hierarchy coincide with the notion of statistical solutions of the mean field 
PDE. This important piece of information clarifies the meaning of the infinite hierarchy.

Spohn's uniqueness theorem can also be combined with the following uniform stability result on the BBGKY hierarchy to produce a quantitative stability 
estimate on statistical solutions of the mean field PDE. The following quantitative stability estimate uses the following variant of Monge-Kantorovich distance. 
For each $P\in\cP_{1,sym}((\bR^d)^M)$ and $Q\in\cP_{1,sym}((\bR^d)^N)$, consider 
$$
\DDist_{MK,1}(P,Q)=\inf_{\rho\in\Pi(P,Q)}\iint_{\bR^{dM}\times\bR^{dN}}\Dist_{MK,1}(\mu_{X_M},\mu_{Y_N})\rho(dX_M,dY_N)\,.
$$
In this formula, $P$ and $Q$ are viewed as Borel probability measures on $\cP_1(\bR^d)$ concentrated on the set of $M$- and $N$-particle empirical 
measures respectively.

\begin{Thm}\index{Quantitative stability estimate for statistical solutions}\lb{T-UnifStabBBGKY}
Let $M,N\ge 1$, and let $P^{in}_M\in\cP_{1,sym}((\bR^d)^M)$ and $Q^{in}_N\in\cP_{1,sym}((\bR^d)^N)$. Assume that the interaction kernel $K$ satisfies 
the conditions (HK1-HK2). Let $t\mapsto P_M(t)$ and $t\mapsto Q_N(t)$ be respectively the solutions of the $M$-particle and the $N$-particle Liouville 
equations with initial data $P^{in}_M$ and $Q^{in}_N$ respectively. Then

\smallskip
\noindent
a) for each $t\in\bR$, one has
$$
\DDist_{MK,1}(P_M(t),Q_N(t))\le e^{2L|t|}\DDist_{MK,1}(P_M^{in},Q_N^{in})\,.
$$

\noindent
b) for each $t\in\bR$, each $m,M,N\in\bN^*$ such that $M,N\ge m$, and for each bounded and Lipschitz continuous function $\phi_m$ defined on 
$(\bR^d)^m$, one has
$$
\ba
{}&|\la P_{M:m(t)}-Q_{N:m}(t),\phi_m\ra|
\\
&\le m\left(e^{2L|t|}\Lip(\phi_m)\DDist_{MK,1}(P_M^{in},Q_N^{in})+(m-1)\|\phi_m\|_{L^\infty}\left(\frac1M+\frac1N\right)\right)\,.
\ea
$$
\end{Thm}

For a proof of this result, see \cite{FGVRicci}. 

\smallskip
Notice however that Spohn's uniqueness theorem for the infinite mean field hierarchy, even in the particular case of factorized initial data, is more than
what is needed to justify the mean field limit. It would be enough to prove the uniqueness of those solutions of the infinite hierarchy that are limits of 
the sequence of marginals of $N$-particle distributions as $N\to\infty$. This weaker uniqueness property follows from Theorem \ref{T-UnifStabBBGKY},
without using Spohn's uniqueness theorem.

\subsection{Symmetric functions of infinitely many variables}\lb{SS-Lions}

In various places in this notes --- and especially in the last theorem --- we encountered the idea of viewing elements of $\cP_{sym}((\bR^d)^N)$ as Borel 
probability measures on $\cP(\bR^d)$ concentrated on the set of empirical measures. 

In fact, the identification
$$
(\bR^d)^N/\fS_N\ni(z_1,\ldots,z_N)\mapsto\mu_{Z_N}:=\frac1N\sum_{k=1}^N\de_{z_k}\in\cP_{sym}(\bR^d)
$$
can be found in \cite{Grunbaum71} (see especially p. 330 there). 

This point of view was pushed much further by P.-L. Lions. He constructed a complete mathematical framework for handling continuous symmetric functions 
of infinitely many variables that depend weakly on each variable, and used it systematically in his 2007-2008 lectures at the Coll\`ege de France on mean field 
games \cite{PLLionsCdF07}. This remarkable circle of ideas also appears in the recent work of Mischler, Mouhot and Wennberg (see \cite{MischlerMouhot}
and \cite{MischlerMouhotWenn}) on the mean field limit and on quantitative estimates on the propagation of chaos in classical statistical mechanics.

We introduce, for each $X,Y\in Q^N$, the notation
$$
d_{LP}(X,Y):=\inf\{\eps>0\hbox{ s.t. }\#\{k=1,\ldots,N\,|\,|x_k-y_k|>\eps\}<N\eps\}\,.
$$
This quantity is related to the Levy-Prokhorov distance $\Dist_{LP}$ on Borel probability measures on $Q$ in the following manner:
$$
\Dist_{LP}(\mu_X,\mu_Y)=\inf_{\si\in\fS_N}d_{LP}(X,Y_\si)\,,
$$
where $Y_\si:=(y_{\si(1)},\ldots,y_{\si(N)})$. We recall the definition of the Levy-Prokhorov distance $\Dist_{LP}$ on $\cP(Q)$:
$$
\Dist_{LP}(P_1,P_2):=\inf\left\{\eps>0\,|\,\inf_{\pi\in\Pi(P_1,P_2)}\iint_{Q\times Q}\indc_{|x-y|>\eps}\pi(dxdy)<\eps\right\}\,.
$$
We also recall that the Levy-Prokhorov distance metricizes the weak topology on $\cP(Q)$, so that $(\cP(Q),\Dist_{LP})$ is a compact metric space --- see for
instance \cite{BillingCvgProba}. 

Lions' key observation is summarized in the following theorem.

\begin{Thm}[P.-L. Lions \cite{PLLionsCdF07}]\index{Symmetric function of ``infinitely many variables''}
Let $Q$ be a compact metric space and consider, for each $N\ge 1$, a symmetric function $u_N\in C(Q^N)$. Assume that
$$
\sup_{N\ge 1}\sup_{(z_1,\ldots,z_N)\in Q^N}|u_N(z_1,\ldots,z_N)|<\infty\,,
$$
and that
$$
\sup\{|u_N(X)-u_N(Y)|\hbox{ s.t. }X,Y\in Q^N\hbox{ and }d_{LP}(X,Y)<\eps\}\to 0
$$
uniformly in $N\ge 1$ as $\eps\to 0$. Then there exists $U\in C(\cP(Q))$ and a subsequence $u_{N_k}$ of $u_N$ such that
$$
\sup_{Z\in Q^{N_k}}|u_{N_k}(Z)-U(\mu_{Z})|\to 0\hbox{ as }N_k\to\infty\,.
$$
\end{Thm}

\smallskip
This point of view is obviously dual of the Hewitt-Savage theorem quoted above. In fact, as noticed by P.-L. Lions, it leads to a very simple proof of the
Hewitt-Savage theorem. Lions' argument \cite{PLLionsCdF07} is sketched below.

Consider a sequence $(P_j)_{j\ge 1}$ of probability measures satisfying the compatibility condition (CC) above. Lions' idea is to consider
$$
C(\cP(Q))\ni U\mapsto L(U):=\lim_{N\to\infty}\int_{Q^N}U(\mu_{Z_N})P_N(dZ_N)\in\bR\,.
$$
That the limit on the right hand side of the equality above exists follows from considering the case of a monomial. Indeed, when $U=M_{k,\phi}$ with the 
notation used in the previous section, i.e.
$$
U(\mu)=M_{k,\phi}(\mu):=\int_{Q^k}\phi(x_1,\ldots,x_k)\mu^{\otimes k}(dx_1\ldots dx_k)
$$
where $\phi\in C(Q^k)$, a straightforward computation shows that 
$$
\int_{Q^N}M_{k,\phi}(\mu_{Z_N})P_N(dZ_N)=\int_{Q^k}\phi(z_1,\ldots,z_k)P_k(dz_1\ldots dz_k)+O(1/N)\,.
$$
Hence
$$
L(M_{k,\phi})=\int_{Q^k}\phi(z_1,\ldots,z_k)P_k(dz_1\ldots dz_k)
$$
for all $k\ge 1$ and all $\phi\in C(Q^k)$.

The set of polynomials, i.e. of linear combinations of monomials, is a subalgebra of $C(\cP(Q))$ since
$$
M_{k,\phi}(\mu)M_{l,\psi}(\mu)=M_{k+l,\phi\otimes\psi}(\mu)
$$
with 
$$
\phi\otimes\psi(x_1,\ldots,x_{k+l}):=\phi(x_1,\ldots,x_k)\psi(x_{k+1},\ldots,x_{k+l})\,.
$$
We also use the convention
$$
M_{0,1}=1\,.
$$
This subalgebra is dense in $C(\cP(Q))$ for the topology of uniform convergence by the Stone-Weierstrass theorem. 

Since 
$$
\left|\int_{Q^N}U(\mu_{Z_N})P_N(dZ_N)\right|\le\sup_{\mu\in\cP(Q)}|U(\mu)|\,,
$$
for each\footnote{Since $(\cP(Q),\Dist_{LP})$ is compact, any element $U$ of $C(\cP(Q))$ is bounded on $\cP(Q)$, so that
$$
\sup_{\mu\in\cP(Q)}|U(\mu)|<\infty\,.
$$} $U\in C(\cP(Q))$, and since the limit 
$$
L(M_{k,\phi}):=\lim_{N\to\infty}\int_{Q^N}M_{k,\phi}(\mu_{Z_N})P_N(dZ_N)=\int_{Q^k}\phi(z_1,\ldots,z_k)P_k(dz_1\ldots dz_k)
$$
exists for each $k\ge 1$ and each $\phi\in C(Q^k)$, we conclude that this limit exists for each $U\ni C(\cP(Q))$. 

Since obviously $L\ge 0$ (being the limit of linear functionals defined by $P_N$ which is a positive measure), and since
$$
L(1)=\int_{Q^N}P_N(dZ_N)=1\quad\hbox{ for each }N\ge 1
$$
we conclude that the linear functional $L$ is represented by a unique probability measure $\pi\in\cP(\cP(Q))$, i.e.
$$
L(U)=\int_{\cP(Q)}U(\mu)\pi(d\mu)\,.
$$
Specializing to monomials
$$
\ba
L(M_{k,\phi})&=\int_{Q^k}\phi(z_1,\ldots,z_k)P_k(dz_1\ldots dz_k)=\int_{\cP(Q)}M_{k,\phi}(\mu)\pi(d\mu)
\\
&=\int_{\cP(Q)}\left(\int_{Q^k}\phi(z_1,\ldots,z_k)\mu^{\otimes k}(dz_1\ldots dz_k)\right)\pi(d\mu)
\\
&=\int_{Q^k}\phi(z_1,\ldots,z_k)\left(\int_{\cP(Q)}\mu^{\otimes k}(dz_1\ldots dz_k)\pi(d\mu)\right)
\ea
$$
and this means that
$$
P_k=\int_{\cP(Q)}\mu^{\otimes k}\pi(d\mu)
$$
with the unique probability measure $\pi$ defined above. This is precisely the representation formula in the Hewitt-Savage theorem.

\subsection{The case of singular interaction kernels}\index{Singular interaction kernels}

The method for proving the mean field limit presented above is based on Dobrushin's estimate and, as such, is limited to cases where the interaction kernel 
$k$ is Lipschitz continuous in both its arguments. This is most annoying since many interaction kernels of physical interest are singular on the diagonal. All
the examples presented in the first section (i.e. the Vlasov-Poisson system and the vorticity formulation of the Euler equation for incompressible fluids in two
space dimensions) involve the fundamental solution of the Laplacian, leading to interaction kernels that become singular as the distance between the two 
interacting particles vanishes. Obtaining rigorous derivations of both models as the mean field limit of large particle systems remains an important open
problem.

However, some remarkable results have been obtained in this direction.

In the case of the vorticity formulation of the two-dimensional Euler equation for incompressible fluids, the mean field limit for the dynamics of a large number 
of vortex centers is analogous to the convergence problem for vortex methods in the numerical analysis of the Euler equation. These methods approximate
the vorticity field $\om\equiv\om(t,x)$ as follows:
$$
\om(t,\cdot)\simeq\sum_{k=1}^N\om_k\de_{x_k(t)}\,,
$$
where $\om_k$ is the (constant) intensity of the vortex centered at $x_k(t)$. (Notice the slight difference with the mean-field limit discussed above, where
each vortex would have the same intensity $\om_k=1/N$.) The motion of the vortex centers $x_k(t)$ is governed by the following ODE system:
$$
\dot{x}_k(t)=\sum_{l=1\atop l\not=k}^N\om_lK_\eps(x_k(t)-x_l(t))\,,\quad k=1,\ldots,N\,,
$$
where $K_\eps$ is an approximation of the vortex interaction kernel
$$
K(x)=-\tfrac1{2\pi}\frac{Jx}{|x|^2}\,,
$$
(with $J$ standing for the rotation of an angle $-\tfrac{\pi}2$).

In one variant of these methods, called ``the vortex blob method''\index{Vortex blob method}, the interaction potential is truncated near the singularity at the 
origin so as to remain smooth as the distance between interacting vortices vanishes. The method based on Dobrushin's estimate presented above for proving 
the mean field limit applies without modification to the convergence of the vortex blob method. See chapter 5 in \cite{MarchioPulvi} for more details on this topic.

In another variant of these methods, called ``the vortex point method''\index{Vortex point method}, there is no regularization of the interaction kernel, i.e. one 
takes $K_\eps=K$ in the system of ODEs above governing the motion of vortices. The vorticity field $\om\equiv\om(t,x)$ is approximated as follows
$$
\om(t,\cdot)\simeq\om^h(t,\cdot):=\sum_{k\in\bZ^2}\om^h_k\de_{x^h_k(t)}\,,
$$
where $h>0$, and the vortex centers satisfy
$$
\left\{
\ba
{}&\dot{x}^h_k(t)=\sum_{l\in\bZ^2\atop l\not=k}^N\om^h_lK(x^h_l(t)-x^h_k(t))\,,
\\
&x^h_k(0)=hk\,,\quad k\in\bZ^2\,,
\ea
\right.
$$
while the vortex intensities are chosen so that
$$
\om^h_k=\om(0,hk)\,,\qquad k\in\bZ^2\,.
$$
The convergence of the vortex point method for initial data in the Schwartz class $\cS(\bR^2)$ of infinitely differentiable functions with rapidly decaying 
derivatives of all orders has been proved in \cite{GoodHou1} --- see also \cite{GoodHou2,Schochet96,Hauray09}. The key argument is a control of the
minimal distance between vortex centers in Proposition 1 of \cite{Hauray09}. 

One should also mention recent attempts to justify the derivation of the Vlasov-Poisson system as the mean field limit of the Liouville equation for a large
number of identical point particles with (unmollified) Coulomb interaction: see \cite{HaurayJabin07,HaurayJabin11}. These papers prove the mean field 
limit and the propagation of chaos for large systems of point particles with singular interaction force field of order $O(d^{-\a})$ for $\a\le 1$, where $d$
designates the distance between the two interacting particles.

\subsection{From particle systems to the Vlasov-Maxwell system}

There are several difficulties in adapting the method for proving the mean field limit presented above to the case of the Vlasov-Maxwell system; see 
\cite{GolseCMP12} Êfor a detailed discussion of this problem.

First, the source term in the system of Maxwell's equations is not a probability distribution --- as in the case of the Vlasov-Poisson system, where the
electric field $E\equiv E(t,x)$ is given by
$$
E(t,x)=-\grad_x\phi_f(t,x)\,,\quad\hbox{ with }-\Dlt_x\phi_f(t,x)=\rho_f(t,x)\,.
$$
(We recall that
$$
\rho_f(t,x):=\int_{\bR^3}f(t,x,v)dv
$$
so that $\rho_f(t,\cdot)$ is a probability distribution on $\bR^3$ if $f(t,\cdot,\cdot)$ is a probability distribution on $\bR^3\times\bR^3$, which can be
assumed without loss of generality since the total particle number is invariant under the Vlasov-Poisson dynamics.) 

In the case of the Vlasov-Maxwell system, the source term in Maxwell's equations is the 4-vector $(\rho_f,j_f)$, defined as follows:
$$
\rho_f(t,x):=\int_{\bR^3}f(t,x,\xi)d\xi\,,\qquad j_f(t,x):=\int_{\bR^3}v(\xi)f(t,x,\xi)d\xi\,.
$$

This difficulty is handled by an idea introduced in earlier works on Vlasov-Maxwell type systems \cite{BouGolPal04,BouGolPal03}. The idea is to 
represent the solution of Maxwell's equations for the electromagnetic field in terms of a \textit{single momentum distribution} of electromagnetic
potential\index{Distribution of electromagnetic potential}, as follows. Consider the Cauchy problem for the wave equation
$$
\left\{
\ba
{}&\Box_{t,x}u_f(t,x,\xi)=f(t,x,\xi)\,,
\\
&u_f\rstr_{t=0}=\d_tu_f\rstr_{t=0}=0\,,
\ea
\right.
$$
where $\Box_{t,x}:=\d_t^2-\Dlt_x$ is the d'Alembert operator in the variables $t$ and $x$. In other words, the momentum variable is a simple parameter
in the wave equation above. The self-consistent electromagnetic field in the Vlasov-Maxwell system is represented as
$$
\ba
E(t,x)=&-\d_t\int_{\bR^d}v(\xi)u_f(t,x,\xi)d\xi-\grad_x\int_{\bR^d}u_f(t,x,\xi)d\xi
\\
&-\d_tA_0(t,x)-\grad_x\phi_0(t,x)\,,
\\
B(t,x)=&\Rot_x\int_{\bR^d}v(\xi)u_f(t,x,\xi)d\xi+\Rot_xA_0(t,x)\,,
\ea
$$
where $\phi_0\equiv\phi_0(t,x)\in\bR$ and $A_0\equiv A_0(t,x)\in\bR^3$ are respectively a scalar and a vector potential satisfying
$$
\Box_{t,x}\phi_0=0\,,\qquad\hbox{Êand }\Box_{t,x}A_0=0\,,
$$
together with appropriate initial conditions so that the formulas above for $E$ and $B$ match the prescribed initial conditions in the Cauchy problem
for the Vlasov-Maxwell system.

With this representation for the electromagnetic field in the Vlasov-Maxwell system, the source term in the field equation is now $f$ itself, a probability
distribution  in the single particle phase space $\bR^3_x\times\bR^3_\xi$ whenever $f\rstr_{t=0}$ is a probability distribution, since the integral
$$
\iint_{\bR^3\times\bR^3}f(t,x,\xi)dxd\xi
$$
is an invariant of the dynamics defined by the Vlasov-Maxwell equations.

A second difficulty is that the solution of the Cauchy problem for the wave equation defining $u_f$ involves the space-time convolution with Kirchoff's
kernel (the forward fundamental solution of the d'Alembert operator): see for instance \cite{ShaStru}. This formula is equivalent to the retarded potential 
formula for solving the system of Maxwell's equations. At variance with the solution of the Poisson equation
$$
-\Dlt_x\phi_f(t,x)=\rho_f(t,x)\,,
$$
the formula giving the solution $u_f$ of the wave equation above in terms of $f$ is not local in the time variable $t$. Physically, this is due to the fact that 
the electromagnetic field is propagated at the speed of light $c>0$ in the Vlasov-Maxwell system, while the electrostatic field in the Vlasov-Poisson 
system is propagated instantaneously --- in other words, the speed of light is considered as infinite in the Vlasov-Poisson system. It remains to check that 
this difference in structure between the Vlasov-Poisson and the Vlasov-Maxwell systems does not rule out the possibility of an estimate 
\textit{\`a la} Dobrushin. 

This second difficulty was partially handled in an earlier work \cite{ElsKieRicCMP09} for the simpler\index{Vlasov-d'Alembert system} Vlasov-d'Alembert 
system
$$
\left\{
\ba
{}&(\d_t+v(\xi)\cdot\grad_x)f(t,x,\xi)-\grad_x\phi_f(t,x)\cdot\grad_\xi f(t,x,\xi)=0\,,
\\
&\Box_{t,x,}\phi_f(t,x)=\rho_f(t,x)\,,
\ea
\right.
$$
where $\rho_f$ is defined in terms of $f$ as above, and $v(\xi)=\grad_\xi\sqrt{1+|\xi|^2}$ as in the Vlasov-Maxwell system.

A third difficulty --- albeit a less essential one --- is to choose a regularization procedure for the interaction potential that does not destroy the delicate
invariance properties of the Vlasov-Maxwell system. Let $\chi_\eps\equiv \chi_\eps(x)$ be a regularizing sequence in $\bR^3$, chosen so that the
function $\chi_\eps$ is even for each $\eps>0$. The Vlasov-Maxwell system is regularized by replacing the momentum distribution of electromagnetic
potential $u_f$ with the solution of the Cauchy problem
$$
\left\{
\ba
{}&\Box_{t,x}u_f^\eps(t,x,\xi)=\chi_\eps\star_x\chi_\eps\star_xf(t,x,\xi)\,,
\\
&u_f\rstr_{t=0}=\d_tu_f\rstr_{t=0}=0\,,
\ea
\right.
$$
where $\star_x$ designates the convolution product in the $x$ variable. This regularization procedure, originally due to E. Horst, is such that both the
conservation of the total particle number and the conservation of some variant of the total energy are satisfied by the mollified system (see \cite{ReinCMS}).

The interested reader is referred to \cite{GolseCMP12} for a complete discussion of the material presented in this section.

\section[Mean field limits for quantum models]{The mean field problem in quantum mechanics}

\subsection{The $N$-body problem in quantum mechanics}

At variance with classical mechanics, the $N$-body problem in quantum mechanics is a PDE, and not a system of ODEs. The state at time $t$ of a system of 
$N$ identical point particles is defined by its $N$-body wave function\index{$N$-particle wave function}
$$
\Psi_N\equiv\Psi_N(t,x_1,\ldots,x_N)\in\bC\,,\qquad x_1,\ldots,x_N\in\bR^d\,.
$$
We assume that the reader is more or less familiar with the formalism of quantum mechanics, and we shall not attempt to recall more than a few basic facts.
An excellent introduction to quantum mechanics can be found in \cite{BasdDali}.

The meaning of the wave function is such that $|\Psi(x_1,\ldots,x_N)|^2$ is the (joint) probability density of having particle $1$ at the position $x_1$, particle 
$2$ at the position $x_2$,\dots,and particle $N$ at the position $x_N$ at time $t$. This implies the normalization\footnote{This normalization condition is not
satisfied by ``generalized eigenfunctions'' of an operator with continuous spectrum. Consider the two following examples, where $\fH=L^2(\bR)$. 

\smallskip
\noindent
(a) Let $\cH=-\tfrac12\frac{d^2}{dx^2}+\tfrac12x^2$ (the quantum harmonic oscillator), which has discrete spectrum only. The sequence of eigenvalues of
$\cH$ is $n+\tfrac12$ with $n\in\bN$. Besides $\Ker(H-(n+\tfrac12)I)=\bC h_n$ for each $n\in\bN$, with
$$
h_n(x):=\frac{1}{\sqrt{2^nn!}\pi^{1/4}}e^{-x^2/2}H_n(x)\,,\quad\hbox{ where }H_n(x)=(-1)^ne^{x^2}\frac{d^n}{dx^n}e^{-x^2}\,.
$$
The function $H_n$ is the $n$th Hermite polynomial, and the sequence $(h_n)_{n\ge 0}$ is a Hilbert basis of $\fH$. In particular one has the orthonormality
condition
$$
\int_\bR h_n(x)h_m(x)dx=\de_{mn}\,,\quad\hbox{ for all }m,n\ge 0
$$
where $\de_{mn}$ is the Kronecker symbol (i.e. $\de_{mn}=0$ if $m\not=n$ and $\de_{mn}=1$ if $m=n$).

\smallskip
\noindent
(b) Let $P=-i\frac{d}{dx}$ (the momentum operator), which has continuous spectrum only. The spectrum of $P$ is the real line $\bR$, and the generalized 
eigenfunctions of $P$ are the functions $e_k:\,x\mapsto e_k(x):=e^{i2\pi kx}$. For each $k\in\bR$, one has $Pe_k=2\pi ke_k$ but $e_k\notin\fH$. However 
one has the ``formula'' analogous to the orthonormality condition in case (a):
$$
\int_{\bR}\overline{e_k(x)}e_l(x)dx=\de_0(k-l)
$$
where $\de_0$ is the Dirac mass at the origin. The integrand on the left hand side of the equality above is not an element of $L^1(\bR)$, and the integral is 
not a Lebesgue integral. The identity above should be understood as the Fourier inversion formula on the class $\cS'(\bR)$ of tempered distributions on the
real line $\bR$.

In the case (a), if the system is in an eigenstate corresponding with the eigenvalue $n+\tfrac12$ of the operator $\cH$, its wave function is of the form
$\psi=\om h_n$ with $|\om|=1$, so that $\|\psi\|_{\fH}=1$.

In the case (b), if the system is in an eigenstate corresponding with the element $k$ of the spectrum of the operator $P$, it cannot be described by any
wave function in $\fH$, but only by a generalized eigenfunction of $P$, that does not belong to $\fH$.

In the discussion below, we shall never  consider quantum states described by generalized eigenfunctions as in (b), but only quantum states corresponding 
with normalized wave functions.}
$$
\int_{(\bR^d)^N}|\Psi_N(t,x_1,\ldots,x_N)|^2dx_1\ldots dx_N=1\,,
$$
that is satisfied for all $t\in\bR$.

We assume that the interaction between a particle at position $x$ and a particle at position $y$ is given by a $2$-body potential $V(x-y)\in\bR$. Henceforth,
we assume that $V$ is even, so that the force exerted by the particle at position $y$ on the particle at position $x$, i.e. $-\grad V(x-y)$, exactly balances the
force exerted by the particle at position $x$ on the particle at position $y$, i.e. $-\grad V(y-x)$. (Indeed, $\grad V$ is odd since $V$ is even.)

With these data, we can write the Schr\"odinger equation governing the $N$-body wave function for a system of $N$ identical particles of mass $m$ with
$2$-body interaction given by the potential $V$\index{$N$-particle Schr\"odinger equation}:
$$
i\hbar\d_t\Psi_N=-\tfrac{\hbar^2}{2m}\sum_{k=1}^N\Dlt_{x_k}\Psi_N+\sum_{1\le k<l\le N}V(x_k-x_l)\Psi_N\,,
$$
where $\hbar$ is the reduced Planck constant.

The $N$-body Schr\"odinger equation above is a PDE which is the analogue in quantum mechanics of Newton's second law
$$
\left\{
\ba
{}&m\dot{x}_k=\xi_k\,,
\\
&\dot{\xi}_k=-\sum_{l=1\atop l\not=k}^N\grad V(x_k-x_l)\,,\qquad k=1,\ldots,N
\ea
\right.
$$
in classical mechanics.

The question of existence and uniqueness of a solution of the Schr\"odinger equation of $N$-particle quantum dynamics is settled by the following
remarkable result. (Without loss of generality, we assume that $\hbar^2/m=1$.)

\begin{Thm} [Kato] \lb{T-Kato}\index{Kato's theorem}
Assume that the space dimension is $d=3$. If, for some $R>0$, 
$$
V\rstr_{B(0,R)}\in L^2(B(0,R))\qquad\hbox{ and }\quad V\rstr_{\bR^3\setminus B(0,R)}\in L^\infty(\bR^3\setminus B(0,R))\,,
$$
then, for all $N\ge 1$, the unbounded operator
$$
\cH_N:=-\tfrac12\sum_{k=1}^N\Dlt_{x_k}+\sum_{1\le k<l\le N}V(x_k-x_l)
$$
has a self-adjoint extension on $\fH_N:=L^2((\bR^3)^N)$ and generates a unitary group $e^{-it\cH_N}$ on $\fH_N$.
\end{Thm}

\smallskip
See chapter V, \S 5.3 in \cite{KatoBk} and \cite{Kato51} for a proof of this result.

\smallskip
This result is the analogue in quantum mechanics of the existence and uniqueness of the solution of the Cauchy problem for the motion equations in classical 
mechanics that follow from the Cauchy-Lipschitz theorem. Observe that applying the Cauchy-Lipschitz theorem to the system of ODEs resulting from Newton's 
second law applied to each particle would require $V$ to have Lipschitz continuous first order derivatives. The assumptions on the regularity of the potential in 
Kato's result are obviously much less stringent --- for instance, the Coulomb potential $V(z)=C/|z|$ obviously satisfies these assumptions. (That the Coulomb 
potential satisfies the assumptions in Kato's theorem is of course very satisfying since most computations in atomic physics involve the Coulomb interaction 
between the electrons and the nuclei in atoms and molecules.) 

\smallskip
As in the case of classical mechanics, our goal is to study the behavior of $\Psi_N(t)=e^{-it\cH_N}\Psi_N^{in}$ in the limit as $N\to\infty$, under appropriate scaling assumptions involving the interaction potential $V$ and the particle number $N$.

First we need to define the analogue of the mean field scaling used in the case of classical mechanics. Our argument for choosing this scaling is based on 
considering the energy of the particle system.

\smallskip
In quantum mechanics, the energy of a system of $N$ identical point particles of mass $m$ with pairwise interaction described in terms of the real-valued
$2$-body potential $V$ is 
$$
\ba
\sum_{k=1}^N\int_{(\bR^d)^N}&\tfrac{\hbar^2}{2m}|\grad_{x_k}\Psi_N(x_1,\ldots,x_N)|^2dx_1\ldots,dx_N
\\
&+
\sum_{1\le k<l\le N}\int_{(\bR^d)^N}V(x_k-x_l)|\Psi_N(x_1,\ldots,x_N)|^2dx_1\ldots,dx_N\,,
\ea
$$
where the first term is the kinetic energy, while the second term is the potential energy.

Pick a typical length scale $L$ and energy scale $E$ in the system of $N$ particles under consideration. We introduce the dimensionless space variables
and potential as follows\index{Mean field scaling (quantum case)}:
$$
\hat x_k=\frac{x_k}{L}\,,\qquad\hat V(\hat x_k-\hat x_l)=\frac1E V(x_k-x_l)\,.
$$
Likewise, the $N$-particle wave function is scaled as 
$$
\hat\Psi_N(\hat t,\hat x_1,\ldots,\hat x_N)=L^{dN/2}\Psi_N(t,x_1,\ldots,x_N)\,,
$$
where $T$ is a time scale to be defined later, and where
$$
\hat t=\frac{t}{T}
$$
is the dimensionless time variable. 

Observe that 
$$
\int|\hat\Psi_N(\hat t,\hat x_1,\ldots,\hat x_N)|^2d\hat x_1\ldots d\hat x_N
=
\int|\Psi_N(t,x_1,\ldots,x_N)|^2dx_1\ldots dx_N=1
$$
with this scaling.

With these dimensionless quantities, the kinetic and potential energies become respectively
$$
\ba
\hbox{Kinetic energy}&=\sum_{k=1}^N\int_{(\bR^d)^N}\tfrac{\hbar^2}{2mL^2}|\grad_{\hat x_k}\hat\Psi_N|^2d\hat x_1\ldots d\hat x_N\,,
\\
\hbox{Potential energy}&=\sum_{1\le k<l\le N}\int_{(\bR^d)^N}E\hat V(\hat x_k-\hat x_l)|\hat\Psi_N|^2d\hat x_1\ldots d\hat x_N\,.
\ea
$$

Observe that there are $N$ terms in the kinetic energy, while the potential energy involves $\tfrac12N(N-1)$ terms. Therefore, we scale the system so that 
the kinetic energy and the potential energy are of same order of magnitude, by choosing $E$ and $L$ so that
$$
\frac{\hbar^2}{mL^2}=NE\,.
$$
Thus
$$
\ba
-\tfrac{\hbar^2}{2m}&\sum_{k=1}^N\Dlt_{x_k}\Psi_N+\sum_{1\le k<l\le N}V(x_k-x_l)\Psi_N
\\
&=
L^{-dN/2}\frac{\hbar^2}{mL^2}\left(-\tfrac12\sum_{k=1}^N\Dlt_{\hat x_k}\hat\Psi_N+\frac1N\sum_{1\le k<l\le N}\hat V(\hat x_k-\hat x_l)\hat\Psi_N\right)\,.
\ea
$$

At this point, we set
$$
\hat H_N\hat\Psi_N:=-\tfrac12\sum_{k=1}^N\Dlt_{\hat x_k}\hat\Psi_N+\frac1N\sum_{1\le k<l\le N}\hat V(\hat x_k-\hat x_l)\hat\Psi_N\,,
$$
and define the time scale $T$ as follows:
$$
T=\frac{mL^2}{\hbar}\,.
$$

With the rescaled time variable $\hat t=t/T$, the $N$-body Schr\"odinger equation in mean-field scaling becomes
$$
i\d_{\hat t}\hat\Psi_N=-\tfrac12\sum_{k=1}^N\Dlt_{\hat x_k}\hat\Psi_N+\frac1N\sum_{1\le k<l\le N}\hat V(\hat x_k-\hat x_l)\hat\Psi_N\,.
$$

Henceforth we drop all hats on variables, and consider as our starting point the scaled equation
$$
 i\d_t\Psi_N=-\tfrac12\sum_{k=1}^N\Dlt_{x_k}\Psi_N+\frac1N\sum_{1\le k<l\le N}V(x_k-x_l)\Psi_N\,.
$$

A formal computation based on the fact that $V$ is real-valued shows that\index{Global conservation of mass (quantum case)}
$$
\frac{d}{dt}\int_{(\bR^d)^N}|\Psi_N(t,x_1,\ldots,x_N)|^2dx_1\ldots dx_N=0\,,
$$
so that
$$
\int_{(\bR^d)^N}|\Psi_N(t,x_1,\ldots,x_N)|^2dx_1\ldots dx_N=1\,.
$$
for all $t\in\bR$.

The rigorous argument is based on Kato's theorem (Theorem \ref{T-Kato}) stated above in space dimension $d=3$: with the notation
$$
H_N:=-\tfrac12\sum_{k=1}^N\Dlt_{x_k}+\frac1N\sum_{1\le k<l\le N}V(x_k-x_l)
$$
the one-parameter group $e^{-itH_N}$ is unitary on $\fH_N=L^2((\bR^3)^N)$, so that
$$
\int_{(\bR^3)^N}|\Psi_N(t,x_1,\ldots,x_N)|^2dx_1\ldots dx_N=\|e^{-itH_N}\Psi_N^{in}\|_{\fH_N}^2=1
$$
for all $t\in\bR$.

\subsection{The target mean-field equation}

The target quantum mean-field equation is the quantum analogue of the Vlasov-Poisson equation. 

Its unknown is the single-particle wave function $\psi\equiv\psi(t,x)\in\bC$, with $x\in\bR^d$

The basic idea in the mean field approximation is the same as in the classical case. Since the probability of finding a particle in an infinitesimal volume 
element $dx$ at time $t$ is $|\psi(t,x)|^2dx$, one expects that the action on the $k$-th particle located at $x_k$ is defined in terms of the potential $V$ by
the formula
$$
\frac1N\sum_{l=1\atop k\not=l}^NV(x_k-x_l)\sim\int_{\bR^d}V(x_k-z)|\psi(t,z)|^2dz\,.
$$

This suggests that the target mean-field equation for single-particle wave function obtained as the limit\index{Hartree's equation} of the $N$-body Schr\"odinger 
equation is 
$$
i\d_t\psi(t,x)=-\tfrac12\Dlt_x\psi(t,x)+\psi(t,x)\int_{\bR^d}V(x-y)|\psi(t,y)|^2dy\,.
$$
This equation is known as the Hartree equation --- and also sometimes referred to as the\index{Schr\"odinger-Poisson equation} Schr\"odinger-Poisson 
equation when $V(z)$ is the Coulomb potential in space dimension $3$. Indeed, if 
$$
V(z)=\frac1{4\pi|z|}\,,\qquad z\in\bR^3\setminus\{0\}\,,
$$
the equation above is equivalent to the system
$$
\left\{
\ba
{}&i\d_t\psi(t,x)=-\tfrac12\Dlt_x\psi(t,x)+U(t,x)\psi(t,x)\,,
\\
&-\Dlt_xU(t,x)=|\psi(t,x)|^2\,.
\ea
\right.
$$
This system is the quantum analogue of the Vlasov-Poisson system.

The Hartree equation has been studied extensively by various authors: see \cite{Bove1,GinVeloHartree}. The following statement is a consequence of
Theorem 3.1 in \cite{GinVeloHartree}.

\begin{Prop}\lb{P-ExUnHartree}
Assume that $V$ is a real-valued, even element of $L^\infty(\bR^d)$. For each $\psi^{in}\in H^k(\bR^d)$, there exists a unique mild solution\footnote{A mild
solution of the Cauchy problem
$$
\left\{
\ba
{}&\dot{u}(t)=Au(t)+F[u(t)]\,,
\\
&u\rstr_{t=0}=u^{in}\,,
\ea
\right.
$$
where $A$ generates a strongly continuous contraction semigroup on a Banach space $X$ and $F:\,X\to X$ is a continuous map, is an $X$-valued continuous
function $I\ni t\mapsto u(t)\in X$ defined on $I=[0,\tau]$ with $\tau\in[0,+\infty]$ that is a solution of the integral equation 
$$
u(t)=e^{tA}u^{in}+\int_0^te^{(t-s)A}F[u(s)]ds
$$
for each $t\in I$.} of the Cauchy problem for Hartree's equation
$$
\left\{
\ba
{}&i\d_t\psi(t,x)=-\tfrac12\Dlt_x\psi(t,x)+(V\star_x|\psi|^2)\psi(t,x)\,,\quad x\in\bR^d
\\
\psi\rstr_{t=0}=\psi^{in}
\ea
\right.
$$
such that $\psi\in C_b(\bR;H^1(\bR^d))\cap C(\bR;H^k(\bR^2))$. Besides, this solution satisfies the conservation laws of mass and energy, viz.
$$
\|\psi(t,\cdot)\|_{L^2(\bR^d)}=\hbox{Const.}
$$
and
$$
\tfrac12\|\grad_x\psi(t,\cdot)\|^2_{L^2(\bR^d)}+\tfrac12\iint_{\bR^d\times\bR^d}V(x-y)|\psi(t,x)|^2|\psi(t,y)|^2dxdy=\hbox{Const.}
$$
for all $t\in\bR$.
\end{Prop}

\subsection{The formalism of density matrices}

Before studying the $N$-particle Schr\"odinger equation in the large $N$ limit, we need another formulation of the quantum $N$-body problem.

Henceforth we denote by $\la\cdot|\cdot\ra$ (or by $\la\cdot|\cdot\ra_{\fH_N}$ to avoid ambiguity if needed) the inner product of the Hilbert space 
$\fH_N:=L^2((\bR^d)^N)$, and by $|\cdot|$ or $|\cdot|_{\fH_N}$ the associated norm, defined by the formula
$$
|\Phi|_{\fH_N}:=\la\Phi|\Phi\ra_{\fH_N}^{1/2}\,.
$$

To the $N$-particle wave function $\Psi_N\equiv\Psi_N(x_1,\ldots,x_N)$ that is an element of $\fH_N:=L^2((\bR^d)^N)$, one associates the bounded linear 
operator $D_N$ on $\fH$  defined as follows:
$$
D_N:=\hbox{ orthogonal projection on }\bC\Psi_N\hbox{ in }\fH_N\,,
$$
i.e.
$$
D_N\Phi_N=\Psi_N\int_{(\bR^d)^N}\overline{\Psi_N}\Phi_N(x_1,\ldots,x_N)dx_1\ldots dx_N=\la\Psi_N|\Phi_N\ra\Psi_N\,.
$$
This linear operator is called the ``density matrix'' of the $N$-particle system\index{Density matrix}.

Whenever convenient, we use the notation common in the physics literature
$$
D_N:=|\Psi_N\ra\la\Psi_N|
$$
which is the equivalent of 
$$
D_N:=\Psi_N\otimes L_{\Psi_N}\,,\quad\hbox{ with }L_{\Psi_N}:=\la\Psi_N|\cdot\ra\in\fH_N^*\,.
$$
Obviously, $D_N$ is an integral operator, whose integral kernel is
$$
D_N(x_1,\ldots,x_N,y_1,\ldots,y_N)=\Psi_N(x_1,\ldots,x_N)\overline{\Psi_N(y_1,\ldots,y_N)}\,.
$$

In the sequel, we shall sometimes abuse the notation and use the same letter to designate an integral operator and its integral kernel.

Being an orthogonal projection in $\fH_N$, the operator $D_N$ is self-adjoint and nonnegative:
$$
D_N=D_N^*=D_N^2\ge 0\,.
$$
This property is the quantum analogue of the positivity of the $N$-particle distribution function in classical statistical mechanics.

The density matrix $D_N$ for a system of identical, indistinguishable particles satisfies the following\index{Indistinguishable particles}
symmetry: for all $\si\in\fS_N$, one has
$$
S_\si D_N=D_N\,,
$$
where $S_\si$ is the transformation on integral kernels defined by the formula
$$
S_\si D_N(x_1,\ldots,x_N,y_1,\ldots,y_N)\!:=\!D_N(x_{\si^{-1}\!(1)},\ldots,x_{\si^{-1}\!(N)},y_{\si^{-1}\!(1)},\ldots,y_{\si^{-1}\!(N)})\,.
$$

This assumption is not to be confused with the condition
$$
U_\si\Psi_N=\Psi_N\,,
$$
where $U_\si$ is the unitary operator defined on $\fH_N$ by the formula
$$
U_\si\Psi_N(x_1,\ldots,x_N):=\Psi_N(x_{\si^{-1}(1)},\ldots,x_{\si^{-1}(N)})\,.
$$
This assumption is an assumption on the \textit{statistics} of particles, i.e. on whether the particles under consideration follow the Bose-Einstein or the 
Fermi-Dirac statistics. Particles following the Bose-Einstein statistics are called bosons, and the wave function of any system of $N$ identical bosons
satisfies the condition\index{Bose-Einstein statistics}\index{Bosons}
$$
U_\si\Psi_N=\Psi_N\,.
$$
Particles following the Fermi-Dirac statistics are called fermions, and the wave function of any system of $N$ identical fermions satisfies the condition
\index{Fermi-Dirac statistics}\index{Fermions}
$$
U_\si\Psi_N=(-1)^{\Sign(\si)}\Psi_N\,.
$$

\smallskip
\noindent
\textbf{Exercise:} Express $U_{\si\si'}$ and $U^*_\si$ in terms of $U_\si$ and $U_{\si'}$ for all $\si,\si'\in\fS_N$, and check that $U_\si$ is a unitary operator
on $\fH_N$. Express the transformation $S_\si$ in terms of $U_\si$ for all $\si\in\fS_N$.

\smallskip
In fact, whenever
$$
U_\si\Psi_N=\om(\si)\Psi_N\quad\hbox{ with }\om(\si)\in\bC\hbox{ and }|\om(\si)|=1
$$
one has
$$
S_\si D_N=D_N\,.
$$

\smallskip
In particular, the density matrix of a system of $N$ identical bosons ($\om(\si)=1$) or $N$ identical fermions ($\om(\si)=(-1)^{\Sign(\si)}$) satisfies
$$
S_\si D_N=D_N\,.
$$

\smallskip
\noindent
\textbf{Exercise:} Assume that, for some $\Psi_N\in\fH_N$ such that $|\Psi_N|=1$ and for all $\si\in\fS_N$, one has
$$
U_\si\Psi_N=\om(\si)\Psi_N\quad\hbox{ with }\om(\si)\in\bC\hbox{ and }|\om(\si)|=1\,.
$$
Compute $\om(\si\si')$ in terms of $\om(\si)$ and $\om(\si')$ for all $\si,\si'\in\fS_N$. Prove that $\om(\si)=\pm1$ for all $\si\in\fS_N$.

\smallskip
Assuming that the time-dependent wave function $\Psi_N(t,\cdot)$ of a system of $N$ identical particles satisfies the $N$-body Schr\"odinger equation 
with initial data $\Psi_N^{in}$, i.e. 
$$
\left\{
\ba
{}&i\d_t\Psi_N=H_N\Psi_N:=-\tfrac12\sum_{k=1}^N\Dlt_{x_k}\Psi_N+\frac1N\sum_{1\le k<l\le N}V(x_k-x_l)\Psi_N\,,
\\
&\Psi_N\rstr_{t=0}=\Psi_N^{in}\,,
\ea
\right.
$$
our next task is to find the equation satisfied by $D_N:=|\Psi_N\ra\la\Psi_N|$.

Since $e^{-itH_N}$ is a unitary operator on $\fH_N$ provided that the interaction potential $V$ satisfies the assumptions of Kato's theorem above,
one has
$$
\ba
D_N(t)&=|e^{-itH_N}\Psi_N^{in}\ra\la e^{-itH_N}\Psi_N^{in}|
\\
&=e^{-itH_N}D_N^{in}(e^{-itH_N})^*=e^{-itH_N}D_N^{in}e^{itH_N}\,,
\ea
$$
where
$$
D_N^{in}:=|\Psi_N^{in}\ra\la\Psi_N^{in}|\,.
$$

It is instructive to do the analogous computation in finite dimension (i.e. with matrices). Assume that $A,B\in M_n(\bC)$, and set
$$
M(t):=e^{tA}Be^{-tA}\,.
$$ 
Then the function $\bR\ni t\mapsto M(t)\in M_n(\bC)$ is obviously of class $C^\infty$ (and even real analytic) and satisfies\footnote{For linear operators
$A,B$ on the Hilbert space $\fH$, we designate by $[A,B]$ their commutator $[A,B]:=AB-BA$. There is obviously a difficulty with the domain of $[A,B]$
viewed as an unbounded operator on $\fH$ if $A$ and $B$ are unbounded operators on $\fH$. This difficulty will be deliberately left aside, as we shall
mostly consider $B\mapsto[A,B]$ as an unbounded operator on $\cL(\fH)$, which is different matter.}
$$
\ba
\frac{d}{dt}M(t)=\frac{d}{dt}(e^{tA}Be^{-tA})&=Ae^{tA}Be^{-tA}-e^{tA}BAe^{-tA}
\\
&=Ae^{tA}Be^{-tA}-e^{tA}Be^{-tA}A
\\
&=AM(t)-M(t)A=[A,M(t)]\,.
\ea
$$
since $A$ commutes with $e^{-tA}$. 

In view of this elementary computation, we leave it to the reader to verify the following statement.

\begin{Prop}
Assume that the potential $V$ satisfies the assumptions in Kato's theorem (Theorem \ref{T-Kato}), and let $\Psi_N^{in}\in\fH_N=L^2((\bR^d)^N)$. 

Let $\Psi_N\equiv\Psi_N(t,x_1,\ldots,x_N)$ be the solution of the $N$-particle Schr\"odinger equation
$$
\left\{
\ba
{}&i\d_t\Psi_N=-\tfrac12\sum_{k=1}^N\Dlt_{x_k}\Psi_N\!+\!\!\!\sum_{1\le k<l\le N}\!\!\!V(x_k\!-\!x_l)\Psi_N\,,\quad x_1,\ldots,x_N\in\bR^d\,,\,\,t\in\bR\,,
\\
&\Psi_N\rstr_{t=0}=\Psi_N^{in}\,,
\ea
\right.
$$
i.e.
$$
\Psi_N(t,\cdot)=e^{-itH_N}\Psi_N^{in}\,,
$$
with scaled Hamiltonian
$$
H_N:=-\tfrac12\sum_{k=1}^N\Dlt_k+\frac1N\sum_{1\le k<l\le N}V_{kl}\,,
$$
where $\Dlt_k$ designates the Laplacian in the $k$th variable, while $V_{kl}$ is the operator defined as the multiplication by $V(x_k-x_l)$.

Then the density matrix
$$
D_N(t):=|\Psi_N(t,\cdot)\ra\la\Psi_N(t,\cdot)|=e^{-itH_N}D_N(0)e^{itH_N}
$$
satisfies the von Neumann equation\index{Von Neumann equation}
$$
i\dot D_N(t)=[H_N,D_N(t)]=-\tfrac12\sum_{k=1}^N[\Dlt_k,D_N(t)]+\frac1N\sum_{1\le k<l\le N}[V_{kl},D_N]\,.
$$

In terms of integral kernels
$$
\ba
i\d_tD_N(t,x_1,\ldots,x_N,y_1,\ldots,y_N)&
\\
=-\tfrac12\sum_{k=1}^N(\Dlt_{x_k}-\Dlt_{y_k})D_N(t,x_1,\ldots,x_N,y_1,\ldots,y_N)&
\\
+\frac1N\sum_{1\le k<l\le N}(V(x_k-x_l)-V(y_k-y_l))D_N(t,x_1,\ldots,x_N,y_1,\ldots,y_N)&\,.
\ea
$$
\end{Prop}

The fundamental properties of the density matrix are propagated under the flow associated to the von Neumann equation.

\begin{Prop}
Assume that the potential $V$ satisfies the assumptions in Kato's theorem, and let 
$$
D_N^{in}=(D_N^{in})^*=(D_N^{in})^2\ge 0\,.
$$
Then for each $t\in\bR$, one has
$$
D_N(t):=e^{-itH_N}D_N^{in}e^{itH_N}=D_N(t)^*=D_N(t)^2\ge 0\,,
$$
with scaled Hamiltonian
$$
H_N:=-\tfrac12\sum_{k=1}^N\Dlt_k+\frac1N\sum_{1\le k<l\le N}V_{kl}
$$
where $\Dlt_k$ designates the Laplacian in the $k$th variable, while $V_{kl}$ is the operator defined as the multiplication by $V(x_k-x_l)$.

Likewise, if particles are indistinguishable initially, i.e. if
$$
S_\si D_N^{in}=D_N^{in}\,,
$$
then
$$
S_\si D_N(t)=D_N(t)
$$
for all $t\in\bR$.
\end{Prop}

\begin{proof}
First
$$
D_N(t)^*=(e^{-itH_N}D_N^{in}e^{itH_N})^*=e^{-itH_N}(D_N^{in})^*e^{itH_N}
$$
since $H_N$ is self adjoint, so that $(e^{itH_N})^*=e^{-itH_N}$. Then
$$
\ba
D_N(t)^2&=(e^{-itH_N}D_N^{in}e^{itH_N})^2
\\
&=e^{-itH_N}(D_N^{in})^2e^{itH_N}=(e^{-itH_N}D_N^{in}e^{itH_N})^2=D_N(t)\,.
\ea
$$
That $S_\si D_N(t)=D_N(t)$ for all $t\in\bR$ if $S_\si D_N^{in}=D_N^{in}$ is a straightforward consequence of the following lemma.

\begin{Lem}\lb{L-Sym}
Let $U_\si$ be the unitary operator defined on $\fH_N$ for each $\si\in\fS_N$ by
$$
U_\si\Psi_N(x_1,\ldots,x_N):=\Psi_N(x_{\si^{-1}(1)},\ldots,x_{\si^{-1}(N)})\,.
$$
Then
$$
U_\si e^{itH_N}=e^{it H_N}U_\si
$$
for each $t\in\bR$ and each $\si\in\fS_N$.
\end{Lem}

\smallskip
Thus, for each $t\in\bR$ and each $\si\in\fS_N$, one has
$$
\ba
S_\si D_N(t)&=U_\si D_N(t)U_\si^*=U_\si e^{-itH_N}D_N^{in}e^{it H_N}U_\si^*
\\
&=e^{-itH_N}U_\si D_N^{in}U_\si^*e^{it H_N}=e^{-itH_N}D_N^{in}e^{it H_N}=D_N(t)\,,
\ea
$$
provided that
$$
S_\si D_N^{in}=U_\si D_N^{in}U_\si^*=D_N^{in}\,.
$$
\end{proof}

\begin{proof}[Proof of Lemma \ref{L-Sym}]
If $\Phi\in C^\infty_c((\bR^d)^N)$, one has
$$
\Dlt_kU_\si\Phi=U_\si\Dlt_{\si(k)}\Phi
$$
and 
$$
V_{kl}U_\si\Phi=U_\si V_{\si(k),\si(l)}\Phi\,,
$$
so that
$$
H_NU_\si\Phi=U_\si H_N\Phi\,.
$$
Therefore
$$
U_\si e^{itH_N}=e^{it H_N}U_\si
$$
for each $t\in\bR$ and each $\si\in\fS_N$. 
\end{proof}

\section[Elements of operator theory]{Elements of operator theory}\lb{S-OPTH}

We recall that  $\fH$ designates a complex, separable Hilbert space $\fH$, with inner product denoted by $\la\cdot|\cdot\ra$ and norm denoted by $|\cdot|$. 
(To avoid ambiguity, we also use the notation $\la\cdot|\cdot\ra_\fH$ and $|\cdot|_\fH$ whenever needed.) We always assume that the inner product is 
antilinear in its first argument and linear in its second argument. The set of Hilbert basis of $\fH$ --- i.e. the set of orthonormal and total families in $\fH$ --- is 
denoted by $HB(\fH)$.

The set of bounded operators on $\fH$ is denoted by $\cL(\fH)$, with operator norm\index{Operator norm}
$$
\|A\|=\sup_{|x|=1}|Ax|\,.
$$
We recall that $\cL(\fH)$ endowed with the operator norm $\|\cdot\|$ is a Banach algebra --- and even a $C^*$-algebra, since 
$$
\|A^*A\|=\|A\|^2\quad\hbox{Êfor each }A\in\cL(\fH)\,.
$$

A bounded operator $A$ on $\fH$ is said to be compact\index{Compact operators} if $A(\overline{B(0,1)})$ is relatively compact in $\fH$. The set of 
compact operators on $\fH$ is denoted by $\cK(\fH)$; one easily checks that $\cK(\fH)$ is a closed (for the operator norm) two-sided ideal in $\cL(\fH)$.

We recall the following fundamental facts about compact operators:

\smallskip
\bu if $A$ is a compact operator on $\fH$, its adjoint $A^*$ is also a compact operator;

\smallskip
\bu $\cK(\fH)$ is the operator-norm closure of the set of finite rank operators on $\fH$.

\smallskip
Next we introduce the trace norm\index{Trace norm}: for each $A\in\cL(\fH)$
$$
\|A\|_1:=\sup_{e,f\in HB(\fH)}\sum_{k\ge 0}|\la Ae_k|f_k\ra|\in[0,\infty]\,.
$$
A bounded operator $A$ on $\fH$ is said to be a trace-class operator\index{Trace-class operators} if and only if it has finite trace norm. The set of trace-class 
operators on $\fH$ is denoted by $\cL^1(\fH)$; in other words,
$$
\cL^1(\fH):=\{A\in\cL(\fH)\hbox{ s.t. }\|A\|_1<\infty\}\,.
$$
One easily checks that $\cL^1(\fH)$ is a two-sided ideal in $\cL(\fH)$, that $A\mapsto\|A\|_1$ defines a norm on $\cL^1(\fH)$, and that $\cL^1(\fH)$ equipped
with the trace norm $\|\cdot\|_1$ is a separable Banach space. Besides, the map
$$
\cL^1(\fH)\ni A\mapsto A^*\in\cL^1(\fH)
$$ 
is an isometry for the trace norm. The following properties can be checked easily:

\smallskip
\noindent
(a) $\cL^1(\fH)\subset\cK(\fH)\subset\cL(\fH)$ with $\|A\|\le\|A\|_1$ for all $A\in\cL^1(\fH)$, and

\noindent
(b) for all $A\in\cL^1(\fH)$ and $B\in\cL(\fH)$
$$
\|AB\|_1\hbox{ and }\|BA\|_1\le\|A\|_1\|B\|\,,\quad A\in\cL^1(\fH)\hbox{ and }B\in\cL(\fH)\,.
$$

There is a natural notion of trace for operators in infinite dimensional spaces, that can be defined for each trace-class operator. For each $A\in\cL^1(\fH)$,
\index{Trace of a trace-class operator}
$$
\Tr(A):=\sum_{k\ge 0}\la Ae_k|e_k\ra\,,
$$
for each $(e_k)_{k\ge 0}\in HB(\fH)$. We leave it to the reader to check that the expression on the right hand side of the previous equality is independent of 
the choice of $(e_k)_{k\ge 0}\in HB(\fH)$. 

Perhaps the most important property of the trace is the following identity: for all $A\in\cL^1(\fH)$ and $B\in\cL(\fH)$, the operators $AB$ and $BA$ are both
trace-class operators on $\fH$, and one has
$$
\Tr(AB)=\Tr(BA)\,.
$$

Let us recall the notion of polar decomposition of an operator\index{Polar decomposition of an operator}. For each $A\in\cL(\fH)$, there exists $|A|,U\in\cL(\fH)$ 
such that
$$
UU^*=U^*U=I\,,\quad|A|=|A|^*\ge 0\hbox{ and }A=|A|U\,.
$$
In this decomposition
$$
|A|=\sqrt{AA^*}\,.
$$
Then
$$
A\in\cL^1(\fH)\Leftrightarrow|A|\in\cL^1(\fH)\hbox{ and }\|A\|_1=\Tr(|A|)\,.
$$
The polar decomposition is the analogue of polar coordinates in operator theory: $|A|$ is the analogue of $|z|$ for $z\in\bC$, while $U$ is the analogue of 
$z/|z|$ for $z\in\bC^*$.

\smallskip
\noindent
\textbf{Exercise:} Let $A\in\cL^1(\fH)$. Prove that $\Tr(\sqrt{AA^*})=\Tr(\sqrt{A*A})=\|A\|_1$.

\smallskip
The two following facts are important:

\bu $\cL^1(\fH)$ is the (topological) dual of $\cK(\fH)$, with duality defined by the trace as follows
$$
\cL^1(\fH)\times\cK(\fH)\ni(A,K)\mapsto\Tr(AK)\in\bC\,;
$$

\bu $\cL(\fH)$ is the (topological) dual of $\cL^1(\fH)$, with duality defined by the trace as follows
$$
\cL(\fH)\times\cL^1(\fH)\ni(B,A)\mapsto\Tr(BA)\in\bC\,.
$$

Another important class of bounded operators on $\fH$ is the class of Hilbert-Schmidt operators\index{Hilbert-Schmidt operators}, denoted by $\cL^2(\fH)$, 
defined as follows:
$$
A\in\cL^2(\fH)\Leftrightarrow A^*A\in\cL^1(\fH)\Leftrightarrow AA^*\in\cL^1(\fH)\,.
$$
The class $\cL^2(\fH)$ is a closed two-sided ideal in $\cL(\fH)$. It is a Banach space for the Hilbert-Schmidt norm defined as follows:
$$
\|A\|_2:=\|A^*A\|_1^{1/2}=\|AA^*\|_1^{1/2}\,.
$$
One easily checks that
$$
\cL^1(\fH)\subset\cL^2(\fH)\subset\cK(\fH)\,.
$$
 
An important particular case is $\fH=L^2(\cX)$, where $\cX$ is a measurable subset of $\bR^d$, equipped with the Lebesgue measure. In that case, a 
bounded operator $A$ on $L^2(\cX)$ is a Hilbert-Schmidt operator if and only if it is an integral operator of the form
$$
Af(x)=\int_{\cX}a(x,y)f(y)dy\quad\hbox{ for all }f\in L^2(\cX)\,,
$$
with 
$$
a\in L^2(\cX\times\cX)\,.
$$
In that case
$$
\|A\|_2=\left(\iint_{\cX\times\cX}|a(x,y)|^2dxdy\right)^{1/2}=\|a\|_{L^2(\cX\times\cX)}\,.
$$

\smallskip
References for the material presented so far in this section are chapter VI of \cite{Brezis} and chapter XIX, section 1 of \cite{Horm3}.

\smallskip
According to the discussion above, each $A\in\cL^1(L^2(\bR^n))$ is an integral operator with integral kernel $a\in L^2(\bR^n\times\bR^n)$. The integral 
kernel of a trace-class operator on $L^2(\bR^n)$ has an additional interesting property, recalled below.

\begin{Lem}\lb{L-ContTrCl}
If $A\in\cL^1(L^2(\bR^n))$, its integral kernel $a\equiv a(x,y)$ is such that the map\footnote{For each topological space $X$ and each vector space
$E$ on $\bR$ equipped with the norm $\|\cdot\|_E$, we designate by $C(X,E)$ the set of continuous maps from $X$ to $E$, and we denote
$$
C_b(X,E):=\{f\in C(X,E)\hbox{ s.t. }\sup_{x\in X}\|f(x)\|_E<\infty\}\,.
$$
In other words, $C_b(X,E)$ is the set of bounded continuous functions defined on $X$ with values in $E$.}
$$
z\mapsto a(x,x+z)\hbox{ belongs to }C_b(\bR^n_z;L^1(\bR^n_x))\,.
$$
\end{Lem}

\smallskip
\noindent
\textbf{Exercise.} Here is an outline of the proof of Lemma \ref{L-ContTrCl}. Let $A\in \cL^1(\fH)$ with $\fH:=L^2(\bR^n)$.

\noindent
1) Prove that the integral kernel $a$ of $A$ can be represented as
$$
a(x,y)=\sum_{n\ge 1}\l_ne_n(x)\overline{f_n(y)}
$$
with $e_n,f_n\in\fH$ such that $|e_n|_{\fH}=|f_n|_{\fH}=1$ and with
$$
\sum_{n\ge 1}|\l_n|<\infty\,.
$$
2) Prove that, for each $\phi\in\fH$, one has
$$
\int_{\bR^n}|\phi(x+h)-\phi(x)|^2dx\to 0
$$
as $|h|\to 0$.

\noindent
3) Prove that
$$
\int_{\bR^n}|a(x,x+z)-a(x,x+z')|dx\to 0\quad\hbox{ as }|z-z'|\to 0\,,
$$
and conclude.

\smallskip
In particular, if $A\in\cL^1(L^2(\bR^n))$, then its integral kernel $a\equiv a(x,y)$ is such that the map $x\mapsto a(x,x)$ is well-defined as an element of
$L^1(\bR^n)$, and
$$
\Tr(A)=\int_{\bR^n}a(x,x)dx\,.
$$

\smallskip
The converse of this last statement is obviously false: if $A\in\cL(\fH)$ is an integral operator with integral kernel $a(x,y)$, that the map $x\mapsto a(x,x)$ 
belongs to $L^1(\bR^n)$ does not imply in general that $A\in\cL^1(L^2(\bR^n))$.

\smallskip
\noindent
\textbf{Exercise:} Construct an example of operator that belongs to $\cL^2(L^2(\bR^n))$ but not to $\cL^1(L^2(\bR^n))$.

\smallskip
In the sequel, we shall need the notion of ``partial trace'' for a trace-class operator defined on the tensor product of two Hilbert spaces. Let $\fH_1$ and
$\fH_2$ be two complex, separable Hilbert spaces. 

\begin{Def}\index{Partial trace}
For each $A\in\cL^1(\fH_1\otimes\fH_2)$, the first partial trace of $A$, denoted
$$
A_{:1}=\Tr_1(A)\in\cL^1(\fH_1)\,,
$$
is the unique element of $\cL^1(\fH_1)$ such that
$$
\Tr_{\fH_1}(A_{:1}B)=\Tr_{\fH_1\otimes\fH_2}(A(B\otimes I_{\fH_2}))
$$
holds  for each $B\in\cL(\fH_1)$, where $I_{\fH_2}$ is the identity on $\fH_2$.
\end{Def}

\smallskip
The second partial trace of $A$ is the element of $\cL^1(\fH_2)$ defined analogously.

\smallskip
Let $A$ be an integral operator with integral kernel $a\equiv a(x_1,x_2,y_1,y_2)$ in the case where $\fH_j=L^2(\bR^{n_j})$ for $j=1,2$. In that case, 
$\fH_1\otimes\fH_2=L^2(\bR^{n_1+n_2})$, and one easily checks that $A_{:1}$ is the integral operator with integral kernel 
$$
a_{:1}(x_1,y_1)=\int_{\bR^{n_2}}a(x_1,z,y_1,z)dz
$$
on $\fH_1=L^2(\bR^{n_1})$.

\smallskip
There is an easy generalization of the notion of partial trace to the case of trace-class operators defined on $N$-fold tensor products of separable Hilbert 
spaces. Given $N$ complex, separable Hilbert spaces $\fH_1,\ldots,\fH_N$ and a trace-class operator $A\in\cL^1(\fH_1\otimes\ldots\otimes\fH_N)$, for 
each $1\le k\le N$, we define
$$
A_{:k}=\Tr_k(A)
$$
to be the first partial trace $\Tr_1(A)$ in the decomposition
$$
\fH_1\otimes\ldots\otimes\fH_N\simeq(\fH_1\otimes\ldots\otimes\fH_k)\otimes(\fH_{k+1}\otimes\ldots\otimes\fH_N)\,.
$$

\smallskip
If one views $A\in\cL^1(\fH_1\otimes\ldots\otimes\fH_N)$ such that $A=A^*\ge 0$ and $\Tr(A)=1$ as the operator analogue of an $N$-particle probability 
measure, then $A_{:k}$ is the analogue of its (first) $k$-particle marginal considered above in our description of the BBGKY hierarchy for the mean field
limit in classical statistical mechanics.

\smallskip
More details on partial traces can be found in \cite{BaGoMau,BGGM}.

\bigskip
\noindent
\textbf{Remark:} All the quantum states of the $N$-particle system considered so far are defined by an $N$-particle wave function $\Psi_N(t,\cdot)$ that
belongs to $L^2((\bR^d)^N)$ and satisfies the normalization condition $\|\Psi_N(t,\cdot)\|_{L^2((\bR^d)^N)}=1$ for each $t\in\bR$. An $N$-particle 
density matrix $D_N(t)$ is associated to this $N$-particle wave function by the formula $D_N(t):=|\Psi_N(t,\cdot)\ra\la\Psi_N(t,\cdot)|$, and we have seen
that this $N$-particle density matrix $D_N(t)$ satisfies
$$
D_N(t)=D_N(t)^*=D_N(t)^2\,,\quad\hbox{ and }\Tr(D_N(t))=1\,.
$$
Conversely, if $D_N(t)\in\cL^1(L^2((\bR^d)^N))$ satisfies the conditions above, it is an orthogonal projection with rank $1$, and therefore there exists
$\Phi_N(t,\cdot)\in L^2((\bR^d)^N)$ such that $\|\Phi_N(t,\cdot)\|_{L^2((\bR^d)^N)}=1$ and $D_N(t):=|\Phi_N(t,\cdot)\ra\la\Phi_N(t,\cdot)|$. Besides, the
function $\Phi(t,\cdot)$ is defined uniquely up to multiplication by a complex number of modulus $1$. 

Such quantum states are referred to as ``pure states''.

However, there also exist more general quantum states, that are described by an $N$-particle density matrix $D_N(t)\in\cL^1(L^2((\bR^d)^N))$ satisfying
$$
D_N(t)=D_N(t)^*\ge 0\,,\quad\hbox{ and }\Tr(D_N(t))=1\,.
$$

The condition $D_N(t)=D_N(t)^*\ge 0$ is equivalent to the fact that 
$$
\la\Phi_N|D_N(t)\Phi_N\ra_{L^2((\bR^d)^N)}\ge 0\quad\hbox{ for all }\Phi_N\in L^2((\bR^d)^N)\,.
$$
Equivalently, the spectrum of $D_N(t)$ consists of nonnegative eigenvalues (since $D_N(t)\in\cL^1(L^2((\bR^d)^N))$, we already know that $D_N(t)$ 
is compact, so that its spectrum consists of eigenvalues only).

With the normalization condition $\Tr(D_N(t))=1$, one sees that $D_N(t)$ is of the form
$$
D_N(t)=\sum_{n\ge 1}\l_n(t)|\psi_n(t,\cdot)\ra\la\psi_n(t,\cdot)|
$$
where $(\psi_n(t,\cdot))_{n\ge 1}$ is a complete orthonormal system in $L^2(\bR^d)^N)$ and
$$
\l_n(t)\ge 0\,,\quad\sum_{n\ge 1}\l_n(t)=1\,.
$$
In other words, $D_N(t)$ is a (possibly infinite) convex combination of the density matrices $|\psi_n(t,\cdot)\ra\la\psi_n(t,\cdot)|$ corresponding with pure 
states. 

While the Schr\"odinger equation governs the evolution of quantum states described by wave functions, the von Neumann equation governs the evolution
of mixed as well as pure states in quantum mechanics. If $D_N(t)$ satisfies the von Neumann equation
$$
i\dot{D}_N(t)=[H_N,D_N(t)]
$$
where
$$
H_N:=-\tfrac12\sum_{k=1}^N\Dlt_{x_k}+\frac1N\sum_{1\le k<l\le N}V(x_k-x_l)
$$
and $V$ is such that $iH_N$ is the generator of a unitary group on $L^2(\bR^d)^N)$ (as in Kato's theorem for instance), then the eigenvalues of $D_N(t)$
satisfy
$$
\l_n(t)=\l_n(0)\,,\quad\hbox{ for each }n\ge 1\,.
$$
The eigenfunctions of $D_N(t)$ are transformed as
$$
\psi_n(t,\cdot)=e^{-itH_N}\psi_n(0,\cdot)\,,\quad\hbox{ for each }n\ge 1\,.
$$

Throughout the following sections, we shall be dealing with pure states, although most of our arguments could easily be extended to mixed states. 

The notion of partial trace discussed above leads to a very natural way of constructing mixed states, discussed in the exercise below.  

\smallskip
\noindent
\textbf{Exercise.} Let $\Psi_M\in L^2((\bR^d)^M)$ be an $M$-particle wave function such that $\|\Psi_M\|_{L^2((\bR^d)^M)}=1$, and let $N\ge 1$ be such 
that $N<M$. Set $D_M:=|\Psi_M\ra\la\Psi_M|$ and consider the operator $D_N:=D_{M:N}$.

\noindent
1) Check that $D_N=D_N^*\ge 0$, and that $\Tr(D_N)=1$. 

\noindent
2) Assume that $D_N=|\psi\ra\la\psi|$ with $\psi\in L^2((\bR^d)^N)$ such that $\|\psi\|_{L^2((\bR^d)^N)}=1$ and compute $\Tr(D_N|\psi\ra\la\psi|)$ in terms 
of $\Psi_M$.

\noindent
3) Under which condition on $\Psi_M$ does one have $D_N=D_N^2$? (Hint: use the result of question 2 and the equality case in the Cauchy-Schwarz
inequality to find that $\Psi_M$ is of the form $\Psi_M(x_1,\ldots,x_M)=\psi(x_1,\ldots,x_N)\phi(x_{N+1},\ldots,x_M)$.)

\smallskip
This exercise is inspired from the second paragraph in \S 14 of \cite{LL3}, where it is explained that mixed states arise for instance as quantum states for
subsystems of a larger system in some pure state. The result obtained in the exercise above suggests that the state of such a subsystem (particles $1$
to $N$) is pure if and only if the subsystem considered (the $N$ first particles are independent from the M-N remaining particles.)

\section[BBGKY hierarchy in quantum mechanics]{The BBGKY hierarchy in quantum mechanics}

In this section, we outline the BBGKY method presented above in the case of classical statistical mechanics, and explain how it leads to a rigorous derivation
of the mean-field Hartree equation from the $N$-particle Schr\"odinger equation.

First we introduce some elements of notation. 

For all $N\ge 1$, we set $X_N:=(x_1,\ldots,x_N)$ and $Y_N:=(y_1,\ldots,y_N)$; likewise for all $N\ge 1$ and $1\le k< N$, we adopt the notation
$Z_N^k=(z_{k+1},\ldots,z_N)$.

Given an integral operator $D_N$ acting on $L^2((\bR^d)^N)$ with integral kernel $D_N(X_N,Y_N)$ we denote by $D_{N:k}$ the operator with integral 
kernel
$$
D_{N:k}(X_k,Y_k):=\int_{(\bR^d)^{N-k}}D_N(X_k,Z_N^k;Y_k,Z_N^k)dZ_N^k\,.
$$
In terms of operators,
$$
D_{N:k}=\Tr_k(D_N)\,.
$$ 
Here again, we use the same letter to designate an integral operator and its integral kernel. This obvious abuse of notation is deliberately chosen to avoid
the unnecessary multiplication of mathematical symbols.

\subsection{The quantum BBGKY hierarchy}

First we explain how to deduce the BBGKY hierarchy from the $N$-particle Schr\"odinger equation.

\begin{Thm}\index{Quantum BBGKY hierarchy}\lb{T-QBBGKY}
Assume that $V$ is a real-valued, even function belonging to $L^\infty(\bR^d)$. Let $D_N^{in}=(D_N^{in})^*\ge 0$ be a trace-class operator with
$\Tr(D_N^{in})=1$ on $\fH_N=\fH^{\otimes N}=L^2((\bR^d)^N)$ with $\fH=L^2(\bR^d)$, satisfying the indistinguishability symmetry, i.e.
$$
S_\si D_N^{in}=D_N^{in}\quad\hbox{ for each }\si\in\fS_N\,.
$$
Let
$$
D_N(t):=e^{-itH_N}D_N^{in}e^{itH_N}\,,\qquad t\in\bR\,,
$$
where
$$
H_N=-\tfrac12\sum_{k=1}^N\Dlt_k+\frac1N\sum_{1\le k<l\le N}V_{kl}
$$
is the $N$-particle Schr\"odinger operator. Here $\Dlt_k$ designates the Laplacian acting on the $k$-th variable $x_k$, while $V_{kl}$ designates the
multiplication by $V(x_k-x_l)$.

Then the sequence of partial traces of $D_N$ satisfies the BBGKY hierarchy of differential equations
$$
\ba
i\dot D_{N:k}=&-\tfrac12\sum_{j=1}^k[\Dlt_j,D_{N:k}]+\frac{N-k}N\sum_{j=1}^k[V_{j,k+1},D_{N:k+1}]_{:k}
\\
&+\frac1N\sum_{1\le l<m\le k}[V_{lm},D_{N:k}]
\ea
$$
for all $k\ge 1$, with the convention 
$$
D_{N:N}=D_N\,,\quad\hbox{ and }D_{N:k}=0\hbox{ whenever }k>N\,.
$$

In terms of integral kernels, the BBGKY hierarchy reads
$$
\ba
i\d_tD_{N:k}(t,X_k,Y_k)=-\tfrac12\sum_{j=1}^k(\Dlt_{x_j}-\Dlt_{y_j})D_{N:k}(t,X_k,Y_k)
\\
+\frac{N-k}{N}\sum_{j=1}^k\int_{\bR^d}(V(x_j-z)-V(y_j-z))D_{N:k+1}(t,X_k,z,Y_k,z)dz
\\
+\frac1N\sum_{1\le l<m\le k}(V(x_l-x_m)-V(y_l-y_m))D_{N:k}(t,X_k,Y_k)
\ea
$$
for all $k\ge 1$.
\end{Thm}

\begin{proof}
Consider the free $N$-body Schr\"odinger group 
$$
W_N(t):=\exp\left(-it\tfrac12(\Dlt_1+\ldots+\Dlt_N)\right)\,,
$$
and set
$$
F_N(t):=W_N(t)D_N(t)W_N(-t)=W_N(t)D_N(t)W_N(t)^*\,.
$$
Starting from the von Neumann equation, one deduces that
$$
i\dot F_N(t)=\frac1N\sum_{1\le l<m\le N}[W_N(t)V_{lm}W_N(t)^*,F_N(t)]\,.
$$
Since $W_N(t)$ is a unitary operator on $L^2((\bR^d)^N)$, the operator $F_N(t)$ is trace-class on $L^2((\bR^d)^N)$. Since $V\in L^\infty(\bR^d)$, all
the operators $V_{lm}$ are bounded on $L^2((\bR^d)^N)$, so that $[W_N(t)V_{lm}W_N(t)^*,F_N(t)]$ is trace-class on $L^2((\bR^d)^N)$.Ta- king partial
 traces of both sides of the equality above leads to
$$
i\dot F_{N:k}(t)=\frac1N\sum_{1\le l<m\le N}[W_N(t)V_{lm}W_N(t)^*,F_N(t)]_{:k}
$$
for all $k\ge 1$.

\smallskip
Then we must distinguish three cases.

\smallskip
\noindent
a) If $k< l<m\le N$, then
$$
[W_N(t)V_{lm}W_N(t)^*,F_N(t)]_{:k}=[\tilde V_{lm}(t),F_{N:m}(t)]_{:k}=0\,,
$$
with the notation
$$
\tilde V_{lm}(t):=W_N(t)V_{lm}W_N(t)^*=e^{-it\tfrac12(\Dlt_l+\Dlt_m)}V_{lm}e^{it\tfrac12(\Dlt_l+\Dlt_m)}\,.
$$
Indeed, let $A\in\cL(L^2((\bR^d)^k))$, and denote by $I_{p}$ the identity in $\cL(L^2((\bR^d)^p))$. Then
$$
[A\otimes I_{m-k},\tilde V_{lm}(t)]=0\,,
$$ 
so that
$$
\ba
\Tr_{\fH_k}([\tilde V_{lm}(t),F_{N:m}(t)]_{:k}A)&=\Tr_{\fH_m}([\tilde V_{lm}(t),F_{N:m}(t)]A\otimes I_{m-k})
\\
&=\Tr_{\fH_k}([A\otimes I_{m-k},\tilde V_{lm}(t)]F_{N:m}(t))=0\,,
\ea
$$
with the notation $\fH_k:=L^2((\bR^d)^k)$.

\smallskip
\noindent
b) If $1\le l<m\le k\le N$, then
$$
[W_N(t)V_{lm}W_N(t)^*,F_N(t)]_{:k}=[W_k(t)V_{lm}W_k(t)^*,F_{N:k}(t)]\,.
$$
Indeed, with the same notation as above, for each $A\in\cL(L^2((\bR^d)^k))$,
$$
\ba
\Tr_{\fH_N}&([W_N(t)V_{lm}W_N(t)^*,F_N(t)]A\otimes I_{N-k})
\\
&=\Tr_{\fH_N}([\tilde V_{lm}(t),F_N(t)]A\otimes I_{N-k})
\\
&=\Tr_{\fH_N}(\tilde V_{lm}(t)F_N(t)A\otimes I_{N-k})
\\
&\quad-\Tr_{\fH_N}(F_N(t)\tilde V_{lm}(t)A\otimes I_{N-k})
\\
&=\Tr_{\fH_N}(F_N(t)A\otimes I_{N-k}\tilde V_{lm}(t))
\\
&\quad-\Tr_{\fH_N}(F_N(t)\tilde V_{lm}(t)A\otimes I_{N-k})
\\
&=\Tr_{\fH_N}(F_N(t)(A\tilde V_{lm}(t))\otimes I_{N-k})
\\
&\quad-\Tr_{\fH_N}(F_N(t)(\tilde V_{lm}(t)A)\otimes I_{N-k})
\\
&=\Tr_{\fH_k}(F_{N:k}(t)[A,\tilde V_{lm}(t)])
\\
&=\Tr_{\fH_k}([\tilde V_{lm}(t),F_{N:k}(t)]A)
\\
&=\Tr_{\fH_k}([W_k(t)V_{lm}W_k(t)^*,F_{N:k}(t)]A)\,.
\ea
$$

\smallskip
\noindent
c) It remains to treat the case $1\le l\le k<m\le N$. Denote by $\si\in\fS_N$ the transposition exchanging $k+1$ and $m$. At this point we use the symmetry of 
both operators $D_N$ and $W_N$, i.e. the fact that
$$
S_\si D_N(t)=D_N(t)\quad\hbox{ and }S_\si W_N(t)=W_N(t)
$$
for each $t\in\bR$. First, one can check by direct inspection that
$$
S_\si[\tilde V_{l,k+1}(t),F_N(t)]=[\tilde V_{l,m}(t),S_\si F_N(t)]=[\tilde V_{l,m}(t),F_N(t)]\,.
$$
Then, for each $A\in\cL(L^2((\bR^d)^k))$ and with the same notation as above,
$$
\ba
\Tr_{\fH_k}([\tilde V_{l,m}(t),F_N(t)]_{:k}A)&=\Tr_{\fH_m}([\tilde V_{l,m}(t),F_N(t)]A\otimes I_{m-k})&
\\
&=\Tr_{\fH_m}([\tilde V_{l,m}(t),F_N(t)]S_\si(A\otimes I_{m-k}))&
\\
&=\Tr_{\fH_m}(S_\si[\tilde V_{l,m}(t),F_N(t)]S^2_\si(A\otimes I_{m-k}))&
\\
&=\Tr_{\fH_m}(S_\si[\tilde V_{l,m}(t),F_N(t)](A\otimes I_{m-k}))&
\\
&=\Tr_{\fH_m}([\tilde V_{l,k+1}(t),F_N(t)](A\otimes I_{m-k}))&
\\
&=\Tr_{\fH_k}([\tilde V_{l,k+1}(t),F_N(t)]_{:k}A)&\,.
\ea
$$
(The second equality above follows from the fact that $S_\si(A\otimes I_{m-k})=A\otimes I_{m-k}$ since $\si$ acts as the identity on $\{1,\ldots,k\}$; the third
equality follows from the identities $\Tr(S_\si T)=\Tr(U_\si TU^*_\si)=\Tr(T)$ and $S_\si(BA)=U_\si BU^*_\si U_\si AU_\si=S_\si(B)S_\si(A)$, while the fourth
equality is based on the relation $S^2_\si=\hbox{Id}$.)

Thus
$$
[\tilde V_{l,m}(t),F_N(t)]_{:k}=[\tilde V_{l,k+1}(t),F_N(t)]_{:k}=[\tilde V_{l,k+1}(t),F_{N:k+1}(t)]_{:k}
$$
whenever $1\le l\le k<m\le N$ --- the last equality being obvious.

\bigskip
Therefore
$$
\ba
i\dot F_{N:k}(t)&=\frac1N\sum_{l=1}^k\sum_{m=k+1}^N[\tilde V_{l,m}(t),F_N(t)]_{:k}+\frac1N\sum_{1\le l<m\le k}[\tilde V_{lm}(t),F_{N:k}(t)]
\\
&=\frac{N-k}N\sum_{l=1}^k[\tilde V_{l,k+1}(t),F_{N:k+1}(t)]_{:k}+\frac1N\sum_{1\le l<m\le k}[\tilde V_{lm},F_{N:k}(t)]\,,
\ea
$$

Undoing the conjugation with $W_k(t)$, we next arrive at the equality
$$
\ba
i\dot D_{N:k}=&-\tfrac12\sum_{j=1}^k[\Dlt_j,D_{N:k}]+\frac1N\sum_{1\le l<m\le k}[V_{lm},D_{N:k}]
\\
&+\frac{N-k}N\sum_{j=1}^k\left(e^{-it\tfrac12\Dlt_{k+1}}[V_{j,k+1},D_{N:k+1}]e^{it\tfrac12\Dlt_{k+1}}\right)_{:k}\,.
\ea
$$

It remains to check that
$$
\left(e^{-it\tfrac12\Dlt_{k+1}}[V_{j,k+1},D_{N:k+1}]e^{it\tfrac12\Dlt_{k+1}}\right)_{:k}=[V_{j,k+1},D_{N:k+1}]_{:k}\,.
$$
For each $A\in\cL(L^2((\bR^d)^k))$ and with the same notation as above, one has
$$
\ba
{}&\Tr\left(\left(e^{-it\tfrac12\Dlt_{k+1}}[V_{j,k+1},D_{N:k+1}]e^{it\tfrac12\Dlt_{k+1}}\right)_{:k}A\right)
\\
&=
\Tr\left(\left(e^{-it\tfrac12\Dlt_{k+1}}[V_{j,k+1},D_{N:k+1}]e^{it\tfrac12\Dlt_{k+1}}\right)A\otimes I_1\right)
\\
&=
\Tr\left([V_{j,k+1},D_{N:k+1}]e^{it\tfrac12\Dlt_{k+1}}(A\otimes I_1)e^{-it\tfrac12\Dlt_{k+1}}\right)
\\
&=
\Tr([V_{j,k+1},D_{N:k+1}]A\otimes I_1)
\\
&=
\Tr([V_{j,k+1},D_{N:k+1}]_{:k}A)
\ea
$$
since
$$
e^{it\tfrac12\Dlt_{k+1}}(A\otimes I_1)e^{-it\tfrac12\Dlt_{k+1}}=A\otimes I_1\,,
$$
and this concludes the proof.
\end{proof}

\smallskip
References for this section are \cite{Spohn80,BaGoMau,BGMMinneapolis}.

\subsection{The infinite quantum mean field hierarchy}

From the quantum BBGKY hierarchy, we formally deduce an infinite hierarchy that is the quantum analogue of the infinite mean field hierarchy in classical
statistical mechanics. The formal argument is as follows: assume that
$$
D_{N:k}(t)\to D_k(t)
$$
for each fixed $k\ge 1$ as $N\to\infty$ in some sense to be made precise, so that
$$
[V_{lm},D_{N:k}(t)]\to[V_{lm},D_k(t)]\,.
$$
Then we pass to the limit as $N\to\infty$ in both terms
$$
\frac1N\sum_{1\le l<m\le k}[V_{lm},D_{N:k}]
$$
and
$$
\frac{N-k}N\sum_{j=1}^k[V_{j,k+1},D_{N:k+1}]_{:k}\,.
$$
The first term vanishes as $N\to\infty$ since
$$
\frac1N\sum_{1\le l<m\le k}[V_{lm},D_{N:k}]=O\left(\frac{k^2}{N}\right)\to 0\,,
$$
while the second term satisfies
$$
\frac{N-k}N\sum_{j=1}^k[V_{j,k+1},D_{N:k+1}]_{:k}\to\sum_{j=1}^k[V_{j,k+1},D_{k+1}]_{:k}
$$
for each $k\ge 1$ as $N\to\infty$.

Therefore, we expect that the sequence of operators $(D_k(t))_{k\ge 1}$ should satisfy the \index{Infinite hierarchy (quantum case)} infinite hierarchy of 
differential equations \index{Mean field hierarchy (quantum case)}
$$
i\dot D_k=-\tfrac12\sum_{j=1}^k[\Dlt_j,D_k]+\sum_{j=1}^k[V_{j,k+1},D_{k+1}]_{:k}\,,\qquad k\ge 1\,.
$$

In terms of integral kernels
$$
D_{N:k}(t,X_k,Y_k)\to D_k(t,X_k,Y_k)\,,
$$
for each $k\ge 1$ as $N\to\infty$, where $D_k(t,X_k,Y_k)$ is the integral kernel of the operator $D_k(t)$. Then, the sequence of functions 
$(D_k(t,X_k,Y_k))_{k\ge 1}$ satisfies the infinite hierarchy of integro-differential equations indexed by $k\ge 1$:
$$
\ba
i\d_tD_k(t,X_k,Y_k)&=-\tfrac12\sum_{j=1}^k(\Dlt_{x_j}-\Dlt_{y_j})D_k(t,X_k,Y_k)
\\
&+\sum_{j=1}^k\int_{\bR^d}(V(x_j-z)-V(y_j-z))D_{k+1}(t,X_k,z,Y_k,z)dz\,.
\ea
$$

Of course, the argument above is purely formal and remains to be justified rigorously. This is achieved by the following theorem.

\begin{Thm}\lb{T-QuantHier}
Assume that the potential $V$ is an even, real-valued function such that
$$
V\in C(\bR^d\setminus\{0\})\cap L^2_{loc}(\bR^d)\,,\quad V\to 0\hbox{ at }\infty\,,\quad\hbox{ and }V\ge -V_m
$$
for some $V_m>0$. Assume that the initial density matrix
$$
D_N^{in}=|\Psi_N^{in}\ra\la\Psi_N^{in}|\quad\hbox{ with }\int_{(\bR^d)^N}|\Psi_N^{in}(X_N)|^2dX_N=1
$$
satisfies the indistinguishability symmetry, i.e.
$$
S_\si D_N^{in}=D_N^{in}\quad\hbox{ for each }\si\in\fS_N\,,
$$
and the energy bound
$$
\ba
\cE_N^{in}:=&\tfrac12\sum_{j=1}^N\int_{(\bR^d)^N}|\grad_{x_j}\Psi_N^{in}(X_N)|^2dX_N
\\
&+
\frac1N\sum_{1\le l<m\le N}\int_{(\bR^d)^N}V(x_l-x_m)|\Psi_N^{in}(X_N)|^2dX_N=O(N)\,.
\ea
$$
Assume that, for each $k\ge 1$ and in the limit $N\to\infty$,
$$
D_{N:k}^{in}\to D_k^{in}\hbox{ in }\cL^1(L^2((\bR^d)^k)\hbox{ weak-*}\,.
$$

Then

\smallskip
\noindent
(a) the sequence $(D_{N:k})_{k\ge 1}$ indexed by $N$ is relatively compact in the infinite product space $\prod_{k\ge 1}L^\infty(\bR;\cL^1(L^2((\bR^d)^k)))$, 
endowed with the product of weak-* topologies --- $L^\infty(\bR;\cL^1(L^2((\bR^d)^k)))$ being the dual of the Banach space $L^1(\bR;\cK(L^2((\bR^d)^k)))$ 
equipped with the norm
$$
||| K |||:=\int_{\bR}\|K(t)\|_1dt\,;
$$

\noindent
(b) each limit point $(D_k)_{k\ge 1}$ of this sequence as $N\to\infty$ solves the infinite quantum mean field hierarchy written above in the sense of distributions, 
with initial data
$$
D_k(0)=D_k^{in}\,,\qquad k\ge 1\,.
$$
\end{Thm}

\smallskip
Although the proof of this result relies on relatively soft functional analytic techniques, we shall not give it in detail. We just sketch the key arguments and
refer the interested reader to \cite{BaGoMau} for the complete proof. 

\smallskip
The general strategy is to use the uniform bounds on the operator $D_N(t)$ deduced from the condition $\Tr(D_N(t))=1$ and the conservation of energy
to obtain the compactness property in statement (a) of the theorem above. 

Next, one writes each equation in the BBGKY hierarchy\footnote{The BBGKY hierarchy has been derived from the $N$-particle Schr\"odinger equation in
the case where the potential $V$ belongs to $L^\infty(\bR^d)$ --- see Theorem \ref{T-QBBGKY}. However its validity in the sense of distributions under the
assumptions of Theorem \ref{T-QuantHier} results from the same arguments as those used in the derivation of the infinite hierarchy and presented below.}
in the sense of distributions and passes to the limit as $N\to\infty$ for each $k\ge 1$ fixed. This strategy uses in particular various observations listed below.

\smallskip
A first important ingredient in the proof is the Cauchy-Schwarz inequality
$$
|D_{N:k}(t,X_k,Y_k)|^2\le D_{N:k}(t,X_k,X_k)D_{N:k}(t,Y_k,Y_k)
$$
since $D_N(t)$ is a nonnegative operator. It implies in particular the bound
$$
\iint_{(\bR^d)^k}|D_{N:k}(t,X_k,Y_k)|^2dX_kdY_k\le 1
$$
that is uniform in $t\in\bR$ and $N,k\ge 1$. We recall the convention $D_{N:k}=0$ whenever $k>N$  and the equality
$$
\int_{(\bR^d)^{k}}D_{N:k}(t,Z_k,Z_k)dZ_k=1\,,
$$
since 
$$
\int_{(\bR^d)^N}D_N(t,Z_N,Z_N)dZ_N=1\,.
$$

\smallskip
A second crucial ingredient in the proof is the $H^1$ bound
$$
\tfrac12\int_{(\bR^d)^N}|\grad_{x_l}\Psi_N(t,X_N)|^2dX_N\le\frac{\cE_N^{in}}{N}+\frac{N-1}{2N}V_m
$$
for each $l=1,\ldots,N$. This estimate is a consequence of the initial energy bound postulated on $\Psi_N^{in}$, together with the conservation of energy
under the Schr\"odinger group, implying that
$$
\ba
\tfrac12\sum_{l=1}^N\int_{(\bR^d)^N}|\grad_{x_l}\Psi_N(t,X_N)|^2dX_N
\\
=\cE_N-\frac1N\sum_{1\le k<l\le N}\int_{(\bR^d)^N}|V(x_k-x_l)||\Psi_N(t,X_N)|^2dX_N
\\
\le\cE_N+\tfrac12(N-1)V_m\int_{(\bR^d)^N}|\Psi_N(t,X_N)|^2dX_N=\cE_N+\tfrac12(N-1)V_m
\ea
$$
and of the assumption of indistinguishable particles, so that one has, for each pair $k,l$ such that $1\le k\not=l\le N$,
$$
\int_{(\bR^d)^N}|\grad_{x_k}\Psi_N(t,X_N)|^2dX_N=\int_{(\bR^d)^N}|\grad_{x_l}\Psi_N(t,X_N)|^2dX_N\,.
$$

\smallskip
A third key step in the proof is based on the following result bearing on trace-class operators.

\begin{Lem}
Consider a sequence of integral operators $A_n\in\cL^1(L^2(\bR^d))$ with integral kernels $a_n\equiv a_n(x,y)$ satisfying
$$
A_n\to 0\hbox{ in }\cL^1(L^2(\bR^d))\hbox{ weak-*}
$$
as $n\to\infty$, and
$$
\sup_{n\ge 1}\int_{\bR^d}|a_n(z,z+h)-a_n(z,z)|dz\to 0\hbox{ as }|h|\to 0\,.
$$
Then, for each $\chi\in C_c(\bR^d)$
$$
\int_{\bR^d}a_n(z,z)\chi(z)dz\to 0\quad\hbox{ as }n\to\infty\,.
$$
\end{Lem}

\smallskip
See section 2 of \cite{BaGoMau} for the proof of this lemma. 

\smallskip
This result is applied in the following context. For each $\Theta\in C_c((\bR^d)^k)$, consider
$$
a_{N,k}(t,z,w):=\iint_{(\bR^d)^{k}\times(\bR^d)^{k}}\Theta(X_k)\Theta(Y_k)D_{N:k+1}(t,X_n,z,Y_k,w)dX_kdY_k
$$
Then by the Sobolev estimate above, one has
$$
\int_{\bR^d}|a_{N,k}(t,z,z+h)-a_{N,k}(t,z,z)|dz\le |h|\|\Theta\|_{L^\infty}^2\left(\frac{2\cE_N^{in}}N+\frac{N-1}{N}V_m\right)\,.
$$
Thus
$$
\int_{\bR^d}a_{N,k}(z,z)\chi(z)dz\to\int_{\bR^d}a_k(z,z)\chi(z)dz
$$
as $N\to\infty$, where
$$
a_k(t,z,w):=\iint_{(\bR^d)^{k}\times(\bR^d)^{k}}\Theta(X_k)\Theta(Y_k)D_{k+1}(t,X_n,z,Y_k,w)dX_kdY_k\,,
$$
and where $\chi\in C^\infty(\bR)$ is such that $0\le\chi\le 1$, with $\chi(z)=1$ for $1/R\le z\le R$ and $\Supp(\chi)\subset[1/2R,R+1]$. This is the key
step in obtaining the limit of the interaction term
$$
\int_{\bR^d}(V(x_k-z)-V(y_k-z))D_{N:k+1}(t,X_k,z,Y_k,z)dz\,.
$$
 
\bigskip
Some remarks are in order before going further.

\smallskip
Observe first that the assumptions on the potential $V$ include the (repulsive) Coulomb potential in space dimension $d=3$
$$
V(z)=\tfrac1{4\pi}\frac1{|z|}
$$
between identical charged particles. This case is obviously of fundamental importance in view of the (many) applications to atomic physics and quantum
chemistry, in which the interacting particles are electrons.

Our second observation is of a more mathematical nature. Under the assumptions on the potential used in the theorem, the interaction integrand
$$
(V(x_j-z)-V(y_j-z))D_{k+1}(t,X_k,z,Y_k,z)
$$
may fail to belong to $L^1(\bR^d_z)$ for each (or almost every) $X_k,Y_k$. Yet our analysis defines the interaction integral
$$
\int_{\bR^d}(V(x_j-z)-V(y_j-z))D_{k+1}(t,X_k,z,Y_k,z)dz
$$
as a Radon \textit{measure} in $(X_n,Y_n)$ --- instead of a \textit{function} of these same variables.

\smallskip
A general reference for the material presented in this section is \cite{BaGoMau}.

\section[Derivation of Hartree's equation]{The mean field limit in quantum mechanics and and Hartree's equation}

\subsection{Mathematical statement of the mean field limit}

The quantum analogue of the mean field limit obtained in classical statistical mechanics is the following result.

\begin{Thm}\lb{T-QMFHier}
Let $V$ be a real-valued, even, bounded measurable function defined on $\bR^d$, such that $V\to 0$ at infinity. Let the initial single-particle wave function 
$\psi^{in}\in H^1(\bR^d)$ satisfy $\|\psi^{in}\|_{L^2}=1$, and let $\psi\equiv\psi(t,x)$ be the solution of the Cauchy problem for the Hartree equation
$$
\left\{
\ba
{}&(i\d_t+\tfrac12\Dlt_x)\psi(t,x)+\psi(t,x)\int_{\bR^d}V(x-z)|\psi(t,z)|^2dz=0\,,
\\
&\psi(0,x)=\psi^{in}(x)\,.
\ea
\right.
$$
Let $\Psi_N^{in}=(\psi^{in})^{\otimes N}$ and let $\Psi_N(t)=e^{-itH_N}\Psi_N^{in}$ be the solution of $N$-particle Schr\"odinger equation in the mean
field scaling, i.e. with
$$
H_N=-\tfrac12\sum_{j=1}^N\Dlt_{x_j}+\frac1N\sum_{1\le k<l\le N}V(x_k-x_l)\,.
$$

Then, for each $k\ge 1$, one has
$$
D_{N:k}(t):=|\Psi_N(t)\ra\la\Psi_N(t)|_{:k}\to D_k(t)
$$
in $L^\infty(\bR;\cL^1(L^2((\bR^d)^k)))$ weak-* as $N\to\infty$, where
$$
D_k(t)=|\psi(t)^{\otimes k}\ra\la\psi(t)^{\otimes k}|\,;
$$
--- in other words
$$
D_k(t,X_k,Y_k)=\prod_{j=1}^k\psi(t,x_j)\overline{\psi(t,y_j)}\,.
$$
\end{Thm}

\smallskip
As already observed in the case of the mean field limit in classical mechanics, using the hierarchy approach requires knowing that the mean field PDE is well-
posed: see Proposition \ref{P-ExUnHartree} above.

\smallskip
Here again, we only sketch the key arguments in the proof, and refer the interested reader to \cite{Spohn80} or \cite{BaGoMau} for a complete account.

By Theorem \ref{T-QMFHier}, we know that the sequence $(D_{N:k})_{k\ge 1}$ indexed by $N$ is relatively compact in 
$$
\prod_{k\ge 1}L^\infty(\bR;\cL^1(L^2((\bR^d)^k)))
$$
equipped with the product topology, where $L^\infty(\bR;\cL^1(L^2((\bR^d)^k)))$ is endowed with the weak-* topology for each $k\ge 1$. 

Besides, Theorem \ref{T-QMFHier} implies that any limit point $(D_k)_{k\ge 1}$ of that sequence is a solution of the infinite quantum mean field hierarchy
with initial data
$$
D_k^{in}:=|(\psi^{in})^{\otimes k}\ra\la(\psi^{in})^{\otimes k}|\,,\qquad k\ge 1\,.
$$

\begin{Prop}
Let $V$ be a real-valued, even, bounded measurable function defined on $\bR^d$, such that $V\to 0$. Let the initial single-particle wave function 
$\psi^{in}\in H^1(\bR^d)$ satisfy $\|\psi^{in}\|_{L^2}=1$, and let $\psi\in C(\bR;H^1(\bR^d))\cap C^1(\bR;L^2(\bR^d))$ be the solution of the Cauchy problem
for the Hartree equation
$$
\left\{
\ba
{}&(i\d_t+\tfrac12\Dlt_x)\psi(t,x)+\psi(t,x)\int_{\bR^d}V(x-z)|\psi(t,z)|^2dz=0\,,
\\
&\psi(0,x)=\psi^{in}(x)\,.
\ea
\right.
$$
Then the sequence
$$
\rho_k(t)=|\psi(t,\cdot)^{\otimes k}\ra\la\psi(t,\cdot)^{\otimes k}|
$$ 
is a particular solution of the Cauchy problem for the infinite quantum mean field hierarchy with initial data 
$$
\rho_k^{in}:=|(\psi^{in})^{\otimes k}\ra\la(\psi^{in})^{\otimes k}|\,,\qquad k\ge 1\,.
$$
\end{Prop}

\begin{proof}
Since  $\psi\in C(\bR;H^1(\bR^d))\cap C^1(\bR;L^2(\bR^d))$ satisfies the Hartree equation, one has
$$
\ba
i\d_t(\psi(t,x)\overline{\psi(t,y)})&=-\tfrac12(\Dlt_x-\Dlt_y)\psi(t,x)\overline{\psi(t,y)}
\\
&+\psi(t,x)\overline{\psi(t,y)}\int_{\bR^d}(V(x-z)-V(y-z))|\psi(t,z)|^2dz
\ea
$$
in the sense of distributions on $\bR_t\times\bR^d_x\times\bR^d_y$. Next, an elementary computation shows that, for each $k>1$,
$$
\ba
i\d_t\rho_k(t,X_k,Y_k)&=i\sum_{j=1}^k\prod_{l=1\atop l\not=j}^k\psi(t,x_l)\overline{\psi(t,y_l)}\d_t(\psi(t,x_j)\overline{\psi(t,y_j)})
\\
&=-\tfrac12\sum_{j=1}^k\prod_{l=1\atop l\not=j}^k\psi(t,x_l)\overline{\psi(t,y_l)}(\Dlt_{x_j}-\Dlt_{y_j})\psi(t,x_j)\overline{\psi(t,y_j)}
\\
&+\sum_{j=1}^k\rho_k(t,X_k,Y_k)\int_{\bR^d}(V(x_j-z)-V(y_j-z))|\psi(t,z)|^2dz
\\
&=-\tfrac12\sum_{j=1}^k(\Dlt_{x_j}-\Dlt_{y_j})\prod_{l=1}^k\psi(t,x_l)\overline{\psi(t,y_l)}
\\
&+\sum_{j=1}^k\int_{\bR^d}(V(x_j-z)-V(y_j-z))\rho_k(t,X_k,Y_k)|\psi(t,z)|^2dz
\\
&=-\tfrac12\sum_{j=1}^k(\Dlt_{x_j}-\Dlt_{y_j})\rho_k(t,X_k,Y_k)
\\
&+\sum_{j=1}^k\int_{\bR^d}(V(x_j-z)-V(y_j-z))\rho_{k+1}(t,X_k,z,Y_k,z)dz\,.
\ea
$$
Therefore, the sequence $(\rho_k(t,X_k,Y_k))_{k\ge 1}$ is a solution of the infinite quantum mean field hierarchy.
\end{proof}

\smallskip
It remains to prove that all limit points $(D_k)_{k\ge 1}$ of the sequence $(D_{N:k})_{k\ge 1}$ as $N\to\infty$ fall in some uniqueness class for the infinite 
quantum mean field hierarchy. If this is the case, each limit point $(D_k)_{k\ge 1}$, being a solution of the infinite quantum mean field hierarchy with 
the same initial data as $(\rho_k)_{k\ge 1}$, must coincide with it for all $t\in\bR$. By compactness and uniqueness of the limit point, one concludes that
$$
D_{N:k}(t)\to\rho_k(t):=|\psi(t,\cdot)^{\otimes k}\ra\la\psi(t,\cdot)^{\otimes k}|\,,\qquad\hbox{ for each }k\ge 1
$$
in the weak-* topology of $L^\infty(\bR;\cL^1(L^2((\bR^d)^k)))$ as $N\to\infty$.

Since the integral kernels of the operators $\rho_k$ are of factorized form, i.e.
$$
\rho_k(t,X_k,Y_k)=\prod_{j=1}^k\psi(t,x_j)\overline{\psi(t,y_j)}\,,\qquad k\ge 1\,,
$$
this last statement is the quantum analogue of the propagation of chaos studied above in the context of the mean field limit in classical statistical mechanics.

\smallskip
In order to carry out this program, two things need to be checked:

\smallskip
\noindent
a) one needs an estimate for the interaction terms 
$$
D_{k+1}\to[V_{j,k+1},D_{k+1}]_{:k}
$$
for $j=1,\ldots,k$, and

\noindent
b) one needs an estimate for the growth of
$$
D_k\quad\hbox{ as }k\to\infty\,.
$$

\smallskip
In the next section, we present an abstract framework explaining the role of each estimates.

\subsection{A tool for studying infinite hierarchies}

We first review a seemingly unrelated result, namely the Nirenberg-Ovcyannikov abstract analogue of the Cauchy-Kovalevska theorem. More precisely, 
we recall below a variant of this theorem due to Nishida.

\subsubsection{The Nirenberg-Ovcyannikov abstract Cauchy-Kovalevska theorem}

The Nirenberg-Ovcyannikov abstract Cauchy-Kovalevska theory\index{Abstract Cauchy-Kovalevska theorem} bears on the Cauchy problem for differential equations of the form
$$
\dot u(t)=F(t,u(t))\,,
$$
in situations where one can apply neither the Cauchy-Lipschitz theorem nor the Peano existence theorem (see chapter 1 in \cite{HormNL}). Such situations 
arise typically when the solution $u(t)$ belongs to function spaces that are not invariant under the mapping $v\mapsto F(t,v)$. We first introduce the functional 
setting and present the main result obtained independently by Nirenberg \cite{Niren} and Ovcyannikov \cite{Ovcya}. 

The setting involves a family (a scale) of Banach spaces $(B_r)_{r>0}$ indexed by the positive real numbers. The norm in $B_r$ is denoted by $\|\cdot\|_r$.
This scale of Banach spaces is nonincreasing, in the sense that, for each $r,r'>0$, one has
$$
r'\le r\Rightarrow B_r\subset B_{r'}\quad\hbox{ with }\quad\|v\|_{r'}\le\|v\|_r\quad\hbox{Êfor each }v\in B_r\,.\leqno{(HB_r)}
$$

Next, we specify how the time-dependent ``vector field'' $F$ behaves with respect to the scale of Banach spaces $(B_r)_{r>0}$. 

Our first assumption bears on the joint dependence of $F$ in both of its arguments:

\smallskip
\noindent
(a) there exists $r_0>0$ s.t. 
$$
\ba
0<r'<r<r_0&\Rightarrow F\in C(\bR\times B_r,B_{r'})
\\
&\hbox{and }B_r\ni v\mapsto F(t,v)\in B_{r'}\hbox{ is linear for all }t\in\bR\,.
\ea
\leqno{(HF_a)}
$$

\smallskip
A second assumption bears on the dependence of $F$ in its second argument, uniformly in the time variable $t$:

\smallskip
\noindent
(b) there exists $C>0$ s.t.
$$
0<r'<r<r_0\Rightarrow\|F(t,v)\|_{r'}\le C\frac{\|v\|_r}{r-r'}\,.\leqno{(HF_b)}
$$

We state only the uniqueness part of the Nirenberg-Ovcyannikov theorem, and refer the interested reader to \cite{Nishida} for a the complete statement ---
that is an existence theorem as well, and moreover bears on possibly nonlinear ``vector fields''  $F$.

\begin{Thm}\lb{T-CKNO}
For all $r\in(0,r_0)$ and all $\a>0$, the only solution $u$ of the Cauchy problem
$$
\left\{
\ba
{}&\dot u(t)=F(t,u(t))\,,
\\
&u(0)=0\,,
\ea
\right.
$$
such that $u\in C^1((-\a,\a),B_r)$ is $u=0$.
\end{Thm}

\smallskip
In order to fully appreciate the meaning of assumption (HF$_b$), we briefly explain how this theorem is related to the classical Cauchy-Kovalevska theorem.

Consider the PDE with unknown $u\equiv u(t,z)\in\bC$,
$$
\d_tu(t,z)=P(t,z,\d_z)u(t,z)\,,
$$
where $t\in\bR$ and $z\in\bC$. It is assumed that $P$ is a 1st order linear differential operator of the form
$$
P(t,z,\d_z)u(t,z):=A(t,z)\d_zu(t,z)+B(t,z)u(t,z)\,,
$$
with coefficients $A$ and $B$ bounded and analytic on $(-T,T)\times\Om_{r_0}$, where $\Om_r\subset\bC$ is the open strip defined as follows:
$$
\Om_r:=\{z\in\bC\hbox{ s.t. }|\IM(z)|<r\}\,.
$$

Denote by $B_r$ the space of bounded holomorphic functions on $\Om_r$ with
$$
\|v\|_r:=\sup_{z\in\Om_r}|v(z)|\,.
$$

We recall that, for each function $f$ that is holomorphic and bounded in the strip $\Om_r$, one has
$$
|f^{(k)}(z)|\le\frac{\sup_{z\in\Om_r}|f(z)|}{(r-r')^k}\qquad\hbox{ for all }z\in\Om_{r'}\hbox{ and }k\ge 1\,,
$$
because of\index{Cauchy's estimates} Cauchy's estimates for holomorphic functions (see for instance chapter 4, section 2.3 in \cite{Ahlfors}).

Therefore, one has
$$
\ba
\|P(t,\cdot,\d_z)u(t,\cdot)\|_{r'}&\le\sup_{z\in\Om_{r_0}}|A(t,z)|\sup_{z\in\Om_{r'}}|\d_zu(t,z)|
\\
&+\sup_{z\in\Om_{r_0}}|B(t,z)|\sup_{z\in\Om_{r'}}|u(t,z)|
\\
&\le\sup_{z\in\Om_{r_0}}|A(t,z)|\frac{\sup_{z\in\Om_{r}}|u(t,z)|}{r-r'}
\\
&+\sup_{z\in\Om_{r_0}}|B(t,z)|\sup_{z\in\Om_{r}}|u(t,z)|
\\
&\le C\frac{\sup_{z\in\Om_{r}}|u(t,z)|}{r-r'}
\ea
$$
with
$$
C:=\sup_{z\in\Om_{r_0}}|A(t,z)|+r_0\sup_{z\in\Om_{r_0}}|B(t,z)|\,.
$$
In other words, the 1st order differential operator $P(t,z,\d_z)$ satisfies assumption (HF$_b$) as a consequence of the Cauchy inequalities for bounded
holomorphic functions in the strip $\Om_r$.

Observe that differential operators $Q(t,z,\d_z)$ of order higher than $1$ do not satisfy assumption (HF$_b$), but an inequality of the form
$$
\|Q(t,\cdot,\d_z)u(t,\cdot)\|_{r'}\le C\frac{\sup_{z\in\Om_{r}}|u(t,z)|}{(r-r')^k}\,,\qquad\hbox{Êwith }k=\hbox{Êorder of }Q\,.
$$
This is a weaker inequality that the one in (HF$_b$), which follows again from Cauchy's estimates, this time on higher order derivatives of holomorphic 
functions.

In fact, this is not surprising: the Cauchy-Kovalevska theorem guarantees the local existence and uniqueness of a local analytic solution of the Cauchy 
problem for the PDE
$$
\left\{
\ba
{}&\d_tu(t,z)=Q(t,z,\d_z)u(t,z)
\\
&u\rstr_{t=0}=u^{in}
\ea
\right.
$$
defined for $|t|<T$ and $|\IM(z)|<r^*$ for some $T>0$ and some $r^*>0$. This is obviously impossible in general if $Q$ is of order higher than $1$.
Consider indeed the simplest possible example where $Q(t,z,\d_z)=\d_z^2$. It is well known that the Cauchy problem for the backward heat equation 
is not well-posed in any one of the class $B_r$. Indeed, for each $t>0$, the semigroup $e^{t\d^2_z}$ maps the set of distributions with compact support 
on $\bR$ into the set of functions that are entire holomorphic (i.e. holomorphic on $\bC$). 

\subsubsection{Application of the abstract Cauchy-Kovalevska theorem to infinite hierarchies}\index{Application of the Cauchy-Kovalevska theory to hierarchies}

In this section, we give an abstract theorem for proving the uniqueness of the solution of infinite hierarchies based on the abstract Cauchy-Kovalevska
theorem presented above.

Our setting involves a sequence of Banach spaces $(E_n)_{n\ge 1}$. For each $n\ge 1$, the norm in $E_n$ is denoted by $|\cdot|_n$.

Along with the Banach spaces $E_n$, our setting also involves 

\smallskip
\noindent
a) a strongly continuous group of isometries $U_n(t)$ defined on $E_n$ for each $n\ge 1$, and

\smallskip
\noindent
b) a sequence of bounded linear operators $L_{n,n+1}:\,E_{n+1}\to E_n$ defined for each $n\ge 1$.

\smallskip
It will be assumed that the sequence of operators $L_{n,n+1}$ satisfies the following bound: there exists $C>0$ such that
$$
\|L_{n,n+1}\|_{\cL(E_{n+1},E_n)}\le Cn\,,\qquad\hbox{ for each }n\ge 1\,.\leqno{(HL)}
$$

\smallskip
In the setting so defined, we consider the infinite hierarchy of differential equations
$$
\dot u_n(t)=U_n(t)L_{n,n+1}U_{n+1}(-t)u_{n+1}(t)\,,\qquad n\ge 1\,.
$$

\begin{Thm}\lb{T-UniqHier}
Let $t^*>0$ and let $u_n\in C^1([0,t^*],E_n)$ for each $n\ge 1$ be a solution of
$$
\left\{
\ba
{}&\dot u_n(t)=U_n(t)L_{n,n+1}U_{n+1}(-t)u_{n+1}(t)\,,\qquad n\ge 1\,,
\\
&u_n(0)=0\,.
\ea
\right.
$$

Assume there exists $R>0$ such that 
$$
\sup_{0\le t\le t^*}\|u_n(t)\|_n\le R^n\qquad\hbox{ for all }n\ge 1\,.\leqno{(Hu)}
$$
Then 
$$
u_n(t)=0\quad\hbox{ for all }t\in[0,t^*] \hbox{ and each }n\ge 1\,.
$$
\end{Thm}

\begin{proof}
We apply the abstract Cauchy-Kovalevska theorem in the following setting. First we choose the scale of Banach spaces: for each $r>0$, we set
$$
B_r:=\left\{v=(v_n)_{n\ge 1}\in\prod_{n\ge 1}E_n\hbox{ s.t. }\|v\|_r:=\sum_{n\ge 1}r^n|v_n|_n<\infty\right\}\,.
$$
Obviously, if $0<r'<r$, one has
$$
B_r\subset B_{r'}\,,
$$
and
$$
\|v\|_{r'}=\sum_{n\ge 1}r'^n|v_n|_n\le\sum_{n\ge 1}r^n|v_n|_n=\|v\|_r
$$
for each $v\in B_r$.

The mapping $F$ is defined on sequences $v=(v_n)_{n\ge 1}$ as follows:
$$
F(t,v):=(U_n(t)L_{n,n+1}U_{n+1}(-t)v_{n+1})_{n\ge 1}\,.
$$

With this definition, observing that
$$
nr'^n\le r'^n+r^{n-1}r+\ldots+r'r^{n-1}+r^n=\frac{r^{n+1}-r'^{n+1}}{r-r'}\,,
$$
assumption (HL) implies that
$$
\ba
\|F(t,v)\|_{r'}\le C\sum_{n\ge 1}n{r'}^n|v_n|_n\le C\sum_{n\ge 1}\frac{r^{n+1}-{r'}^{n+1}}{r-r'}|v_n|_n
\\
\le\frac{C}{r-r'}\sum_{n\ge 1}r^{n+1}|v|_{n+1}\le\frac{C\|v\|_r}{r-r'}
\ea
$$
for each $r,r'>0$ such that $r'<r$, where $C$ is the constant that appears in the control of the norm of $L_{n,n+1}$ in condition (HL). 

This implies that $F$ satisfies assumption (HF$_b$). As for assumption (HF$_a$), it follows from assumption (HF$_b$) and the fact  that $U_n(t)$ is a 
linear group of isometries on $E_n$ for each $n\ge 1$.

Thus, if a sequence $u(t)=(u(t)_n)_{n\ge 1}$ with $u(t)_n\in E_n$ for each $n\ge 1$ and each $t\in[0,t^*]$ satisfies the growth condition (Hu) and the 
infinite hierarchy, then $u(t)\in B_r$ for all $t\in[0,t^*]$ and for each $r\in(0,1/R)$. 

We deduce from the differential equation that $u\in C^1([0,t^*],B_r)$ for each $r\in(0,1/R)$. By the abstract Cauchy-Kovalevska theorem, 
we conclude that $u(t)=0$ for each $t\in[0,t^*]$, which means that 
$$
u_n(t)=0\qquad\hbox{ for each }t\in[0,t^*]\hbox{ and each }n\ge 1\,.
$$
This is precisely the expected uniqueness property.\end{proof}

\smallskip
Notice that assumption (Hu) is essential for the uniqueness property above. Here is an easy counterexample, showing that this assumption cannot be
dispensed with.

Consider the infinite hierarchy of ODEs
$$
\left\{
\ba
{}&\dot{y}_k(t)=ky_{k+1}(t)\,,\qquad k\ge 1\,,
\\
&y_k(0)=y^{in}_k\,,
\ea
\right.
$$
This hierarchy has obvious factorized solutions: if $x$ is the solution of the Riccati equation
$$
\left\{
\ba
{}&\dot{x}(t)=x(t)^2\,,
\\
&x(0)=x^{in}\,,
\ea
\right.
$$
i.e.
$$
x(t)=\frac{x^{in}}{1-tx^{in}}\,,
$$
then
$$
y_k(t)=x(t)^k\,,\quad k\ge 1
$$
is a solution of the infinite hierarchy with initial data
$$
y_k^{in}:=(x^{in})^k\,,\quad k\ge 1\,.
$$

The infinite hierarchy  with $y^{in}_k=0$ for all $k\ge 0$ has a unique solution satisfying the growth condition (Hu), which is $y_k(t)=0$ for all $t\in\bR$ and 
$k\ge 1$.

Consider now the function $E$ defined by
$$
E(t)=\left\{\begin{array}{ll}e^{-1/t}&\quad\hbox{ if }t>0\\ 0&\quad\hbox{ if }t\le 0\end{array}\right.
$$
that is of class $C^\infty$ on the real line and satisfies $E^{(n)}(0)=0$ for all $n\ge 0$. Then the formula
$$
y_k(t)=\frac1{(k-1)!}E^{(k-1)}(t)\,,\quad t\in\mathbf{R}\hbox{ and }k\ge 1\,,
$$
also defines a solution of the infinite hierarchy. 

That this solution does not grows exponentially as $k\to\infty$ corresponds with the fact that $E$ is not an analytic function on the line, since it vanishes on 
the half-line without being identically $0$, which contradicts the principle of isolated zeros.
 
\smallskip
References for this section are \cite{Ukai01,BEGMY}.

\subsection{Application to the Hartree limit in the bounded potential case}\index{Hartree limit (bounded potential case)}

Finally we explain how to prove Theorem \ref{T-QMFHier} with the formalism discussed above.

\begin{proof}[Proof of Theorem \ref{T-QMFHier}]

\noindent
\textit{Step 1:} First we write the infinite quantum mean field hierarchy 
$$
i\dot D_j(t)=-\tfrac12\sum_{k=1}^j[\Dlt_j,D_j(t)]+\sum_{k=1}^j[V_{k,j+1},D_{j+1}]_{:j}\,,\quad j\ge 1
$$
in the form
$$
\dot u_j(t)=U_j(t)L_{j,j+1}U_{j+1}(-t)u_{j+1}(t)\,,\quad j\ge 1\,,
$$
where
$$
L_{j,j+1}D_{j+1}:=\sum_{k=1}^j[V_{k,j+1},D_{j+1}]_{:j}
$$
while 
$$
U_j(t)D_j:=e^{-\tfrac12it(\Dlt_1+\ldots+\Dlt_j)}D_je^{\tfrac12it(\Dlt_1+\ldots+\Dlt_j)}\,.
$$
As before, $\Dlt_k$ designates the Laplacian acting on the $k$th variable in $L^2((\bR^d)^j)$, for each $k=1,\ldots,j$, while
$$
u_j(t):=U_j(t)D_j(t)\,.
$$

\smallskip
\noindent
\textit{Step 2:} Next we choose appropriate norms and function spaces so as to apply Theorem \ref{T-UniqHier}. Define $E_j:=\cL^1(L^2((\bR^d)^j))$ 
with trace norm denoted by $\|\cdot\|_1$ (or $\|\cdot\|_{1,j}$ in cases where the previous notation would lead to ambiguities).

The operator $U_j(t)$ defined in Step 1 is an isometry of $E_j$ for all $j\ge 1$ and all $t\in\bR$, since it consists of conjugating elements of $E_j$ with 
the unitary group associated to the $j$-particle free Schr\"odinger equation. Indeed, we recall the following elementary result.

\begin{Lem}
Let $\fH$ be a separable Hilbert space. For each $A\in L^1(\fH)$ and each unitary operator $U$ on $\fH$, then $UAU^*\in\cL^1(\fH)$ and one has
$$
\|UAU^*\|_1=\|A\|_1\,.
$$
\end{Lem}

\smallskip
\noindent
\textbf{Exercise:} prove the lemma above. (Hint: use the definition of the trace-norm.)

\smallskip
Now we need to control the interaction term. 

First, since $V\in L^\infty$, one has
$$
\ba
\|[V_{k,j+1},D_{j+1}]\|_1\le\|V_{k,j+1}D_{j+1}\|_1+\|D_{j+1}V_{k,j+1}\|_1
\\
\le 2\|V_{k,j+1}\|\|D_{j+1}\|_1\le 2\|V\|_{L^\infty}\|D_{j+1}\|_1
\ea
$$
for each $(D_j)_{j\ge 1}$ with $D_j\in\cL^1(L^2((\bR^d)^j))$ for each $j\ge 1$.

\smallskip
Next, we use the following auxiliary result.

\begin{Lem}
Set $\fH=L^2(\bR^{m+n})$ and $\fH_1=L^2(\bR^m)$. Let $K\in\cL^1(\fH)$ be an integral operator of the form
$$
K\phi(x,z)=\int_{\bR^{m+n}}k(x,z,y,w)\phi(y,w)dydw\,.
$$
Let $K_{z,w}$ be the operator defined on $\fH_1$ by
$$
K_{z,w}\psi(x)=\int_{\bR^{m}}k(x,z,y,w)\psi(y)dy
$$
for a.e. $z,w\in\bR^n$. Then

\smallskip
\noindent
(1) the map $h\mapsto[z\mapsto K_{z,z+h}]$ belongs to $C(\bR^n_h,L^1(\bR^n_z,\cL^1(\fH_1)))$;

\smallskip
\noindent
(2) for a.e. $z\in\bR^n$ the operator $K_{z,z}$ is trace-class on $L^2(\bR^m)$ and
$$
\Tr\left|\int_{\bR^n}K_{z,z}dz\right|\le\int_{\bR^n}\Tr|K_{z,z}|dz\le\|K\|_1\,.
$$
\end{Lem}

\smallskip
This result is an amplification of Lemma \ref{L-ContTrCl} --- see Lemma 2.1 in \cite{BaGoMau}.

\smallskip
We apply the lemma above to 
$$
K=[V_{k,j+1},D_{j+1}]\,,
$$
which is an integral operator with integral kernel 
$$
\ba
k(X_j,x_{j+1},Y_j,y_{j+1})
\\
=(V(x_k-x_{j+1})-V(y_k-y_{j+1}))D_{j+1}(X_j,x_{j+1},Y_j,y_{j+1})\,.
\ea
$$
This shows that
$$
\ba
\|[V_{k,j+1},D_{j+1}]_{:j}\|_1&=\Tr\left|\int_{\bR^d}K_{z_{j+1},z_{j+1}}dz_{j+1}\right|&
\\
&\le\|K\|_1\le  2\|V\|_{L^\infty}\|D_{j+1}\|_1&\,.
\ea
$$
Therefore we conclude that 
$$
\|L_{j,j+1}\|_{\cL(E_{j+1},E_j)}\le 2j\|V\|_{L^\infty}\,,
$$
which means that the infinite quantum mean field hierarchy with bounded potential satisfies assumption (HL) in Theorem \ref{T-UniqHier}.

\smallskip
\noindent
\textit{Step 3:} Finally we control the growth of $D_j$ as $j\to\infty$, where $D_j$ is a limit point of $D_{N:j}$ in $L^\infty(\bR_+;\cL^1(L^2((\bR^d)^j)))$
weak-* as $N\to\infty$. First we recall that, for each $N\ge 1$ and each $t\ge 0$, one has
$$
D_N(t)=D_N(t)^*\ge 0\,,\quad\hbox{ and }\|D_N(t)\|_1=\Tr D_N(t)=1\,.
$$
Applying the lemma above shows that
$$
\|D_{N:j}(t)\|_1\le 1\qquad\hbox{ for all }t\ge 0\,,\,\,N,j\ge 1\,.
$$
(We recall the convention $D_{N:j}=0$ whenever $j>N$.) Since
$$
D_{N:j}\wto D_j(t)\hbox{ in }L^\infty(\bR_+;\cL^1(L^2((\bR^d)^j)))\hbox{ weak-*}\,,
$$
we conclude that
$$
\|D_j(t)\|_1\le 1\qquad\hbox{ for all }t\ge 0\,,\,\,j\ge 1\,.
$$

\smallskip
\noindent
\textit{Step 4:} One can apply Theorem \ref{T-UniqHier} since the infinite hierarchy under consideration satisfies the bound (HL) with $C=2\|V\|_{L^\infty}$, 
while each limit point $D_j$ of the sequence $D_{N:j}$ in $L^\infty(\bR_+;\cL^1(L^2((\bR^d)^j)))$ as $N\to\infty$ satisfies the growth estimate (Hu) with 
$R=1$.

Indeed, setting 
$$
v_j(t):=U_j(t)(D_j(t)-|\psi(t,\cdot)^{\otimes j}\ra\la\psi(t,\cdot)^{\otimes j}|)
$$
we know that the sequence $(v_j)_{j\ge 1}$ is a solution of the infinite mean field hierarchy, which satisfies
$$
\|v_j(t)\|_{1,j}\le\|D_j(t)\|_{1,j}+\|\,|\psi(t,\cdot)^{\otimes j}\ra\la\psi(t,\cdot)^{\otimes j}|\,\|_{1,j}
\\
\le 1+\|\psi(t,\cdot)\|_{L^2}^{2j}=2\,.
$$
Since $v_j(0)=0$ for each $j\ge 1$, we conclude from Theorem \ref{T-UniqHier} that $v_j(t)=0$ for each $j\ge 1$ and each $t\ge 0$. In other words,
$$
D_j(t)=|\psi(t,\cdot)^{\otimes j}\ra\la\psi(t,\cdot)^{\otimes j}|
$$
for each $j\ge 1$, which is precisely the desired result.
\end{proof}

\smallskip
Notice that the verification of the growth condition (Hu) on the sequence $(D_j)_{j\ge 1}$ is trivial in this problem. This is not always the case, and we shall
see some examples below, where checking this crucial condition can be a serious difficulty.

\smallskip
References for this section are \cite{Spohn80,BaGoMau}.

\section{Other mean field limits in quantum mechanics}\lb{S-MFQuantFurther}

\subsection{Derivation of the Schr\"odinger-Poisson equation}\index{Schr\"odinger-Poisson limit}

The quantum analogue of the Vlasov-Poisson system is the following variant of the Hartree equation, where the interaction potential is the repulsive
Coulomb potential (between particles of equal charges):
$$
V(z)=\frac1{4\pi |z|}
$$
in space dimension $d=3$. It takes the form
$$
\left\{
\ba
{}&i\d_t\psi(t,x)=-\tfrac12\Dlt_x\psi(t,x)+U(t,x)\psi(t,x)\,,\quad x\in\bR^3\,,
\\
&U(t,x)=\tfrac1{4\pi}\int_{\bR^3}\frac{|\psi(t,y)|^2}{|x-y|}dy\,,
\ea
\right.
$$
or equivalently
$$
\left\{
\ba
{}&i\d_t\psi(t,x)=-\tfrac12\Dlt_x\psi(t,x)+U(t,x)\psi(t,x)\,,\quad x\in\bR^3\,,
\\
&-\Dlt_xU(t,x)=|\psi(t,x)|^2\,.
\ea
\right.
$$
This last form of the mean field equation explains why, as mentioned above, the Hartree equation in this case is also known as the ``Schr\"odinger-Poisson'' 
equation.

Although the Coulomb potential is unbounded near the origin, the Schr\"odin\-ger-Poisson equation can be derived following the strategy described above: 
see \cite{ErdosYau02} for this very interesting result. In that case, the right choice of spaces $E_n$ is as follows. Define
$$
S_j:=\sqrt{I-\Dlt_{x_j}}\,,\qquad j=1,\ldots,N\,.
$$
With the same notation as above, the sequence of Banach spaces in Theorem \ref{T-UniqHier} is defined as follows:
$$
E_n:=\{D\in\cL^1(L^2((\bR^3)^n))\,|\,S_1\ldots S_nDS_1\ldots S_n\in\cL^1(L^2((\bR^3)^n))\}\,,
$$
for all $n\ge 1$, with norm
$$
\|D\|_{E_n}=\Tr|S_1\ldots S_nDS_1\ldots S_n|\,.
$$

With this choice of spaces, the inequality in assumption (HL) in Theorem \ref{T-UniqHier} follows from Hardy's inequality\index{Hardy's inequality} 
$$
\int_{\bR^3}\frac{f(x)^2}{|x|^2}dx\le C\int_{\bR^3}|\grad f(x)|^2dx
$$
(see Lemma 7.1 in \cite{ErdosYau02}).

The estimate (Hu) is a follows from the conservation laws
$$
\frac{d}{dt}\Tr(H_N^mD_N(t))=0\quad\hbox{ for all }m\ge 0\,,
$$
where
$$
H_N:=-\tfrac12\sum_{k=1}^N\Dlt_{x_k}+\frac1N\sum_{1\le k<l\le N}V_{kl}
$$
(with $V_{kl}$ denoting the multiplication by $V(x_k-x_l)$) and
$$
D_N(t)=|e^{itH_N}(\psi^{in})^{\otimes N}\ra\la e^{itH_N}(\psi^{in})^{\otimes N}|\,.
$$
A crucial step in the proof of the growth estimate (Hu) is to compare powers of the $N$-particle Hamiltonian $H_N$ with powers of the $N$-particle free 
Schr\"odinger operator 
$$
L_N:=-\tfrac12\sum_{k=1}^N\Dlt_{x_k}
$$
--- see Propositions 4.1 and 5.1 in \cite{ErdosYau02}. Verifying the growth condition (Hu) is perhaps the most technical parts of \cite{ErdosYau02}.

A more direct proof, by a somewhat different argument, was proposed later in \cite{Pickl} --- see section \ref{SS-Pickl}Ê below.

\subsection{Derivation of the nonlinear Schr\"odinger equation}\index{Nonlinear Schr\"odinger limit}

One can consider the mean field limit with even more singular interactions, viz. $V(z)=\de_0(z)$\index{Dirac potential}.  Equivalently, one can prove the 
mean field limit with a potential whose range shrinks to zero as the number of particles tends to infinity. Specifically, one can consider the $N$-particle 
Schr\"odinger operator in space dimension $1$:
$$
H_N:=-\tfrac12\sum_{k=1}^N\d^2_{x_k}+\frac1N\sum_{k,l=1}^NN^\g U(N^\g (x_k-x_l))
$$
where $U\ge 0$ belongs to the Schwartz class $\cS(\bR)$, and $\g\in(0,1)$\,.

The corresponding mean-field equation is the cubic nonlinear Schr\"odinger equation
$$
i\d_t\psi(t,x)=-\tfrac12\d_x^2\psi(t,x)+a|\psi(t,x)|^2\psi(t,x)\,,\quad x\in\bR\,,
$$
where
$$
a=\int_\bR U(z)dz\,.
$$

In this case again, the strategy outlined in Theorem \ref{T-UniqHier} applies with some appropriate choice of norms. The right choice of Banach spaces 
in this case is as follows:
$$
E_n:=\{D\in\cL^2(L^2(\bR^n))\,|\,S_1\ldots S_nDS_1\ldots S_n\in\cL^2(L^2(\bR^n))\}\,,
$$
with norm
$$
\|D\|_{E_n}=(\Tr(|S_1\ldots S_nDS_1\ldots S_n|^2))^{1/2}\,,
$$
where, as in the previous example, 
$$
S_j:=\sqrt{I-\d^2_{x_j}}\,,\quad j=1,\ldots N\,.
$$

Controling the interaction operator in the $n$-th equation of the infinite mean field hierarchy associated to the Schr\"odinger operator $H_N$ is essentially 
equivalent to controling $n$ integral operators on $L^2(\bR^n)$ with integral kernels of the form
$$
D_{n+1}(x_1,\ldots,x_n,x_1,y_1,\ldots,y_n,x_1)
$$
in terms of $\|D_{n+1}\|_n$. 

If $D_n$ is the density matrix associated to a factorized wave function, i.e. if
$$
D_n(x_1,\ldots,x_n,y_1,\ldots,y_n)=\psi(x_1)\ldots\psi(x_n)\overline{\psi(y_1)}\ldots\overline{\psi(y_n)}\,,
$$
then $D_n\in E_n$ if and only if $\psi\in H^1(\bR)$ and $\|D_n\|_n=\|\psi\|_{H^1}^{2n}$. In this case, 
$$
\ba
D_{n+1}&(x_1,\ldots,x_n,x_1,y_1,\ldots,y_n,x_1)
\\
&=\psi(x_1)|\psi(x_1)|^2\psi(x_2)\ldots\psi(x_n)\overline{\psi(y_1)}\ldots\overline{\psi(y_n)}
\ea
$$
and the estimate in assumption (HL) reduces to the fact that the Sobolev space $H^1(\bR)$ is an algebra. (Indeed, we recall that
$$
u,v\in H^1(\bR)\Rightarrow uv\in H^1(\bR)\,,
$$
with
$$
\|uv\|_{H^1}\le C\|u\|_{H^1}\|v\|_{H^1}\quad\hbox{ for all }u,v\in H^1(\bR)
$$
for some positive constant $C$. This elementary fact is most easily checked in terms of the Fourier transforms of $u$ and $v$.)

Proving the bound in assumption (HL) for general density matrices (and not only in the factorized case) is a rather straightforward generalization of the 
(classical) proof that $H^1(\bR)$ is an algebra. Proving the bound in assumption (Hu) is somewhat more technical and will not be discussed here.

We refer the interested reader to \cite{AdaFGTeta07} --- see also \cite{AdaBarFGTeta} --- for a detailed statement of the result and a complete proof thereof.

The analogous result in space dimension higher than $1$ is considerably more involved --- notice that $H^1(\bR^d)$ is not an algebra whenever $d\ge 2$,
so that the argument above for the bound in assumption (HL) is no longer valid. The proof of the mean field limit for interaction potentials shrinking to a 
Dirac measure can be found in \cite{ElgartErdSchleinYau04,ErdSchleinYau05,ErdSchleinYau06,Pickl2}. An interesting approach to the uniqueness problem 
in the corresponding infinite hierarchy can be found in \cite{KlainMache}.

A general reference for this and the previous section is \cite{PGBBki}. However G\'erard's survey in \cite{PGBBki} predates Pickl's contributions to the subject,
which are described in section \ref{SS-Pickl} below.

\subsection{The time-dependent Hartree-Fock equations}\index{Hartree-Fock limit}

In the quantum mean field theories considered so far, the asymptotic $N$-particle wave function always approached a factorized state of the form
$$
\Psi_N(t,x_1,\ldots,x_N)\sim\prod_{k=1}^N\psi(t,x_k)
$$
in the limit as $N\to\infty$. In fact, we have assumed that the initial state is already of this form, i.e.
$$
\Psi_N(0,x_1,\ldots,x_N)=\prod_{k=1}^N\psi^{in}(x_k)
$$
in all the situations considered so far. In particular, $\Psi_N$ is symmetric in its space variables, i.e.
$$
\Psi_N(t,x_{\si(1)},\ldots,x_{\si(N)})=\Psi_N(t,x_1,\ldots,x_N)
$$
for all $\si\in\fS_N$ and all $x_1,\ldots,x_N\in\bR^d$. In other words, all the mean field theories considered above apply to the case of bosons; the case of 
fermions requires a separate study.

In the case of fermions, one should replace the initial factorized state with 
$$
\Psi_N(0,x_1,\ldots,x_N)=\frac1{\sqrt{N!}}\Det(\psi^{in}_{k}(x_l))_{1\le k,l\le N}
$$
where $(\psi_1,\ldots,\psi_N)$ is an orthonormal system in $L^2(\bR^d)$, i.e.
$$
\int_{\bR^d}\overline{\psi^{in}_k(x)}\psi^{in}_l(x)dx=\de_{kl}\,,\quad k,l=1,\ldots,N\,.
$$
This special form of $N$-particle wave function is referred to as a \textit{Slater determinant}\index{Slater determinant}.

Elementary computations show that Slater determinants satisfy
$$
\int_{(\bR^d)^N}|\Det(\psi^{in}_{k}(x_l))_{1\le k,l\le N}|^2dx_1\ldots dx_N=N!
$$
so that $\Psi_N(0,\ldots)$ is normalized in $L^2((\bR^d)^N)$.

Moreover, one easily checks that the density matrix associated to a Slater determinant has its first marginal given by
$$
D_{N:1}(0):=|\Psi_N(0,\ldots)\ra\la\Psi_N(0,\ldots)|_{:1}=\frac1N\sum_{k=1}^N|\psi^{in}_k\ra\la\psi^{in}_k|\,.
$$
In other words, $ND_{N:1}(0)$ is the orthogonal projector on $\Span(\{\psi_1,\ldots,\psi_N\})$, so that
$$
D_{N:1}(0)^*\!=\!D_{N:1}(0)=ND_{N:1}(0)^2\!\ge 0\,,\quad\hbox{ and }\Tr(D_{N:1}(0))\!=\!1\,.
$$
Higher order marginals of the density matrix associated with a Slater determinant are given in terms of the first marginal by the following expressions:
$$
D_{N:n}(0)=\frac{N^n(N-n)!}{N!}D_{N:1}(0)^{\otimes n}\sum_{\si\in\fS_n}(-1)^{\Sign(\si)}U_\si
$$
for all $n=1,\ldots,N$, where we recall that, for each permutation $\si\in\fS_n$, the operator $U_\si$ is defined by
$$
U_\si\Phi_n(x_1,\ldots,x_n)=\Phi_n(x_{\si^{-1}(1)},\ldots,x_{\si^{-1}(n)})
$$
for all $\Phi_n\in L^2((\bR^d)^n)$.

Consider the $N$-particle Schr\"odinger operator in the mean field scaling, i.e.
$$
H_N:=-\tfrac12\sum_{k=1}^N\Dlt_k+\frac1{N-1}\sum_{1\le k<l\le N}^NV_{kl}
$$
for all $N>1$, where we recall that $V_{kl}$ designates the multiplication by $V(x_k-x_l)$ acting on $L^2((\bR^d)^N)$.

The expected mean field equation in the case of fermions is the time-depen\-dent Hartree-Fock (TDHF) equation, written below in the language of operators
\index{TDHF equation (operator form)}:
$$
i\dot{F}(t)=[-\tfrac12\Dlt,F(t)]+[V_{12},F(t)^{\otimes 2}(I-U_{12})]_{:1}\,.
$$

Assuming that $F(t)$ is an integral operator on $L^2(\bR^d)$ with integral kernel $F(t,x,y)$, the equation above can be recast in terms of integral kernels
as follows:
$$
\ba
{}&i\d_tF(t,x,y)=-\tfrac12(\Dlt_x-\Dlt_y)F(t,x,y)
\\
&+\int_{\bR^d}(V(x-z)-V(y-z))(F(t,x,y)F(t,z,z)-F(t,z,y)F(t,x,z))dz\,,
\ea
$$
or equivalently
$$
\ba
i\d_tF(t,x,y)=&-\tfrac12(\Dlt_x-\Dlt_y)F(t,x,y)
\\
&+F(t,x,y)\int_{\bR^d}(V(x-z)-V(y-z))F(t,z,z)dz
\\
&-\int_{\bR^d}(V(x-z)-V(y-z))F(t,x,z)F(t,z,y)dz\,.
\ea
$$
The existence and uniqueness of the solution of the Cauchy problem for the TDHF equation is extensively studied in \cite{Bove1,Bove2,ChadamGlassey}.

The mean field limit in the case of $N$ fermions with $2$-body bounded interaction is summarized in the following statement.

\begin{Thm}
Assume that $V$ is an even function belonging to $L^\infty(\bR^d)$, and let $(\psi_1,\ldots,\psi_N)$ be an orthonormal system in $L^2(\bR^d)$. Let
$$
\Psi_N(t,\cdot)=e^{-itH_N}\Psi_N^{in}\,,
$$
where 
$$
H_N:=-\tfrac12\sum_{k=1}^N\Dlt_{x_k}+\frac1N\sum_{1\le k<l\le N}V(x_k-x_l)\,,
$$
and
$$
\Psi_N^{in}(x_1,\ldots,x_N)=\frac1{\sqrt{N!}}\Det(\psi_k(x_l))_{1\le k,l\le N}\,.
$$
is the Slater determinant built on the orthonormal system $(\psi_1,\ldots,\psi_N)$.

Then the density operator $D_N(t):=|\Psi_N(t)\ra\la\Psi_N(t)|$ satisfies
$$
D_{N:n}(t)-\frac{N^n(N-n)!}{N!}D_{N:1}(t)^{\otimes n}\sum_{\si\in\fS_n}(-1)^{\Sign(\si)}U_\si\to 0
$$
in $\cL^1(L^2((\bR^d)^n))$ as $N\to\infty$ for each $n\ge 1$, and
$$
D_{N:1}(t)-F_N(t)\to 0
$$
in $\cL^1(L^2(\bR^d))$ as $N\to\infty$, where $F_N(t)$ is the solution of the Cauchy problem
$$
\left\{
\ba
{}&i\dot{F_N}(t)=[-\tfrac12\Dlt,F_N(t)]+[V_{12},F_N(t)^{\otimes 2}(I-U_{12})|_{:1}\,,
\\
&F_N(0)=\frac1N\sum_{k=1}^N|\psi^{in}_k\ra\la\psi^{in}_k|\,.
\ea
\right.
$$
\end{Thm}

See \cite{BGGM,BGGM2} for proofs of this result and of more general theorems in the same direction.

There is another, more familiar formulation of the TDHF equation. Seek the operator $F(t)$ in the form
$$
F(t)=\frac1N\sum_{k=1}^N|\psi_k(t,\cdot)\ra\la\psi_k(t,\cdot)|\,,
$$
or equivalently its integral kernel in the form
$$
F(t,x,y)=\frac1N\sum_{k=1}^N\psi_k(t,x)\overline{\psi_k(t,y)}\,,
$$
where $(\psi_1(t,\cdot),\ldots,\psi_N(t,\cdot))$ is an orthonormal system in $L^2(\bR^d)$. Then the functions $\psi_k(t,\cdot)$ should satisfy the system
of PDEs\index{TDHF equations (orbital form)}
$$
\ba
i\d_t\psi_k(t,x)=&-\tfrac12\Dlt_x\psi(t,x)+\psi_k(t,x)\int_{\bR^d}V(x-z)\frac1N\sum_{l=1}^N|\psi_l(t,z)|^2dz
\\
&-\frac1N\sum_{k=1}^N\psi_l(x)\int_{\bR^d}V(x-z)\overline{\psi_l(t,z)}\psi_k(t,z)dz\,,\quad 1\le k\le N\,.
\ea
$$
In this formulation, the functions $\psi_k$ indexed by $k=1,\ldots,N$ are referred to as ``molecular orbitals'' in the context of quantum chemistry.

These equations are to be compared with the Hartree equation discussed above, i.e.
$$
i\d_t\psi(t,x)=-\tfrac12\Dlt_x\psi(t,x)+\psi(t,x)\int_{\bR^d}V(x-z)|\psi(t,z)|^2dz\,.
$$
Obviously the term 
$$
\psi_k(t,x)\int_{\bR^d}V(x-z)\frac1N\sum_{l=1}^N|\psi_l(t,z)|^2dz
$$
is the analogue in the Hartree-Fock case of the term
$$
\psi(t,x)\int_{\bR^d}V(x-z)|\psi(t,z)|^2dz
$$
in Hartree's equation. The additional term in the Hartree-Fock equation, i.e.
$$
-\frac1N\sum_{k=1}^N\psi_l(x)\int_{\bR^d}V(x-z)\overline{\psi_l(t,z)}\psi_k(t,z)dz
$$
is called the ``exchange interaction integral'' \index{Exchange interaction}and is special to the case of fermions: see for instance \cite{LL3} (Problem 1 on 
p. 233).

Approaching the $N$-particle wave function of a system of fermions by a Slater determinant is obviously the most natural idea. However, other choices are
also possible --- and often made in practice, especially for the purpose of numerical computations in the context of quantum chemistry. One such method is 
the theory of multi-configuration time-dependent Hartree-Fock equations \cite{BardMaus}\index{Multi-configuration TDHF}, where the $N$-particle functions 
is approximated by a linear combination of Slater determinants. For the first detailed analysis of the multi-configuration ansatz, which has a very rich 
mathematical structure, we refer the interested reader to \cite{BardCattMausTrab}.

Hartree-Fock equations are used primarily in the context of quantum chemistry. Therefore, one has to take into account the following interactions

\smallskip
(a) electron-electron,

(b) electron-nuclei, 

(c) nuclei-nuclei.

\smallskip
In many situations, the nuclei are considered either as fixed (as in \cite{LiebSimon}), or as macroscopic objects governed by a system of ODEs (as in 
\cite{CancesLeBris}), while the quantum description of the electron system involves the Hartree-Fock equations. In that case, the only part of the Hamiltonian 
that is approximated by a nonlinear, self-consistent mean-field interaction is the repulsive force (a) between electrons. 

All these interactions involve the Coulomb potential; the derivation of the time dependent Hartree-Fock theory from the $N$-particle Schr\"odinger equation 
with Coulomb interactions remains an open problem at the time of this writing.

\subsection{Pickl's approach to quantum mean field limits}\lb{SS-Pickl}

There are other approaches to the mean field limit in quantum mechanics --- see for instance the work of Rodnianski and Schlein \cite{RodniSchlein}, which proposes an error estimate for the mean field limit, that uses the formalism of Fock spaces in second quantization, as well as the more recent references
\cite{Froh1,Froh2}.

In the present section, we give a brief presentation of yet another approach of the mean field limit that avoids the technicalities of BBGKY hierarchies
as well as the formalism of Fock spaces. The discussion below is based on \cite{Pickl}.

Pickl's idea is to consider some appropriate quantity that measures the distance between the first marginal $D_{N:1}(t)$ of the $N$-particle density matrix
$D_N(t)$ and the single particle density matrix $|\psi(t,\cdot)\ra\la\psi(t,\cdot)|$ built on the solution $\psi$ of the Hartree equation.

Let $\Psi_N$ be the solution of the $N$-particle Schr\"odinger equation
$$
\left\{
\ba
{}&i\d_t\Psi_N(t,x_1,\ldots,x_N)=H_N\Psi_N(t,x_1,\ldots,x_N)\,,\quad x_1,\ldots,x_N\in\bR^d\,,
\\
&\Psi_N\rstr_{t=0}=(\psi^{in})^{\otimes N}\,,
\ea
\right.
$$
with 
$$
H_N:=-\tfrac12\sum_{k=1}^N\Dlt_{x_k}+\frac1N\sum_{1\le k<l\le N}V(x_k-x_l)\,.
$$
Let $\psi$ be the solution of Hartree's equation
$$
\left\{
\ba
{}&i\d_t\psi(t,x)=-\tfrac12\Dlt_x\psi(t,x)+(V\star_x|\psi|^2)\psi(t,x)\,,\quad x\in\bR^d\,,
\\
&\psi\rstr_{t=0}=\psi^{in}\,.
\ea
\right.
$$

Consider the quantity
$$
\cE_N(t):=\Tr(D_{N:1}(t)(I-|\psi(t,\cdot)\ra\la\psi(t,\cdot)|))\,.
$$
Pickl's derivation of the mean field limit is based on proving that
$$
\cE_N(t)\to 0\quad\hbox{ as }N\to\infty\,.
$$
This estimate implies that $D_{N:1}(t)$ converges to $|\psi(t,\cdot)\ra\la\psi(t,\cdot)|$ in operator norm, in view of the lemma below.

Henceforth we denote by $\fH_N$ the $N$-particle Hilbert space $\fH_N:=L^2((\bR^d)^N)$ for each $N\ge 1$, and we set
$$
\rho(t):=|\psi(t,\cdot)\ra\la\psi(t,\cdot)|\,.
$$

\begin{Lem}
Let $N\ge 1$ and $\Psi_N$ be a symmetric element of $\fH_N$ such that $\|\Psi_N\|_{\fH_N}=1$. Let $D_N:=|\Psi_N\ra\la\Psi_N|$ and let $\psi\in\fH_1$ 
satisfy $\|\psi\|_{\fH_1}=1$. Then
$$
\Tr(D_{N:1}(I-|\psi\ra\la\psi|))\to 0\hbox{ as }N\to\infty
$$
if and only if
$$
D_{N:1}\to|\psi\ra\la\psi|\hbox{ in operator norm as }N\to\infty\,.
$$
\end{Lem}

This elementary result is statement (a) in Lemma 2.3 of \cite{Pickl}. In some sense, it can be considered as a variant of the exercise at the end of section
\ref{S-OPTH}. Notice that $\Psi_N(t,\cdot)$ is a symmetric function of $x_1,\ldots,x_N$ for all $t\in\bR$ by Lemma \ref{L-Sym}, and because of the choice 
of the initial data $\Psi_N\rstr_{t=0}=(\psi^{in})^{\otimes N}$, which is itself a symmetric function of $x_1,\ldots,x_N$.

\smallskip
Next we seek to control the evolution of $\cE_N(t)$. We recall that, for $k,l=1,\ldots,N$, the notation $V_{kl}$ designates the operator on $\fH_N$ defined
as
$$
(V_{kl}\Phi_N)(x_1,\ldots,x_N):=V(x_k-x_l)\Phi_N(x_1,\ldots,x_N)\,.
$$
Since $\Tr(D_{N:1}(t))=1$, one has
$$
\ba
\dot{\cE}_N(t)&=-\frac{d}{dt}\Tr(D_{N:1}(t)\rho(t))
\\
&=-\Tr(\dot{D}_{N:1}(t)\rho(t))-\Tr(D_{N:1}(t)\dot{\rho}(t))
\\
&=-\Tr(i[\tfrac12\Dlt,D_{N:1}(t)]\rho(t))-\Tr(iD_{N:1}(t)[\tfrac12\Dlt,\rho(t)])
\\
&+\frac{N-1}{N}\Tr(i[V_{12},D_{N:2}(t)]_{:1}\rho(t))+\Tr(iD_{N:1}(t)[V_{12},\rho(t)^{\otimes 2}]_{:1})\,.
\ea
$$

First
$$
[\tfrac12\Dlt,D_{N:1}(t)]\rho(t)+D_{N:1}(t)[\tfrac12\Dlt,\rho(t)]=[\tfrac12\Dlt,D_{N:1}(t)\rho(t)]\,,
$$
so that
$$
\ba
\Tr([\tfrac12\Dlt,D_{N:1}(t)]\rho(t))&+\Tr(D_{N:1}(t)[\tfrac12\Dlt,\rho(t)])
\\
&=\Tr([\tfrac12\Dlt,D_{N:1}(t)\rho(t)])=0\,.
\ea
$$
Next
$$
\ba
\Tr([V_{12},D_{N:2}(t)]_{:1}\rho(t))&=\Tr([V_{12},D_{N:2}(t)](\rho(t)\otimes I))\,,
\\
\Tr(D_{N:1}(t)[V_{12},\rho(t)^{\otimes 2}]_{:1})&=\Tr(D_{N:2}(t)([V_{12},\rho(t)^{\otimes 2}]_{:1}\otimes I))\,.
\ea
$$

Therefore
$$
\ba
\dot{\cE}_N(t)&=\tfrac{N-1}{N}\Tr(i[V_{12},D_{N:2}(t)]\rho(t)\otimes I)+\Tr(iD_{N:2}(t)[V_{12},\rho(t)^{\otimes 2}]_{:1})
\\
&=-i\Tr(D_{N:2}(t)[\tfrac{N-1}{N}V_{12}-(V\star|\psi(t,\cdot)|^2)\otimes I,\rho(t)\otimes I])\,.
\ea
$$

The result of this computation is summarized in the following lemma.

\begin{Lem}
With the same notation as above, one has
$$
\dot{\cE}_N(t)=\cF_N(t)
$$
where
$$
\cF_N(t)\!:=\!2\IM\Tr\left(D_{N:2}(t)(\rho(t)\!\otimes\! I)A((I\!-\!\rho(t))\otimes\!I)\right)
$$
and 
$$
A:=\tfrac{N-1}{N}V_{12}-(V\star|\psi(t,\cdot)|^2)\otimes I\,.
$$

\end{Lem}

The core of Pickl's argument is the following estimate for $\cF_N(t)$ in terms of $\cE_N(t)$ and $N$.

\begin{Prop}\lb{P-PicklBnd}
Assuming that $r\ge 1$ and denoting by $r'=\frac{r}{r-1}$ the dual H\"older exponent of $r$, one has
$$
|\cF_N(t)|\le 10\|V\|_{L^{2r}}\|\psi(t,\cdot)\|_{L^{2r'}}(\cE_N(t)+\tfrac1N)\,.
$$
\end{Prop}

The proof of Proposition \ref{P-PicklBnd} is sketched below. The interested reader is referred to the original paper \cite{Pickl} for more details.

For the sake of notational simplicity, we set
$$
p(t):=I-\rho(t)\,.
$$
Elementary computations show that 
$$
\ba
\cF_N(t)&=2\IM\Tr(D_{N:2}(t)\rho(t)^{\otimes 2}A(p(t)\otimes\rho(t)))
\\
&+2\IM\Tr(D_{N:2}(t)\rho(t)^{\otimes 2}Ap(t)^{\otimes 2})
\\
&+2\IM\Tr(D_{N:2}(t)(\rho(t)\otimes p(t))Ap(t)^{\otimes 2})\,.
\ea
$$
The term 
$$
\IM\Tr(D_{N:2}(t)(\rho(t)\otimes p(t))A(p(t)\otimes\rho(t)))
$$
vanishes identically by symmetry of $\Psi_N(t,\cdot)$.

The first term on the right hand side is mastered as follows: observe that
$$
(I\otimes\rho(t))A(I\otimes\rho(t))=-\tfrac1N(V\star|\psi(t,\cdot)|^2)\otimes\rho(t)
$$
--- where $V\star|\psi(t,\cdot)|^2$ designates the operator on $\fH_1=L^2(\bR^d)$ defined by
$$
\phi\mapsto(V\star|\psi(t,\cdot)|^2)\phi\,.
$$
Therefore
$$
\ba
|\Tr(D_{N:2}(t)\rho(t)^{\otimes 2}A(p(t)\otimes\rho(t)))|&
\\
=|\Tr(D_{N:2}(t)(\rho(t)\otimes I)(I\otimes\rho(t))A(I\otimes\rho(t))(p(t)\otimes I))|&
\\
\le\|(\rho(t)\otimes I)(I\otimes\rho(t))A(I\otimes\rho(t))(p(t)\otimes I)\|\|D_{N:2}\|_1&
\\
\le\|(\rho(t)\otimes I)(I\otimes\rho(t))A(I\otimes\rho(t))\|&\,.
\ea
$$
Then, observing that $A$ is self-adjoint since $V$ is a real-valued, even function, one has
$$
\ba
\|(\rho(t)\otimes I)(I\otimes\rho(t))A(I\otimes\rho(t))\|^2&
\\
=
\|(\rho(t)\otimes I)((I\otimes\rho(t))A(I\otimes\rho(t)))^2(\rho(t)\otimes I)\|&
\\
=
\frac1{N^2}\|\la\psi(t,\cdot)|(V\star|\psi(t,\cdot)|^2)^2\psi(t,\cdot)\ra\rho(t)\otimes\rho(t)\|&
\\
\le
\frac1{N^2}\int_{\bR^d}(V\star|\psi(t,\cdot)|^2)^2(x)|\psi(t,x)|^2dx&
\\
\le\frac1{N^2}\|(V\star|\psi(t,\cdot)|^2)^2\|_{L^r}\|\psi(t,\cdot)^2\|_{L^{r'}}&
\\
=\frac1{N^2}\|V\star|\psi(t,\cdot)|^2\|^2_{L^{2r}}\|\psi(t,\cdot)\|^2_{L^{2r'}}&
\\
\le\frac1{N^2}\|V\|^2_{L^{2r}}\|\psi(t,\cdot)\|_{L^{2r'}}&\,.
\ea
$$
Indeed,
$$
\|V\star|\psi(t,\cdot)|^2\|_{L^{2r}}\le\|V\|_{L^{2r}}\||\psi(t,\cdot)|^2\|_{L^1}=\|V\|_{L^{2r}}
$$
since the solution of Hartree's equation satisfies $\|\psi(t,\cdot)\|_{L^2}=1$ for all $t\in\bR$.

The third term satisfies
$$
\ba
|\Tr(D_{N:2}(t)(\rho(t)\otimes p(t))Ap(t)^{\otimes 2})|&
\\
\le
\|(\rho(t)\otimes I)A(p(t)\otimes I)\| \|(I\otimes p(t))D_{N:2}(t)(I\otimes p(t))\|_1&
\\
=\|(\rho(t)\otimes I)A(p(t)\otimes I)\|\Tr((I\otimes p(t))D_{N:2}(t)(I\otimes p(t)))&
\\
\le\|(\rho(t)\otimes I)A\|\Tr(D_{N:1}p(t))&
\\
=\|(\rho(t)\otimes I)A\|\cE_N(t)&\,.
\ea
$$
Now
$$
\ba
\|(\rho(t)\otimes I)A\|^2=\|(\rho(t)\otimes I)A^2(\rho(t)\otimes I)\|&
\\
\le\left(2\|V^2\|_{L^{r}}+2\|(V\star|\psi(t,\cdot)|^2)^2\|_{L^r}\right)\||\psi(t,\cdot)|^2\|_{L^{r'}}&
\\
=\left(2\|V\|^2_{L^{2r}}+2\|V\star|\psi(t,\cdot)|^2\|^2_{L^{2r}}\right)\|\psi(t,\cdot)\|^2_{L^{2r'}}&
\\
\le 4\|V\|^2_{L^{2r}}\|\psi(t,\cdot)\|^2_{L^{2r'}}&\,,
\ea
$$
again because the solution $\psi$ of Hartree's equation satisfies $\|\psi(t,\cdot)\|_{L^2}=1$ for all $t\in\bR$.

Controlling the second term in $\cF_N(t)$ is slightly more complicated. 

We shall need the following elements of notation. First we denote by $p(t)$ the projection $p(t)=I_1-\rho(t)=I_1-|\psi(t,\cdot)\ra\la\psi(t,\cdot)|$. 

For each $j=1,\ldots,N$, we set $I_j:=I_{\fH_j}$, and for each $T\in\cL(\fH_1)$, we define $T_{j,N}\in\cL(\fH_N)$ for $j=1,\ldots,N$ as
$$
T_{j,N}:=I_{j-1}\otimes T\otimes I_{N-j}\,.
$$ 
For each $a\in\{0,1\}^{\{1,\ldots,N\}}$ and each $\pi\in\cL(\fH_1)$ such that $\pi^*=\pi=\pi^2$, we define
$$
P_N[a,\pi]:=\prod_{j=1}^N\pi_{j,N}^{1-a(j)}(I-\pi_{j,N})^{a(j)}
$$

\begin{Lem}\lb{L-SpecMN}
Let
$$
M_N(t)=\frac1N\sum_{j=1}^Np_{j,N}(t)
$$
Then
$$
M_N(t)=\sum_{k=1}^N\frac{k}{N}\Pi_{k,N}[\rho(t)]
$$
where
$$
\Pi_{k,N}[\rho(t)]:=\sum_{a\in\{0,1\}^{\{1,\ldots,N\}}\atop a(1)+\ldots+a(N)=k}P_N[a,\rho(t)]\,.
$$
Besides
$$
\Pi_{j,N}[\rho(t)]^*=\Pi_{j,N}[\rho(t)]\,,\quad\hbox{ and }\Pi_{j,N}[\rho(t)]\Pi_{k,N}[\rho(t)]=\de_{jk}\Pi_{j,N}[\rho(t)]
$$
for each $j,k=1,\ldots,N$, and
$$
\sum_{k=0}^N\Pi_{k,N}[\rho(t)]=I_{\fH_N}\,.
$$
\end{Lem}

\smallskip
In other words, the relation
$$
M_N(t)=\sum_{k=0}^N\frac{k}{N}\Pi_{k,N}[\rho(t)]
$$
is the spectral decomposition of the self-adjoint element $M_N(t)$ of $\cL(\fH_N)$. These elementary computations are summarized in formula (6)
of \cite{Pickl}. 

Define, for each $\a>0$,
$$
M_N^\a(t)=\sum_{k=1}^N\left(\frac{k}{N}\right)^\a\Pi_{k,N}[\rho(t)]
$$
and
$$
M_N^{-\a}(t):=\sum_{k=1}^N\left(\frac{k}{N}\right)^{-\a}\Pi_{k,N}[\rho(t)]\,,
$$
that is the pseudo-inverse of $M_N^\a(t)$ extended by $0$ on $\Ker(M_N(t))$. In other words
$$
M_N^\a(t)M_N^{-\a}(t)=I_N-\Pi_{0,N}[\rho(t)]
$$
In particular
$$
M_N^\a(t)M_N^{-\a}(t)p_{j,N}(t)=p_{j,N}(t)
$$
for each $j=1,\ldots,N$.

After this observation, the term that remains to be estimated, i.e.
$$
\Tr(D_{N:2}(t)\rho(t)^{\otimes 2}A(I-\rho(t))^{\otimes 2})
$$
is recast as
$$
\Tr(D_N(t)\rho_{1,N}(t)\rho_{2,N}(t)AM_N^{1/2}(t)M_N^{-1/2}(t)p_{1,N}(t)p_{2,N}(t))\,.
$$
Then
$$
\ba
|\Tr(D_N(t)\rho_{1,N}(t)\rho_{2,N}(t)AM_N^{1/2}(t)M_N^{-1/2}(t)p_{1,N}(t)p_{2,N}(t))|&
\\
\le
\|M_N^{-1/2}(t)p_{1,N}(t)p_{2,N}(t)\Psi_N(t,\cdot))\|_{\fH_N}&
\\
\times\|M_N^{1/2}(t)A\rho_{1,N}(t)\rho_{2,N}(t)\Psi_N(t,\cdot))\|_{\fH_N}&\,.
\ea
$$

The second factor on the right hand side is mastered by observing that
$$
\ba
\|M_N^{1/2}(t)A\rho_{1,N}(t)\rho_{2,N}(t)\Psi_N(t,\cdot))\|^2_{\fH_N}
\\
=\la\Psi_N(t,\cdot)|\rho_{1,N}(t)\rho_{2,N}(t)AM_N(t)A\rho_{1,N}(t)\rho_{2,N}(t)\Psi_N(t,\cdot))\ra
\\
=\tfrac{N-2}N\la\Psi_N(t,\cdot)|\rho_{1,N}(t)\rho_{2,N}(t)Ap_{3,N}(t)A\rho_{1,N}(t)\rho_{2,N}(t)\Psi_N(t,\cdot))\ra
\\
+
\tfrac2N\la\Psi_N(t,\cdot)|\rho_{1,N}(t)\rho_{2,N}(t)Ap_{1,N}(t)A\rho_{1,N}(t)\rho_{2,N}(t)\Psi_N(t,\cdot))\ra
\ea
$$
since $\Psi_N(t,\cdot)$ is symmetric. The first summand is bounded by
$$
\ba
\tfrac{N-2}N\|\rho_{1,N}(t)A^2\rho_{1,N}(t)\|_{\cL(\fH_N)}\|p_{3,N}(t)\Psi_N(t)\|^2_{\fH_N}&
\\
\le
\|\rho_{1,N}(t)A^2\rho_{1,N}(t)\|_{\cL(\fH_N)}\cE_N(t)&
\\
\le
4\|V\|^2_{L^{2r}}\|\psi(t,\cdot)\|^2_{L^{2r'}}\cE_N(t)&\,,
\ea
$$
while the second is bounded by
$$
\tfrac2N\|\rho_{1,N}(t)A^2\rho_{1,N}(t)\|_{\cL(\fH_N)}\le\tfrac8N\|V\|^2_{L^{2r}}\|\psi(t,\cdot)\|^2_{L^{2r'}}\,.
$$

Finally, we use again the fact that $\Psi_N(t,\cdot)$ is symmetric to conclude that
$$
\ba
N(N-1)\|M_N^{-1/2}(t)p_{1,N}(t)p_{2,N}(t)\Psi_N(t,\cdot))\|^2_{\fH_N}&
\\
=
N(N-1)\la\Psi_N(t,\cdot),M_N^{-1}(t)p_{1,N}(t)p_{2,N}(t)\Psi_N(t,\cdot)\ra&
\\
=2\sum_{1\le j<k\le N}\la\Psi_N(t,\cdot),M_N^{-1}(t)p_{j,N}(t)p_{k,N}(t)\Psi_N(t,\cdot)\ra&
\\
\le N^2\la\Psi_N(t,\cdot),M_N^{-1}(t)M_N^2(t)\Psi_N(t,\cdot)\ra&
\\
=N^2\la\Psi_N(t,\cdot),M_N(t)\Psi_N(t,\cdot)\ra&
\\
=N^2\cE_N(t)&\,.
\ea
$$
Putting all these estimates together establishes the bound in Proposition \ref{P-PicklBnd}

Applying Gronwall's estimate, we deduce from Proposition \ref{P-PicklBnd} the following result, stated in space dimension 3 for the sake of simplicity.
The assumptions on the potential $V$ have been adapted to match those in Kato's Theorem \ref{T-Kato}.

\begin{Thm}\lb{T-ThmPickl}
Let $V$ be a real-valued even function such that
$$
V\in L^{2r}(\bR^3)\quad\hbox{ and }V\rstr_{\bR^3\setminus B(0,R)}\in L^\infty
$$
for some $r\ge 1$ and some $R>0$. 

Let $\psi^{in}\in L^2(\bR^3)$ be such that $\|\psi^{in}\|_{L^2}=1$, and assume that the Cauchy problem for Hartree's equation with initial data $\psi^{in}$
has a unique solution $\psi\in C(\bR;L^{2r'}(\bR^3))$, where $r'=\frac{r}{r-1}$ is the dual H\"older exponent of $r$. Let $\Psi_N$ be the solution of the 
Cauchy problem for the $N$-particle Schr\"odinger equation with initial data $(\psi^{in})^{\otimes N}$, and $D_N(t):=|\Psi_N(t,\cdot)\ra\la\Psi_N(t,\cdot)|$.

Then, for each $t\in\bR$, one has
$$
D_{N:1}(t)\to|\psi(t,\cdot)\ra\la\psi(t,\cdot)|\hbox{ in operator norm as }N\to+\infty\,,
$$
with the estimate
$$
\ba
\Tr&(D_{N:1}(t)(I-|\psi(t,\cdot)\ra\la\psi(t,\cdot)|))
\\
&\le\frac1N\left(\exp\left|\int_0^t10\|V\|_{L^{2r}}\|\psi(s,\cdot)\|_{L^{2r'}}ds\right|-1\right)\,.
\ea
$$
\end{Thm}

\smallskip
Some remarks are in order after this rather remarkable result.

\smallskip
The first obvious advantage of Pickl's argument is its simplicity. The proof of Proposition \ref{P-PicklBnd} and the resulting Theorem \ref{T-ThmPickl} provide 
a much simpler derivation of the mean field limit for potentials with a Coulomb singularity at the origin than the Erd\"os-Yau proof \cite{ErdosYau02} based 
on the BBGKY hierarchy. According to Theorem 3.1 in \cite{GinVeloHartree}, if $\psi^{in}\in H^{k}(\bR^3)$, then the unique mild solution $\psi$ of Hartree's 
equation belongs to $C(\bR_+;H^k(\bR^3))$ under the same assumption on $V$ as in Theorem \ref{T-ThmPickl}. Therefore $\psi\in C(\bR;L^{2r'}(\bR^3))$ 
by Sobolev's embedding theorem provided that $k$ is chosen large enough. In fact, even stronger singularities at the origin than Coulomb in the interaction potential can be handled with this method. 

A second advantage of Pickl's method is its versatility. For instance, it can also be applied to the derivation of the Gross-Pitaevski: see \cite{Pickl2}. In this
latter reference, Pickl uses appropriate nonlinear functions of the operator $M_N(t)$, which are defined in terms of the spectral decomposition  provided by
Lemma \ref{L-SpecMN}. The same type of method also applies to study the convergence rate in the mean field limit \cite{KnowlesPickl} in a more general
and systematic way than in Theorem \ref{T-ThmPickl}Ê above.

On the other hand, unlike methods based on the BBGKY hierarchy, Pickl's argument \cite{Pickl} seems to be limited to pure states and to bosonic particles 
so far. Therefore, in spite of its appealing simplicity, Pickl's method does not really supersedes all previous derivations of the mean field limits in quantum
mechanics.

Besides, even if BBGKY hierarchies are a rather cumbersome mathematical object, they are ubiquitous in the field of nonequilibrium statistical physics,
and appear in the treatment of quantum as well as classical mechanical models. This accounts for the importance of these hierarchies in these lectures.

\section[Afterword]{Afterword}

These notes are an updated and (much) expanded version of my earlier (2003) lectures \cite{FGForges}. I express my gratitude to my colleagues A. Muntean,
J. Radema\-cher and A. Zagaris for their kind invitation to the NDNS+ 2012 Summer School at Universiteit Twente in Enschede. The first part of these notes
(sections 1-6) were also presented in a joint course with L. Desvillettes at the 37th Summer School on Mathematical Physics in Ravello (2012). I am also very 
grateful to Profs. T. Ruggeri and S. Rionero, and to the participants for their positive feedback, as well as for the most enjoyable hospitality of the Ravello school.

As in \cite{FGForges}, the core of the material presented here is the mean field limit of $N$-particle systems governed by the equations of classical mechanics 
in the case where the $N$-particle empirical measure is an exact weak solution of the mean field PDE, and the mean field limit of the $N$-particle Schr\"odinger 
equation.

Unlike in \cite{FGForges}, I have chosen to introduce both the approach based on empirical measures and the method based on the BBGKY hierarchy on the 
mean field limit of the $N$-particle systems in classical mechanics with Lipschitz continuous interaction kernel $K$. The discussion of the propagation of chaos 
in section \ref{S-CHAO} is much more detailed here than in \cite{FGForges}. The connection between the empirical measure and the hierarchy approaches is 
hopefully clarified by the discussion in sections \ref{SS-DobruBBGKY}, \ref{SS-QuantPropaChao} and \ref{SS-Lions}. 

The most recent developments --- in particular results obtained since the publication of \cite{FGForges} --- are surveyed in sections \ref{S-MFClassFurther}
and \ref{S-MFQuantFurther}. However, I have chosen to give less many details on the derivation of the time-dependent Hartree-Fock equations from the
$N$-particle Schr\"odinger equation. The discussion of the time-dependent Hartree-Fock equations in the present notes is limited to a presentation of the
equations and to the statement of the main result obtained on this problem. My earlier notes \cite{FGForges} contain a sketch of the proof of the mean field
limit leading to the time-dependent Hartree-Fock equations; the interested reader is referred to \cite{FGForges} for more details in this direction.
 
I have limited myself to a discussion of evolution problems. Yet there are many important results bearing on steady problems: see for instance the work
of Caglioti-Lions-Marchioro-Pulvirenti \cite{CaLiMaPu92} and Kie\ss ling \cite{Kiess} on the Onsager statistical theory of vortices in incompressible fluids.
I have also neglected to discuss the similarities and the differences between the mean field limit of large particles systems and the convergence of particle
methods for mean field equations. The interested reader will find information on these issues in chapter 5 of \cite{MarchioPulvi} in the context of the vortex
formulation of the Euler equations for  two dimensional incompressible fluids. See also, in the context of Vlasov type equations, \cite{Batt,CottRav,Woll} 
and the references therein.

I have also deliberately chosen to avoid discussing results that would involve both the mean field limit (as the particle number $N$ tends to infinity) and
the classical limit of quantum mechanics (assuming that the typical action of an individual particle is large compared to Planck's constant $h$). One of the
reasons for this choice is that this part of the theory is technically more involved and perhaps not as complete as what has been presented above. Yet
there are many important contributions of in this directions, and the interested reader should have a look at \cite{Hepp,NarnhoSewell}, to quote only a few 
references exploring that topic. 

Another conspicuous omission in the second part of these notes is the formalism of Fock spaces and quantum field theory. This mathematical setting allows 
considering data where the particle number can be infinite --- more precisely, the particle number $N$ is an unbounded operator in this formalism. The mean 
field limit can be formulated in this setting: see for instance \cite{BGGM2,Froh1,Froh2}, together with the papers \cite{Hepp,GiniVelo} already mentioned above, 
involving the classical limit in addition to the case of infinitely many particles. Interestingly, there is a striking analogy between the classical limit of quantum 
mechanics in formulated in terms of Egorov's theorem (Theorem 25.3.5 in \cite{Horm4}) and the point of view introduced in \cite{Froh1,Froh2} for the quantum
mean field limit: see \cite{Froh1}Êon p. 1024.

I am indebted in various ways to several colleagues for the material presented in these notes. First I wish to thank my collaborators R. Adami, C. Bardos, 
B. Ducomet, L. Erd\"os, A. Gottlieb, N. Mauser, V. Ricci, A. Teta and H.-T. Yau, with whom I have worked on mean field limits on various occasions. Besides, 
the material presented here is in (many) places far from original, and the present notes owe much to earlier presentations of mean field limits --- such as 
\cite{Sznitman} or unpublished notes by C. Villani, as well as lectures by P.-L. Lions at the Coll\`ege de France. I also learned a lot from a series of lectures
by C. Mouhot presenting his joint work with S. Mischler and B. Wennberg \cite{MischlerMouhot, MischlerMouhotWenn} on quantitative estimates for the
propagation of chaos. Most of what I know on the derivation of Vlasov type equations from classical particle dynamics comes from discussions with J. Batt, 
H. Neunzert, J. Wick on various occasions in the late 80s. I learned of Dobrushin's beautiful estimate \cite{Dobrushin} from a lecture given by M. Pulvirenti 
at Ecole Normale Sup\'erieure in 1997. 

Another key idea in these lectures is the use of the Nirenberg-Ovcyannikov abstract version of the Cauchy-Kovalevska theorem to handle BBGKY type hierarchies. I got acquainted with this idea in a seminar given by S. Ukai at the Sone seminar in 1998 in Kyoto. Ukai used this technique in a slightly
different context (the derivation of the Boltzmann equation from the dynamics of a large number of hard spheres in the Boltzmann-Grad limit), and for a 
slightly different purpose (he wanted to simplify the uniform stability estimates in Lanford's argument). With his usual modesty, Ukai presented his result 
as a failed attempt at simplifying Lanford's notoriously intricate proof, and did not immediately published it --- I came across his paper \cite{Ukai01} only
very recently. In the case of the infinite mean field hierarchies for quantum $N$-particle dynamics, I find the argument based on the abstract variant the 
Cauchy-Kovalevska theorem much more elegant and illuminating than the conventional argument based on the Duhamel series as in Lanford's proof, 
and thus I have chosen to present the former in these notes, as in \cite{FGForges}. The argument based on the Duhamel series can be found in many 
places --- see for instance the proof of Theorem 5.7 on p. 610 in \cite{Spohn80}.

S. Ukai was at the origin of several fundamental results on the mathematical analysis of kinetic models --- such as the existence and uniqueness of global
solutions of the Boltzmann equation for all initial data near Maxwellian equilibrium, the existence and uniqueness of global solutions of the Vlasov-Poisson 
system \cite{UkaiOka} in space dimension $2$, the hydrodynamic limit of the Boltzmann equation leading to the incompressible Navier-Stokes equations
in the regime of  ``small'' initial data. In view of his influence on some of the topics discussed above, these notes are dedicated to his memory.


\printindex

\end{document}